\title{\vspace{-1cm}On diffusive 2D Fokker-Planck-Navier-Stokes systems}
\date{\today}
\documentclass[12pt]{article}
\usepackage{authblk}
\usepackage{blindtext}
\usepackage{amsfonts}
\usepackage{amsmath}
\usepackage{amsthm}
\usepackage{graphicx}
\usepackage[margin=2cm]{geometry}
\newtheorem{theorem}{Theorem}
\newtheorem{lemma}{Lemma}

\newtheorem{remark}{Remark}
\newtheorem{corollary}{Corollary}
\newtheorem{definition}{Definition}
\newcommand{\norm}[1]{\left\lVert#1\right\rVert}
\author[1]{	Joonhyun La \thanks{joonhyun@math.princeton.edu}}
\affil[1] {Department of mathematics, Princeton University}
\begin{document}
\maketitle

\begin{abstract}
We study models kinetic models of polymeric fluids. We introduce a notion of solutions which is based on moments of polymeric distributions. We prove global existence and uniqueness of a large class of initial data for diffusive systems of kinetic equations coupled to fluid equations. As a corollary, we obtain a rigorous derivation of Oldroyd-B closure. We also prove decay of free energy for all the systems considered.
\end{abstract}

\section{Introduction}

Polymeric fluids are important in many branches of science and engineering. In addition, their behavior is very nontrivial; for example, some polymeric fluids develop turbulent flows at low Reynolds numbers, in stark contrast to Newtonian fluids (\cite{Groisman2000}). Thus, to understand the behavior of a wide range of everyday materials, modeling and analysis of polymeric fluids are important. Also, polymeric fluids recently have drawn attention of mathematicians, and they have investigated various models of polymeric flows. In this paper, we focus on simplified models of polymeric flows, which originates from the kinetic theory of dilute polymer solutions. The model (\cite{Bird1987}, \cite{bird1987dynamics}, \cite{Doi1986}, \cite{MR1383323}) regards a complex fluid as a dilute suspension of polymers in a solvent, which is an incompressible Newtonian fluid. The polymer is modeled as an elastic dumbbell, that is, two massless beads joined by a spring with the potential $U(m)$. The configuration of the polymer is represented by its end-to-end vector $m \in M = \mathbb{R}^2$. The fraction of polymers with configuration $m$ is denoted by $f(m) dm$. The complex fluid occupies the physical space $\Omega = \mathbb{R}^2$. In the following, we will provide a more detailed exposition of the model that we consider. The explanation consists of several steps, starting from the description of simple, idealized situation to more complicated, realistic situation. First, we discuss the equilibrium state. Next, we investigate out-of-equilibrium dynamics of homogeneous suspensions. Then we introduce spatial inhomogeneity and fluid effects to the dynamics of the polymer distribution. Finally, the effects of polymers on the flow are explained.
\paragraph{Description of equilibrium distribution.} The equilibrium distribution is realized as a minimizer of a modified free energy
\begin{equation}
\mathcal{E} [f] = \int_M f \log f + U(m) f dm
\end{equation}
where $f \log f$ represents an entropic effect and $U(m) f$ represents the potential energy due to restoring force of the spring. The resulting distribution $f = \frac{e^{-U}}{Z}$ is the equilibrium distribution, where $Z$ is a normalizing factor. In some models (\cite{Constantin2010}, \cite{MR2655905}) the interaction between polymers are also considered and then $U$ may depend on $f$ as well; then the equation
\begin{equation}
f = \frac{1}{Z} e^{-U[f] },
\end{equation}
which is called the Onsager equation, shows various interesting properties, for example, phase transition ((\cite{Constantin2010}, \cite{MR2655905}). However, our models describe dilute solutions and the interaction between polymers are not considered, and $U$ depends only on $m$. 
\paragraph{Out-of-equilibrium dynamics of homogeneous polymer distributions.} Assuming that the polymer solution is homogeneous in physical space, and that there is no fluid flows that disturb polymer distributions, the polymer solution tends to converge to equilibrium distribution and the process is governed by the kinetic equation
\begin{equation}
\partial_t f = \epsilon \nabla_m \cdot \left (f \nabla_m \left ( \frac{\delta \mathcal{E} }{\delta f } \right ) \right ) \label{relaxation}
\end{equation}
where $\epsilon$ is a positive constant quantifying inter-particle diffusivity. We can rewrite this equation as
\begin{equation}
\partial_t f = \epsilon \Delta_m f + \epsilon \nabla_m \cdot \left ( f \nabla_m U \right ).
\end{equation}
The system (\ref{relaxation}) has $\mathcal{E}$ as a Lyapunov functional:
\begin{equation}
\frac{d}{dt} \mathcal{E} = - \mathcal{D}
\end{equation}
where
\begin{equation}
\mathcal{D} = \epsilon \int_{M} f | \nabla_m \left ( \log f + U \right ) | ^2 dm.
\end{equation}
It is also known that the equation (\ref{relaxation}) can be understood as a limit of steepest descent of sum of Wasserstein distance and the free energy functional (\cite{MR1617171}). 
\paragraph{Consideration of spatial inhomogeneity and fluid effects.} Polymer distributions are in fact spatially inhomogeneous, and the flows of the solvent influence the distributions of polymers. Thus, $f = f(x, m, t)$ depends also on $x$, and kinetic evolution of $f$ depends on the fluid velocity $u(x, t)$. The effect of fluid is twofold: first, it transports the polymer particles, and second it stretches and rotates polymer particles due to inhomogeneity of fluid field. The Fokker-Planck equation then reads
\begin{equation}
\partial_t f + u \cdot \nabla_x f + (\nabla_x u)m \cdot \nabla_m f = \epsilon \left ( \Delta_m f + \nabla_m \cdot \left ( f \nabla_m U \right ) \right ) + \nu_2 \Delta_x  f \label{inhomoFP}
\end{equation}
where $\nu_2 > 0$ is the coefficient for center-of-mass diffusion. If $u$ is divergence-free, then $( \nabla_x u ) m$ is also divergence-free in $m$ variable. There are two types of variants of (\ref{inhomoFP}) that are widely used in the literature. The first type of variants (non-diffusive models) sets $\nu_2 = 0$ in the center-of-mass diffusion term $\nu_2 \Delta_x f$, and the second (corotational models) replaces the fluid effect term $ (\nabla_x u )m \cdot \nabla_m f$ by $ \Omega(u) m \cdot \nabla_m f$, where $\Omega (u) =  \frac{1}{2} \left ( (\nabla_x ) u - ((\nabla_x ) u )^T \right )$ is the vorticity tensor. 
\begin{remark}
Non-diffusive models are considered because the center-of-mass diffusion coefficients are known to be significantly smaller than other effects (\cite{Bhave1991}). However, the center-of-mass diffusion effects are physically justified (\cite{El-Kareh1989}) and play a central role in the stabilization of the flow in the long run. Diffusive models ($\nu_2 > 0$) are discussed by many authors (\cite{MR2552164}, \cite{MR2338493}, \cite{MR2819196}, \cite{MR2981267}, \cite{MR2902849}, \cite{Schieber2006}, \cite{MR3698145}), and center-of-mass diffusion effects are added to stabilize the numerical algorithm (\cite{Thomases2007}, \cite{Thomases2011}) in numerical simulations of polymeric flows. 
\end{remark}
\begin{remark}
In corotational models (\cite{MR2340887}, \cite{MR1763488}, \cite{MR2456183}), fluid flows do not stretch polymers, but they only rotates polymers. In most models $U(m)$ is a radial function in $m$, and fluid flows do not influence the total elastic energy of the polymers in the corotational setting. Corotational models enjoy better a priori estimates due to this decoupling, which make well-posedness problems easier.
\end{remark}
\paragraph{Effects of polymers to flows.} So far, we discussed how the microscopic system behaves, and how macroscopic (fluid) effects influence the microscopic system. At this point, we discuss how the macroscopic system is influenced by the microscopic (polymer) system. The polymers influence flows (the "micro-macro" interaction) by an added stress tensor $\sigma$. The stress tensor $\sigma$ is given by Kramer's expression(\cite{MR1383323}):
\begin{equation}
\sigma (x, t) = \int_M m \otimes \left ( \nabla_m U (m) \right ) f dm. \label{Kramer}
\end{equation}
The fluid velocity field $u(x, t)$ solves the incompressible Navier-Stokes equation
\begin{equation}
\begin{gathered}
\partial_t u + u \cdot \nabla_x u = - \nabla_x p + \nu_1 \Delta_x u + K \nabla_x \cdot \sigma, \\
\nabla_x \cdot u = 0 
\end{gathered} \label{NS}
\end{equation}
with $K$ a positive constant and $\nu_1>0$ the kinematic viscosity. The coupled system consisting of (\ref{inhomoFP}), (\ref{Kramer}), and (\ref{NS}) satisfies an energetic principle: the sum of the kinetic energy and free energy dissipates, that is,
\begin{equation}
\frac{d}{dt} \left ( \int_{\mathbb{R}^2} \frac{1}{2} |u(x, t) |^2 + K \mathcal{E}[f] (x, t)  dx \right ) + \int_{\mathbb{R}^2} \nu_1 |\nabla_x u (x, t) |^2 + K \mathcal{D} ' (x, t) dx \le 0, \label{wfreeEdiss}
\end{equation}
where 
\begin{equation}
\mathcal{D} ' (x, t) = \int_M f \left ( \nu_2 |\nabla_m (\log f + U) |^2 + \epsilon |\nabla_x ( \log f + U ) |^2 \right ) dm.
\end{equation}
In fact, it can be shown that this energetic principle can be used to determine polymeric stress from the micro-micro (interaction between polymers) and macro-micro interactions (drift and deformation of polymers due to external fluid field) (\cite{MR2188682}, \cite{MR3050292}, \cite{MR2503655}). We note that due to the effect of spring potential of added polymers, the fluid also exhibits elastic as well as viscous behavior. This type of complicated behavior of a material is called "viscoelasticity" in the literature (\cite{MR1051193}). In addition, there is an a priori estimate which is similar to (\ref{wfreeEdiss}) but stronger: we replace the term $\mathcal{E}[f]$ by the relative entropy of $f$ with respect to the equilibrium distribution $\int f dm \frac{e^{-U}}{Z}$. Then we get the estimate (\ref{entropyestms}). This estimate is known as entropy estimate or free energy estimate in the literature. One of our goal in this paper is to prove this estimate rigorously.
\paragraph{Choice of potential function $U$.} Up until this point, we provided an overview of the system without specifying the potnential $U$. In fact, the mathematical nature of the system may vary depending on the choice of potential $U$. In this paragraph we briefly review the choice of potential $U$. The two most frequent choices for the potentials are Hookean spring, where $U (m) = |m|^2$, and FENE (finite extensible nonlinear elastic) dumbbell model, where $U(m) = -k \log \left ( 1 - \frac{|m|^2}{|m_0 |^2 } \right )$ (\cite{MR3010381}, \cite{MR2503655}). The Hookean spring model has its formal macroscopic closure, which is called the Oldroyd-B model (\cite{MR0035192}); by multiplying $m \otimes m$ to (\ref{inhomoFP}) and integrating in $m$ variable, and using integration by parts we get the formal macroscopic closure for Fokker-Planck equation
\begin{equation}
\partial_t \sigma + u \cdot \nabla_x \sigma = (\nabla_x u ) \sigma + \sigma (\nabla_x u) ^T - 2 \epsilon \sigma + 2 \epsilon \mathbb{I} + \nu_2 \Delta_x \sigma.
\end{equation}
When $\nu_2 > 0$, the global well-posedness is known (\cite{MR2989441}) while the case $\nu_2 = 0$ is open. For some class of initial data, one can justify this formal closure from Fokker-Planck equation (\cite{MR3698145}). In this paper, we extend this justification result to a broader class of initial data. The Oldroyd-B model is widely used due to its simplicity: the system is fully macroscopic, and there is no need to solve Fokker-Planck equation and integrate $f$ over $m$ to compute stress field $\sigma$. Infinite extensibility of polymer both poses difficulties in mathematical investigation and fitting real world data (\cite{MR2503655}). On the other hand, the potentials in FENE models blow up at finite $m$, so finite extensibility of polymers is guaranteed. Choosing these potentials yields mathematical difficulties near the boundaries (\cite{MR3010381}). Also the system is genuinely a multiscale problem; in fact, an exact macroscopic closure is only obtained for the Hookean spring potential. In this article, we consider potentials that lie between these two potentials: we consider potentials $U(m) = |m|^{2q}$, where $q \ge 1$ is a real number. Similar types of potentials have been considered (\cite{MR2902849}, \cite{MR2073140}), while the potentials in them behave as Hookean spring near $m=0$. Our potentials share some of the difficulties of both Hookean and FENE systems: the polymers are infinitely extensible and the problem is multiscale. 
\begin{remark}
There are other models for polymeric fluids (\cite{MR2503655}), not necessarily originated from kinetic models, which have been studied extensively; for example, there are Gisekus models (\cite{Giesekus1982}), Phan-Thien Tanner models(\cite{Thien1977}) which are derived from lattice model, and FENE-P models( \cite{POL:POL110040411}, \cite{MR1918565}, \cite{MR2203938}), which are derived from approximate closure of FENE model. 
\end{remark}

\subsection{Previous works}
There is a vast literature on complex fluids, and it is impossible to give a complete account. 
\paragraph{Oldroyd-B and relevant macroscopic models.} Macroscopic models for viscoelasticity, such as Oldroyd-B, have been studied extensively. First we discuss the results concerning non-diffusive models. Guillop\'{e} and Saut proved local existence, uniqueness of strong solution, and global existence of strong solution for small initial data, in the case of bounded domain, in \cite{MR1077577} and in \cite{MR1055305}. Fern\'{a}ndez-Cara, Guill\'{e}n, and Ortega extended the results of Guillop\'{e} and Saut to $L^p$ setting in \cite{MR1422802}, \cite{MR1633055}, and \cite{MR1893419}. In addition, Hieber, Naito, and Shibata studied the system in the case of exterior domain in \cite{MR2860633}. Chemin and Masmoudi studied the system in critical Besov spaces, and proved local well-posedness of the system and provided a Beale-Kato-Majda type (\cite{MR763762}) criterion in \cite{MR1857990}. Other Beale-Kato-Majda type sufficient conditions were given by Kupferman, Mangoubi, and Titi in \cite{MR2398006}, and by Lei, Masmoudi, and Zhou in \cite{MR2558169}. In addition, Lions and Masmoudi showed global existence of weak solution for corotational models in \cite{MR1763488}. Hu and Lin proved in \cite{MR3434615} global existence of weak solution for non-corotational models, given that the initial deformation gradient is close to the identity and the initial velocity is small. In \cite{MR2165379}, Lin, Liu, and Zhang developed an approach based on deformation tensor and Lagrangian particle dynamics. Lei and Zhou studied the system via incompressible limit in \cite{MR2191777} and proved global existence for small data. Also, Lei, Liu, and Zhou studied global existence for small data and incompressible limit in \cite{MR2393434}. Moreover, in \cite{MR3473592}, Fang and Zi proved global well-posedness for initial data whose vertical velocity field can be large. Constantin and Sun proved global existence for small data with large gradients for Oldroyd-B, and considered regularization of Oldroyd-B model in \cite{MR2901300}. Thomases and Shelley provided numerical evidence for singularities for Oldroyd-B system in \cite{Thomases2007}. 
Next we discuss the results for diffusive Oldroyd-B models. Barrett and Boyaval proved global existence of weak solution in \cite{MR2843021}. In \cite{MR2989441}, Constantin and Kliegl proved global well-posedness of strong solution. 
Also we refer to Elgindi and Rousset (\cite{MR3403757}) and Elgindi and Liu (\cite{MR3349425}) for Oldroyd-B type systems where fluid viscosity is ignored. 
\paragraph{Multiscale models, especially FENE models.} Macro-micro models, especially FENE models and some simplifications of them have been studied by many authors. In this paragraph, we discuss results concerning non-diffusive multiscale models. Renardy proved local existence of solution for FENE models in Sobolev space with potential $U(m) = (1 - |m|^2 ) ^{1 - k}$ for some $k>1$, as well as infinitely extensible models, in \cite{MR1084958}. E, Li, and Zhang considered modified models with stochastic setting in \cite{MR2073140}. Jourdain, Leli\'{e}vre, and Le Bris proved local existence for the FENE model in \cite{MR2039220}, in the setting of coupled system of Navier-Stokes equation and stochastic Fokker-Planck equation. Jourdain, Le Bris, Leli\'{e}vre, Otto proved exponential convergence to equilibrium in \cite{MR2221204} using entropy inequality method. There are also various other local existence results, for example Zhang and Zhang (\cite{MR2221211}), Kreml and Pokorn\'{y} (\cite{MR2610567}), and Masmoudi (\cite{MR2456183}). In \cite{MR2456183} the author controlled the stress tensor by the $H^1$ norm in $m$ coming from diffusion in $m$, thanks to Hardy type inequalities, and noted that initial data do not need to be regular in $m$ variable. Lin, Liu, and Zhang discussed near-equilibrium situations in \cite{MR2306223}. In \cite{MR2454611}, Masmoudi, Zhang, and Zhang proved global well-posedness for corotational case. One remarkable result, global existence of weak solution for FENE model, is proved by Masmoudi in \cite{MR3010381}. The author used defect measure to overcome difficulties from compactness issue. 
\paragraph{Smoluchowski models.} Smoluchowski equations, which refer to the models whose configuration spaces $M$ are compact manifolds, are also discussed by various authors. In \cite{MR2276466}, Constantin, Fefferman, Titi, and Zarnescu studied nonlinear Fokker-Planck equation driven by a time averaged Navier-Stokes system in 2D. Constantin (\cite{MR2391531}), Constantin and Masmoudi (\cite{MR2367203}), Constantin and Seregin (\cite{MR2667634}, \cite{MR2600741}) showed global existence of smooth solutions for large data in 2D was established. In addition, Otto and Tzavaras discussed Doi model in \cite{MR2365451}. 
\paragraph{Diffusive models and other regularized models.} There are results concerning regularized dumbbell models, for example introducing mollifiers to some terms in the equation (\cite{MR2397999}). Especially, dumbbell models with center-of-mass diffusion are discussed by Barrett and S\"{u}li (\cite{MR2338493}, \cite{MR2819196}, \cite{MR2981267}, \cite{MR2902849}, \cite{MR3698145}) , and Barrett and Boyaval (\cite{MR2843021}). Also Schonbek discussed the regularized model, with corotational assumption in \cite{MR2507462}.
\paragraph{A remark on multiscale models.} Concerning the polymer distribution of the macro-micro models, we note that there are two important remarks that were made in previous works. First, in \cite{MR1084958} Renardy pointed out that the natural setting for the distribution is $L^1$ space. Thus, the author proposed a Frech\'{e}t space based on weighted $L^1$ norms and it is used in \cite{MR2221211} also. However, this space involves derivatives of distributions in $m$ variable. Second, in \cite{MR2456183} Masmoudi used a function space which does not require a regularity in $m$ variable. However, the space is $L^2$ based; it requires square integrability of the distribution in the weighted space, that is, $f \in H^1 (\Omega; L^2 (\frac{1}{f_\infty} dm ) dx$, where $\Omega$ is the spatial domain and $f_\infty = e^{-U}$ is the equilibrium distribution. Although $L^2$ based function spaces are widely used (\cite{MR2456183}, \cite{MR2306223}, \cite{MR2454611}, \cite{MR3698145}) for polymer distribution, we propose a function space based on $L^1$ space, following Renardy's point. As far as we know, function spaces used in most literature do not satisfy both criterion simultaneously. One notable exception is \cite{MR3010381}, but we cannot directly apply the method used in \cite{MR3010381} since the proof relies on the finite extensibility of polymers. 
\paragraph{Free energy estimate.} The free energy estimate, which states that the free energy of the system does not increase over time, is well known and widely used. Especially, in kinetic theory literature, it is widely used to prove the convergence to equilibrium (\cite{MR2065020}, \cite{MR1951784}, \cite{MR2409469}). We were not able to find a rigorous proof of this free energy estimate in the coupled setting, and we provide one in the paper. In addition, we report that when the domain is unbounded, there might be a pathological example if no constraint on decay is imposed.

\subsection{Problem description}
We are interested in the following system:
\begin{equation}
\begin{gathered}
\partial_t u + u \cdot \nabla_x u = - \nabla_x p + \nu_1 \Delta_x u + K \nabla_x \cdot \sigma, \\
\nabla_x \cdot u = 0, \\
\sigma = \int_{\mathbb{R}^2} m \otimes (\nabla_m U(m) ) f (x, t, m) dm, \\
\partial_t f + u \cdot \nabla_x f + (\nabla_x u) m \cdot \nabla_m f = \epsilon \left ( \Delta_m f + \nabla_m \cdot \left ( f \nabla_m U \right ) \right ) + \nu_2 \Delta_x f, \\
U(m) = |m|^{2q}, \\
u(0) = u_0, f(0) = \mu_0,
\end{gathered} \label{system}
\end{equation}
where $q \ge 1$ is a real number, and the vector of position, configuration, and time $(x, m, t)$ is in $\mathbb{R}^2 \times \mathbb{R}^2 \times (0, T)$. For the simplicity of notation, we assume that $q$ is an integer, but our method works for any real number $q \ge 1$. We may also normalize $\mu_0$ so that $\int_x \int_m \mu_0 (dm) dx = 1$. The variable $u$ represents the velocity of the solvent fluid, $p$ represents the pressure, $f$ represents the distribution of the polymer, $\sigma$ represents the stress field due to polymer, and $\nu_1, K, \epsilon, \nu_2$ are positive constants. We want to investigate the existence and uniqueness of smooth solution for this system. However, we note that the regularity required for the macroscopic equation (the first equation of (\ref{system})) is not same as the regularity required for the microscopic equation (the fourth equation of (\ref{system})); for flows of the fluid to be smooth, we need the smoothness for $u$, but the only thing that we require for $f$ is the smoothness of $\sigma[f]$. In particular, smoothness in $m$ variable does not seem to be important. In addition, since $f$ contributes to flows of the whole system only by the macroscopic quantity $\sigma[f]$, it would be interesting if we can transform this microscopic-macroscopic system into a fully macroscopic system, possibly a coupled system of infinitely many variables. In this regard, we define the moment solution in section \ref{Momsolprop}, which is a sense of solution for the microscopic equation that we use in this paper. In short, a moment solution is a weak solution such that all moments of $f$ are controlled. A moment of $f$ is a weighted (usually weights are monomials $m^{I}$) integral in $m$ variable, and thus, a macroscopic quantity, depending only on $x$ and $t$. Appropriate initial data for moment solutions are nonnegative measures on $\mathbb{R}_m ^2 \times \mathbb{R}_x ^2$ such that norms of moments of them are controlled. \newline
\begin{remark}
We remark that the idea of transforming an equation to the coupled system of infinitely many variables is not new. In the context of turbulence theory, Friedmann-Keller equation (\cite{Monin1971}) employs an infinite chain of equations for the infinite set of moments.
\end{remark}
Next, we state our main results.We first prove the existence and uniqueness of the moment solution, given smooth flow $u$:
\begin{theorem}[Theorem \ref{momentsolutionexists}]
Given a smooth fluid field $u$(satisfying (\ref{velinitcond})), and appropriate initial data $\mu_0$ (satisfying (\ref{initpositivity}), (\ref{initmoment}), and (\ref{initstress})), there exists unique moment solution for the fourth equation of (\ref{system}). Furthermore, various norms of moments of this moment solution are controlled solely by the initial data and flow field $u$ (estimates (\ref{XrtypeM}), (\ref{L2H1M}), (\ref{L2H1M2}), (\ref{L2H2qM}), and (\ref{LinfL1M})).
\end{theorem}
Presence of the term $\epsilon \nabla_m \cdot \left ( f \nabla_m U \right )$ introduces higher order terms to evolution equations of moments if $q > 1$. Another problem in the justification of this formal calculation is the potential loss of decay in $m$; in formal derivation of evolution equations of moments, we use integration by parts to deal with terms with $\nabla_m f$. We need to know the finiteness of higher moments to justify the integration by parts. In the paper, we see how to overcome this difficulty. Next, we prove that the stress field depends continuously on the flow field. For this result we require finite entropy condition for the initial data.
\begin{theorem}[Theorem \ref{momentsolutionexists}]
Given two smooth fluid fields $u$, $v$, and appropriate initial data $\mu_0$ satisfying finite entropy condition (\ref{initentropy}), if we let $\sigma_1$ and $\sigma_2$ to be stress fields of the moment solutions with velocity fields $u$ and $v$, respectively, then $\sigma_1 - \sigma_2$ is controlled by $u-v$ ((\ref{differenceestimate})).
\end{theorem}
The main reason why we need the finite entropy condition is that we have to deal with $\nabla_m f$ term when taking difference $\sigma_1 - \sigma_2$. It will be clear in the paper that we cannot simply use integration by parts to rule out derivatives in $m$ variable in this case. Then the above theorems can be used to prove local existence and uniqueness of the solution of the system (\ref{system}), using the contraction mapping scheme.
\begin{theorem}[Theorem \ref{localexistence}]
Given $u_0 \in \mathbb{P} W^{2,2}$ and appropriate initial data $\mu_0$ with finite entropy condition, there is a unique solution $(u, f)$ for the system (\ref{system}) for some time. $u$ is the strong solution for macroscopic equation, and $f$ is the moment solution for the microscopic equation with the velocity field $u$.
\end{theorem}
In addition, this result shows that for the Hookean spring potential case ($q=1$), the Oldroyd-B model is the exact closure of the system (\ref{system}). This extends the result (\cite{MR3698145}) of Barrett and S\"{u}li to a larger class of data. Next, we prove global existence and uniqueness of the system (\ref{system}). The proof uses arguments from \cite{MR2989441}, but the first step, (\ref{uL2sigmaL1}), needs a justification, since it involves an $L^1$ estimate for the stress field.
\begin{theorem}[Theorem \ref{globalexistence}]
Given $u_0 \in \mathbb{P} W^{2,2}$, appropriate initial data $\mu_0$ with finite entropy condition, and an arbitrary $T>0$ there exists a unique solution $(u, f)$ for $(0, T)$. In addition, there are explicit bounds ((\ref{uL2sigmaL1}), (\ref{B2_1}), (\ref{B2_2}), (\ref{B3}), (\ref{B4}), and (\ref{B5}) ) for the norm of the solution.
\end{theorem}
Finally, we establish a free energy estimate. Here we make an additional assumption (\ref{initdenentropy}), to guarantee that initial free energy is finite.
\begin{theorem}[Theorem \ref{Freeenergybound}]
For the solution of the system (\ref{system}), its free energy, which is defined as the sum of kinetic energy of the fluid ($\norm{u(t) }_{L^2} ^2$) and free energy of polymer distribution $\left (\int f(t) \log \left ( \frac{f(t) } {\int f(t) dm \frac{e^{-U(m)}}{Z} } \right ) dm dx \right )$, does not increase over time (bound (\ref{entropyestms}).
\end{theorem}
The main challenge for proving this theorem is to control the limit of integrals of nonlinear terms.

\subsection{Structure of the paper}
In section \ref{FsMs}, we introduce relevant functional settings. Specifically, in section \ref{Prel}, we review some basic facts about moment problems, functional analysis, and parabolic PDEs. In section \ref{MFS}, we introduce the function space we use to describe the distribution of polymers. Then in section \ref{Momsolprop} we define the notion of moment solution and investigate its basic properties. Using the settings in the previous section, in section \ref{Solsch}, well-posedness of microscopic equation in the sense of moment solution is outlined, given smooth velocity field $u$. In section \ref{Appsol}, we present the approximation scheme. Main modifications to the original microscopic equation are introduction of cutoff in $m$ variable and mollification of initial data, so that we can integrate by parts freely and they remain smooth. In section \ref{unifBdmom}, we find uniform bounds for moments of approximate solutions, and in section \ref{momsolexistence} we find the moment solution as measures in $m$, which are determined by limits of those moments. The main issue here is that the sense of limit for moments is weaker than pointwise, so we have to rely on Aubin-Lions compactness theorem to establish pointwise convergence and apply results from moment problems. In section \ref{Flfielddep}, we investigate the dependence of stress field on fluid velocity field. In section \ref{wellposedness} we prove local and global well-posedness for the coupled system, and then provide a rigorous proof for the free energy estimate. In section \ref{lwp}, we prove local existence and uniqueness using contraction mapping scheme, and in section \ref{gwp}, we prove global existence for the system, and we obtain explicit bounds for $u$. The coupling of the energy of fluid field and the trace of stress field is crucial in the proof. In section \ref{freeEbound}, we prove the free energy estimate.

\section{Function space and Moment solution} \label{FsMs}

\subsection{Preliminaries} \label{Prel}
Let $\mathcal{M} (\mathbb{R}^2)$ be the space of signed Borel measures. $\mathcal{M} (\mathbb{R}^2)$ is a Banach space, where the norm is the total variation of $\mu$, $|\mu| (\mathbb{R}^2)$. Given $\mu \in \mathcal{M} (\mathbb{R}^2)$, we denote the moment of $\mu$ as
\begin{equation}
M_{a,b} [\mu] = \int_{\mathbb{R}^2} m_1 ^a m_2 ^b \mu(dm),
\end{equation}
where $a, b \ge 0$ are integers, the radial absolute moment of $\mu$ as
\begin{equation}
\bar{M}_{k} [\mu] = \int_{\mathbb{R}^2} |m|^k |\mu| (dm)
\end{equation}
where $k \ge 0$ is an integer, the vector of moments of degree $k$ as
\begin{equation}
\dot{\vec{M}}_{k} [\mu] = \left ( M_{k,0} [\mu], M_{k-1,1} [\mu], \cdots, M_{0,k} [\mu] \right )
\end{equation}
and the vector of moments of degrees up to $k$ as
\begin{equation}
\vec{M}_{k} [\mu] = \left ( \dot{\vec{M}}_{0}[\mu], \dot{\vec{M}}_{1} [\mu], \cdots, \dot{\vec{M}}_{k} [\mu] \right ),
\end{equation}
and the vector of moments of even degrees up to $2k$ as
\begin{equation}
\vec{M}_{2k} ^e [\mu] = \left ( \dot{\vec{M}}_{0}[\mu], \dot{\vec{M}}_{2} [\mu], \cdots, \dot{\vec{M}}_{2k} [\mu] \right ).
\end{equation}
In probability theory, moment problem refers to the problem of determining a probability measure when moments are given. We only briefly mention what is needed for us, and more detailed explanation can be found in \cite{MR3708381}. We first introduce the Riesz functional and positive semidefinite sequence.
\begin{definition}[Riesz' functional]
Given $m = \{ m_{a,b} \}_{(a,b) \in \mathbb{Z}_{\ge 0} ^2 }$, we define the associated Riesz functional $L_m$ on $\mathbb{R}[x]$ by $L_m (x^I) := m^I$ for all $I = (a,b) \in  \mathbb{Z}_{\ge 0} ^2 $.
\end{definition}
\begin{definition}[Positive semidefinite sequence]
A sequence $m = \{ m_{a,b} \}_{(a,b) \in \mathbb{Z}_{\ge 0} ^2 }$ of real numbers is said to be positive semidefinite if for any $k \in \mathbb{N}$,  $c_1, \cdots, c_k \in \mathbb{R}$ and $(a_1, b_1) , (a_2, b_2), \cdots, (a_k, b_k ) \in  \mathbb{Z}_{\ge 0} ^2$, 
\begin{equation}
\sum_{i,j=1} ^k m_{(a_i, b_i) + (a_j, b_j) } c_i c_j \ge 0
\end{equation}
holds, or equivalently, $L_m (h^2 ) \ge 0$ for any $h \in \mathbb{R}[x]$.
\end{definition}
For moment problems for measures on $\mathbb{R}^d$, $d \ge 2$, the multivariate Carleman's condition, which is a constraint on the growth rate of moments over degree, provides a sufficient condition for uniqueness.
\begin{theorem}
Let $\mu, \nu$ be positive measures in $\mathbb{R}^2$ where $M_{a,b} [\mu] = M_{a,b} [\mu] < \infty$. Let $m = \{ M_{a,b} \}_{(a,b) \in \mathbb{Z}_{\ge 0} ^2 }$. If
\begin{equation}
\sum_{n=1}^\infty L_m ( x_1 ^{2n})^{-\frac{1}{2n}} = \sum_{n=1}^\infty L_m ( x_2 ^{2n})^{-\frac{1}{2n}} = \infty \label{mulCar}
\end{equation}
then $\mu = \nu$.
\end{theorem}
The condition (\ref{mulCar}) is known as the multivariate Carleman's condition. 
\begin{theorem}
Let $m = \{ m_{a,b} \}_{(a,b) \in \mathbb{Z}_{\ge 0} ^2 }$ be a positive semidefinite sequence satisfying the multivariate Carleman's condition (\ref{mulCar}). Then there exists a unique non-negative Borel measure $\mu$ such that $m_{a,b} = M_{a,b} [\mu]$ for all $(a,b)$. \label{measureexist}
\end{theorem}
Also we need the following result, which states that if a given measure is determined uniquely by its moments, and if moments of a sequence of measures converge to moments of this measure, then the sequence of measures converge to the measure weakly. We mainly refer to \cite{MR2267655}. 
A sequence of (signed) Borel measures on $\mathbb{R}^2$ is uniformly tight if for every $\epsilon >0$ there is a compact set $K_\epsilon \subset \mathbb{R}^2$ such that $|\mu_n| (\mathbb{R}^2 - K_\epsilon ) < \epsilon$ for all $n$. Also we define the weak convergence of measures.
\begin{definition}
A sequence of Borel measures on $\mathbb{R}^2$ $\{ \mu_n \}$ is called weakly convergent to a Borel measure $\mu$ on $\mathbb{R}^2$ if for every bounded continuous real function $f$ on $\mathbb{R}^2$, one has
\begin{equation}
\lim_{n\rightarrow \infty} \int_{\mathbb{R}^2} f(m) \mu_n (dm) = \int_{\mathbb{R}^2} f(m) \mu (dm).
\end{equation}
\end{definition}
The following lemma is useful.
\begin{lemma}
Let $\mu_n$ be a sequence of nonnegative Borel measures on $\mathbb{R}^2$ which is uniformly bounded in total variation norm and converges weakly to a Borel measure $\mu$. Then for every continuous function $f$ on $\mathbb{R}^2$ satisfying the condition
\begin{equation}
\lim_{R\rightarrow \infty} \sup_n \int_{|f| \ge R} |f| \mu_n (dm) = 0,
\end{equation}
one has
\begin{equation}
\lim_{n \rightarrow \infty} \int_{\mathbb{R}^2} f \mu_n (dm) = \int_{\mathbb{R}^2} f \mu (dm).
\end{equation} \label{uniformintegrability}
\end{lemma}
\begin{proof}
First we let $f_m = \min \left ( |f|, m \right )$. Then $f_m \le |f|$, and from the assumption on $f$ there is some $R_0 >0$ such that 
\begin{equation}
\sup_n \int_{|f| \ge R_0 } |f| \mu_n (dm) \le 1,
\end{equation}
while 
\begin{equation}
\sup_n \int_{|f| \le R_0 } |f| \mu_n (dm) \le R_0 \sup_n \int_{\mathbb{R}^2} \mu_n (dm) = R_0 C < \infty
\end{equation}
so that 
\begin{equation}
\sup_{n,m} \int_{\mathbb{R}^2} f_m \mu_n (dm) \le 1 + C R_0 = M <\infty.
\end{equation}
Since $f_m$ is continuous and bounded, by weak continuity we have
\begin{equation}
\int_{\mathbb{R}^2} f_m \mu(dm) \le M
\end{equation}
and by monotone convergence we have $f \in L^1 (\mu)$. For a given $\epsilon > 0$, we can pick $R > 0$ such that there is some $N>0$ such that for all $n \ge N$
\begin{equation}
\int_{|f| \ge R} |f| \mu_n (dm) + \int_{|f| \ge R} |f| \mu (dm) < \epsilon.
\end{equation}
Let $g = \max ( \min (f, R), -R)$ be the truncation of $f$ up to $R$: $g = f$ if $|f| < R$, $g = R$ if $f \ge R$, and $g = -R$ if $ f \le -R$. Since $g$ is continuous and bounded, there is some $N' >N$ such that for all $n \ge N'$
\begin{equation}
\left | \int_{\mathbb{R}^2} g \mu_n (dm) - \int_{\mathbb{R}^2} g \mu (dm) \right | < \epsilon.
\end{equation}
Then for such $n$, we have
\begin{equation}
\left | \int_{\mathbb{R}^2} f \mu_n (dm) - \int_{\mathbb{R}^2} f \mu (dm) \right | < 3\epsilon,
\end{equation}
as desired.
\end{proof}
Then the Prohorov's theorem states the following.
\begin{theorem}[Prohorov]
The sequence $\mu_n$ of (signed) Borel measures on $\mathbb{R}^2$ contains a weakly convergent subsequence if and only if $\mu_n$ is uniformly tight and uniformly bounded in the total variation norm. \label{Prohorov}
\end{theorem}
Using Prohorov's theorem and Lemma \ref{uniformintegrability}, we can prove the following (\cite{MR2893652}):
\begin{theorem}
Suppose that $\mu_n$ is a sequence of nonnegative Borel measures on $\mathbb{R}^2$ having all moments $M_{a,b} [\mu_n] <\infty$, and $\mu$ is a nonnegative Borel measure on $\mathbb{R}^2$ with $M_{a,b} [\mu] <\infty$ too. Suppose that $\mu$ is determined by its moment: if there is a nonnegative Borel measure $\nu$ such that $M_{a,b} [\mu] = M_{a,b} [\nu]$ for all $a, b$, then $\mu = \nu$. Also suppose that $M_{a,b} [\mu_n] \rightarrow M_{a,b} [\mu]$ for all $a, b$. Then $\mu_n$ converges to $\mu$ weakly, at least for a subsequence. \label{weakconvergence}
\end{theorem}
\begin{proof}
First note that $\bar{M}_{2} [\mu_n]$ is uniformly bounded, say by $C$, since it is convergent: then by Chebyshev, we have
\begin{equation}
\mu_n \left ( \{ m \in \mathbb{R}^2 : |m| > K \} \right ) \le \frac{C}{K^2},
\end{equation}
so $\mu_n$ is uniformly tight. Also since $M_{0,0} [\mu_n]$ is also uniformly bounded, so $\mu_n$ has a weakly convergent subsequence, converging to $\nu$. Note that all $M_{a,b} [\mu_n]$ is uniformly bounded due to convergence, and note that for $a, b \ge 0$ we have that $\mu_n ^{a,b, +} = \frac{(m_1 ^a m_2 ^b )^+ }{1 + |m|^{a+b} } \mu_n $ converges weakly to $\frac{(m_1 ^a m_2 ^b )^+ }{1 + |m|^{a+b} } \nu $ and $\mu_n ^{a, b, -} = \frac{(m_1 ^a m_2 ^b )^- }{1 + |m|^{a+b} } \mu_n $ converges weakly to $\frac{(m_1 ^a m_2 ^b )^- }{1 + |m|^{a+b} } \nu $. Those measures are uniformly bounded in total variation norm, and 
\begin{equation}
\begin{gathered}
\lim_{R \rightarrow \infty } \sup_n \int_{|m|^{a+b} + 1  \ge R^{a+b} +1 } (|m|^{a+b} + 1) \mu_n ^{a,b,+} \\ = \lim_{R \rightarrow \infty } \sup_n \int_{|m|^{a+b} + 1   \ge R^{a+b} +1 } (m_1 ^a m_2 ^b ) ^+ \mu_n (dm) \\
 \le \lim_{R \rightarrow \infty } \frac{1}{R} \sup_n  \int_{|m|^{a+b} + 1  \ge R^{a+b} +1 } |m|^{a+b+1} \mu_n (dm) =0
\end{gathered}
\end{equation}
and same for $\mu_n ^{a,b,-}$. Therefore, by Lemma \ref{uniformintegrability}, we have
\begin{equation}
\lim_{n\rightarrow\infty} \int_{\mathbb{R}^2} m_1 ^a m_2 ^b \mu_n (dm) = \int_{\mathbb{R}^2} m_1 ^a m_2 ^b \nu (dm)
\end{equation}
or $M_{a,b} [\mu_n] \rightarrow M_{a,b} [\nu]$. But since $\mu$ is determined by its moments, we have $\mu = \nu$.
\end{proof}
\begin{remark}
If $\{ M_{a,b} [\mu] \}_{(a,b)}$ satisfies the multivariate Carleman's condition (\ref{mulCar}), and if $\mu_n$ satisfies all the assumptions in Theorem \ref{weakconvergence}, then for all $p \ge 0$ $|m|^{2p} \mu_n$ converges weakly to $|m|^{2p} \mu$ in a subsequence. 
\end{remark}
\begin{proof}
First, we observe that
\begin{equation}
M_{2j,0} [m_1 ^{2p} \mu] =  M_{2(j+p), 0} [\mu],
\end{equation}
which guarantees that $M_{a,b} [m_1 ^{2p} \mu]$ also satisfies the multivariate Carleman's condition. The proof of this claim is given in the last. Also, $m_1 ^{2p} \mu$ is also determined by its moments, and same for $m_2 ^{2p} \mu$. Therefore, by Theorem \ref{weakconvergence} we see that $m_1 ^{2p} \mu_n$ weakly converges to $m_1 ^{2p} \mu$ and similarly $m_2 ^{2p} \mu_n$ weakly converges to $m_2 ^{2p} \mu$. Also $\mu_n$ weakly converges to $\mu$, and we see that $\frac{|m|^{2p}}{1+m_1 ^{2p} + m_2 ^{2p} }$ is a continuous bounded function, so $|m|^{2p} \mu_n $ converges to $|m|^{2p} \mu$ weakly (in subsequence). It only remains to show that 
\begin{equation}
\sum_{j=1} ^{\infty} \left ( \frac{1}{M_{2(j+p),0} } \right )^{-\frac{1}{2j} } = \infty.
\end{equation}
Since $c_j = M_{2(j+p),0}$ satisfies $c_j ^2 \le c_{j-1} c_{j+1}$, by Denjoy-Carleman Theorem (\cite{MR3708381}) it is equivalent to show
\begin{equation}
\int_1 ^\infty \frac{\log T_p (r) }{r^2} dr = \infty,
\end{equation}
where $T_p (r) = \max_j \frac{r^j}{M_{2(j+p),0}}$. However, 
\begin{equation}
T_p (r) = \max_{j\ge 1} \frac{r^j}{M_{2(j+p),0}} \ge \max_{j \ge 1+p} \frac{r^{j}}{M_{2j,0}} \frac{1}{r^p} = T(r) \frac{1}{r^p}.
\end{equation}
But note that already we know $\int_1 ^\infty \frac{\log T(r)}{r^2} dr = \infty$, and $\int_1 ^\infty \frac{\log r}{r^2} dr < \infty$ so we are done.
\end{proof}
Also, we have the following Fatou-type lemma.
\begin{lemma}[Varadarajan]
Suppose that the sequence of (signed) Borel measures $\mu_n$ converges weakly to a Borel measure $\mu$. Then for any functionally open ($f^{-1} ((0, \infty) )$ for some continuous function $f$ on $\mathbb{R}^2$) set $U$ we have
\begin{equation}
\lim \inf_{n} |\mu_n| (U) \ge |\mu| (U).
\end{equation}
In this situation, the sequence $|\mu_n |$ converges weakly to $|\mu|$ precisely when $|\mu_n| (\mathbb{R}^2) \rightarrow |\mu|(\mathbb{R}^2)$.
\label{Fatou}
\end{lemma}
On the other hand, we also need the following (\cite{MR0453964}).
\begin{theorem}
Let $[0, T]$ be endowed with usual $\sigma$-algebra and Lebesgue measure. Let $X$ be a reflexive Banach space. 
For any $1 \le p < \infty$, $(L^p(0, T; X))^* \simeq L^{q}(0, T; X^*)$ where $\frac{1}{p} + \frac{1}{q} = 1$.
\label{dualofBanach}
\end{theorem}
Also we use Banach-Alaoglu theorem.
\begin{theorem}[Banach-Alaoglu]
Let $X$ be a normed space. Hence $X^*$ is also normed with the operator norm. Then the closed unit ball of $X^*$ is compact with respect to the weak* topology.
\end{theorem}
We also need Rellich-Kondrachov theorem and Aubin-Lions lemma.
\begin{theorem}[Rellich-Kondrachov]
Suppose that $\Omega$ is bounded domain with smooth boundary. Then the inclusion $W_0 ^{1,2} (\Omega) \subset L^2 (\Omega)$ and $W_0 ^{1,1} (\Omega) \subset L^1 (\Omega)$ are compact.
\end{theorem}
\begin{theorem}[Aubin-Lions]
Let $X_0, X_1, X_2$ be three Banach spaces, $X_0 \subset X_1 \subset X_2$. Suppose that $X_0$ is compactly embedded in $X_1$ and $X_1$ is continuously embedded in $X_2$. For $1 \le p, q, \le \infty$, let
\begin{equation}
W = \{ u \in L^p ([0, T] ; X_0 ) \, : \, \partial_t u \in L^q ([0, T] ; X_2) \}.
\end{equation}
If $p < \infty$, the embedding $W \subseteq L^p ([0, T]; X_1)$ is compact. If $p = \infty$ and $q > 1$, the embedding $W \subseteq L^p ([0, T]; X_1 )$ is compact.
\end{theorem}
Also we use results from parabolic theory, especially existence, uniqueness, and estimates of Fokker-Planck-Kolmogorov equations. We mainly refer to \cite{MR3443169}. Suppose we are given an open set $\Omega_T = \Omega \times (0, T) \subset \mathbb{R}^d \times (0, T)$, where $\Omega \subset \mathbb{R}^d$ is an open set and $T>0$, and Borel functions $a^{ij}$, $b^i$, and $c$ on $\Omega_T$, where $i,j = 1, \cdots, d$. We suppose that the matrix $A = \left ( a^{ij} \right )_{ij}$ is symetric nonnegative definite. We discuss the Fokker-Planck-Kolmogorov equation of the form 
\begin{equation}
\partial_t \mu = \partial_{x_i} \partial_{x_j} \left ( a^{ij} \mu \right ) - \partial_{x_i} \left ( b^i \mu \right ). \label{FPK}
\end{equation}
Let 
\begin{equation}
L_{A, b} \phi = a^{ij} (x, t) \partial_{x_i} \partial_{x_j} \phi (x, t) + b^{i} (x, t) \partial_{x_i} \phi (x, t),
\end{equation}
which is the adjoint operator of the right side of (\ref{FPK}). 
\begin{definition}
A locally bounded Borel measure $\mu$ on the domain $\Omega_T$, which can be written as $\mu = \mu_t (dx) dt$ is a solution to the Cauchy problem (\ref{FPK}) with $\mu|_{t=0} = \nu$ if $a^{ij}, b^i \in L^1 _{loc} (\mu)$, for every function $\phi \in C_0 ^\infty (\Omega_T )$ we have
\begin{equation}
\int_{\Omega_T} \left ( \partial_t \phi + L_{A, b} \phi \right ) d\mu = 0, \label{weaksolution}
\end{equation}
and for every function $f \in C_0 ^\infty (\Omega)$ there is a set of full measure $J_f \subset (0, T)$, depending on $f$, such that
\begin{equation}
\int_{\Omega} f(x) \nu (dx) = \lim_{t \rightarrow 0, t \in J_f} \int_\Omega f(x) \mu_t (dx).
\end{equation}
\end{definition}
Note that this definition is equivalent to the following: for every function $\phi \in C_0 ^\infty (\Omega)$ there exists a set of full measure $J_\phi \subset (0, T)$, depending on $\phi$, such that for all $t \in J_\phi$ we have
\begin{equation}
\int_\Omega \phi d\mu_t = \int_\Omega \phi d\nu + \lim_{\tau \rightarrow 0+, \tau \in J_\phi } \int_{\tau} ^t \int_\Omega L_{A, b} \phi d\mu_s ds. \label{solCauchy}
\end{equation}
We have the following results. For the proof one can see \cite{MR3443169}, where more general statements and proof are given. Let $\Omega = \mathbb{R}^d$.
\begin{theorem} [Existence, existence of density, and uniqueness of Fokker-Planck-Kolmogorov equation] \label{FPKeu}
Suppose that for every ball $U$ in $\mathbb{R}^d$ the functions $a^{ij}$, $b^i$ are bounded in $U \times [0, T]$ and there exist positive numbers $m$ and $M$ such that $$m \mathbb{I}_d \le A(x, t) \le M \mathbb{I}_d , (x, t) \in \Omega \times [0, T]$$ and there exist positive number  $\lambda$ such that  $$|a^{ij} (x, t) - a^{ij} (y, t) | \le \lambda |x-y| , x, y \in \mathbb{R}^d, t \in (0, T) $$ holds. Then for every probability measure $\nu$, there is a solution to the Cauchy problem (\ref{FPK}) with $\mu |_{t=0} = \nu$, where each $\mu_t$ is a nonnegative Borel measures on $\mathbb{R}^d$, such that for almost all $t \in (0, T)$ we have
\begin{equation}
\mu_t (\mathbb{R}^d) \le \nu(\mathbb{R}^d). \label{subprob}
\end{equation}
Also, $\mu = \rho dx dt$ for some locally integrable function $\rho$. If $J = [T_0, T_1] \subset (0, T)$, $W$ is a neighborhood of $\bar{U} \times J$ with compact closure in $\Omega_T$, then for each $r < \frac{d+2}{d+1} $ one has
\begin{equation}
\norm{\rho}_{L^r (U \times J) } \le C(d, r, \lambda, m, M, W ) \left ( \mu (W) + \norm{b}_{L^1 (W, \mu) } \right )
\end{equation}
where $C(d, r, \Lambda, m, M, W )$ depends only on $d, r, \lambda, m, M,$ and the distance from $U \times J$ to $\partial W$. In addition, suppose further that $\mu$ satisfies the following: for every ball $U \subset \mathbb{R}^d$
\begin{equation}
|b| \in L^2 (\mu, U \times (0, T) ) \label{SP}
\end{equation}
and 
\begin{equation}
|a^{ij}| + |b^i| \in L^1 \left (\mu, \mathbb{R}^d \times (0, T) \right ).
\end{equation}
Then there is no solution to the Cauchy problem (\ref{FPK}) with $\mu |_{t=0} = \nu$ satisfying (\ref{subprob}) and (\ref{SP}) other than $\mu$. Furthermore, suppose that there is a function $V \in C^{2,1} (\Omega_T) \cap C(\mathbb{R}^d \times [0, T) )$ such that for every compact interval $[\alpha, \beta ] \subset (0, T)$ we have
\begin{equation}
\lim_{|x| \rightarrow \infty } \min_{t \in [\alpha, \beta] } V(x, t) = + \infty
\end{equation}
and for some $K, H \in L^1 ((0, T))$, where $H \ge 0$, and for all $(x, t) \in \Omega_T$
\begin{equation}
\partial_t V(x, t) + L_{A,b} V (x, t) \le K(t) + H(t) V(x, t),
\end{equation}
and also $V(\cdot ,0 ) \in L^1 (\nu)$. Then for almost all $t \in (0, T)$ we have $\mu_t (\mathbb{R}^d ) = \nu (\mathbb{R}^d ) = 1$.
\end{theorem} 
Also we have the following result for the square integrability of logarithmic gradients. First we adopt the following convention: for $\rho (x, t) \in W^{1,1}_{loc},$ $$\frac{\nabla_x \rho (x, t)}{\rho (x, t)} := 0 $$ whenever $\rho (x, t) = 0$. Also we recall that a probability measure $\nu$ on $\mathbb{R}^d$ has finite entropy if $\nu = \rho_0 dx$ and $$ \int_{\mathbb{R}^d} \left  | \log \rho_0 (x) \right | \rho_0 (x) dx < \infty.$$
\begin{theorem} [Bounds on entropy production] \label{entropycon}
Suppose that a measure $\mu = (\mu_t)$ is a solution to the Cauchy problem (\ref{FPK}) with $\mu |_{t=0} = \nu$, each $\mu_t$ is a probability measure, and same condition for $a^{ij}$ as in Theorem \ref{FPKeu} holds, and $|b| \in L^2 (\mu, \Omega_T )$. Suppose also that the function $\Lambda (x) = \log \max \left ( |x|, 1 \right )$ belongs to $L^2 (\mu, \Omega_T )$. If the initial distribution $\nu = \rho_0 dx$ on $\mathbb{R}^d = \Omega$ has finite entropy, then $\mu_t = \rho(\cdot, t ) dx$, where $\rho (\cdot, t) \in W^{1,1} (\mathbb{R}^d )$, moreover, for every $\tau < T$ we have
\begin{equation}
\int_0 ^\tau \int_{\mathbb{R}^d} \frac{ | \nabla_x \rho (x, t) | ^2 }{\rho(x, t) } dxdt < \infty. \label{entroprodbd}
\end{equation}
If the integrals $\int_{\mathbb{R}^d} \rho(x, t) \Lambda(x) dx$ remain bounded as $t \rightarrow T$, then (\ref{entroprodbd}) is true with $\tau = T$.
\end{theorem}
We also briefly review the proof of Theorem \ref{entropycon} in section \ref{wellposedness}, to establish the free energy estimate.

\subsection{Function space based on moments} \label{MFS}
We introduce relevant function spaces and the notion of moment solution. We first define two power series based on moments: for $\mu \in L_{loc} ^1 (\mathbb{R}^2, \mathcal{M} (\mathbb{R}^2) )$ we let
\begin{equation}
\begin{gathered}
F[\mu] ^e (r) = \sum_{p=0} ^\infty \frac{\norm{\bar{M}_{2p} [\mu]}_{L^2} }{(2p)!} r^{2p}, \\
F[\mu](r) = \sum_{p=0} ^\infty \frac{\norm{\bar{M}_{p} [\mu]}_{L^2} }{p!} r^p.
\end{gathered}
\end{equation}
Note that $F[\mu] (r)$ is a norm in the space
\begin{equation}
X^r = \{ \mu \in L_{loc} ^1 (\mathbb{R}^2, \mathcal{M} (\mathbb{R}^2 ) ) \, : \, \norm{\mu}_{X^r} = F[\mu] (r) < \infty \}.
\end{equation}
Then $F[\mu] ^e (r)$ is an equivalent norm in $X^r$. Obviously, $F[\mu]^e (r) \le F[\mu] (r)$. On the other hand, by Cauchy-Schwarz inequality,
\begin{equation}
\frac{\norm{\bar{M}_{2j+1}}_{L^2}}{(2j+1)!} r^{2j+1} \le \frac{\norm{\bar{M}_{2j}}_{L^2}}{(2j)!} r^{2j} + \frac{\norm{\bar{M}_{2(j+1)}}_{L^2}}{(2(j+1))!} r^{2(j+1)}  
\end{equation}
and we conclude $F[\mu] (r) \le 3 F [\mu]^e (r)$. We also have the following:
\begin{lemma}
Suppose that $\{M_{a,b} \}_{a,b}$ is a sequence of functions on $\mathbb{R}^2$ such that there is a sequence of functions $\bar{M}_{k}$ on $\mathbb{R}^2$ where
\begin{equation}
\begin{gathered}
\left | M_{a,b} (x) \right | \le \bar{M}_{a+b} (x), \,\, \mathrm{for} \,\,\mathrm{almost} \,\, \mathrm{all} \,\, x, \\
\sum_{p=0} ^\infty \frac{\norm{\bar{M}_{p} }_{L^2} } {p!} r^p < \infty \,\, \mathrm{for} \,\, \mathrm{some} \,\, r >0.
\end{gathered}
\end{equation}
Then for almost every $x \in \mathbb{R}^2$, the sequence $\{M_{a,b} (x) \}_{(a,b)} $ satisfies the multivariate Carleman's condition (\ref{mulCar}). \label{pointwiseCar}
\end{lemma}
\begin{proof}
It suffices to show that for almost every $x$, 
\begin{equation}
\sum_{p=0} ^\infty \bar{M}_{2p} (x) ^{-\frac{1}{2p} } = \infty.
\end{equation}
By Chebyshev's inequality, we have
\begin{equation}
\begin{gathered}
\left | \left \{  x \,: \bar{M}_{2p}  (x) > (2p)!  \left ( \frac{1 }{\lambda} \right  )^{2(p+1)} \right \} \right |  \le 2  \left (\frac{\norm{\bar{M}_{2p}  }_{L^2}} {(2p)!} \lambda ^{2p} \right )^2 \lambda^4 .
\end{gathered}
\end{equation}
Therefore, we have
\begin{equation}
\begin{gathered}
\left | \left \{  x \,: \mathrm{for \,\, some \,\, } p \ge 0, \,\, \bar{M}_{2p}  (x) > (2p)!  \left ( \frac{1 }{\lambda} \right  )^{2(p+1)} \right \} \right |  \le 2   \sum_{p = 0} ^\infty \left (\frac{\norm{\bar{M}_{2p}  }_{L^2}} {(2p)!} \lambda ^{2p} \right )^2 \lambda^4
\end{gathered}
\end{equation}
and by taking $\lambda \rightarrow 0$ we conclude that for almost every $x$, there exist some $\lambda = \lambda(x) \in (0, r) $ such that 
\begin{equation}
\bar{M}_{2p}  (x) \le (2p)!  \left ( \frac{1 }{\lambda} \right  )^{2(p+1)} \,\, \mathrm{for} \,\, \mathrm{all} \,\, p \ge 0
\end{equation}
and thus we have
\begin{equation}
\sum_{p=0} ^\infty \bar{M}_{2p} (x) ^{-\frac{1}{2p} } \ge \sum_{p=1} ^{\infty} \left ( (2p)! \left ( \frac{1}{\lambda} \right )^{2(p+1)} \right )^{-\frac{1}{2p}} \ge C \lambda \sum_{p=1} ^\infty \frac{1}{p} = \infty.
\end{equation}
\end{proof}
We define
\begin{equation}
X^{k,r} = \{ \mu \in X^r \, : \, \nabla_{m, x} ^\ell \mu \in X^r \,\, \mathrm{for} \,\, \mathrm{all} \,\, 0 \le \ell \le k . \}
\end{equation}
We have the following:
\begin{lemma}
$X^{k,r}$ is a Banach space for all $k \ge 0$ with norm $\norm{\mu}_{X^{k,r} } = \sum_{|\ell |\le k }  \norm{\nabla_{x,m}^\ell \mu }_{X^r} $. \label{completeness}
\end{lemma}
\begin{proof}
First note that it suffices to show that $X^r$ is Banach: for a Cauchy sequence $\mu_n$ in $X^{k,r}$ each $\nabla_{x,m} ^\ell \mu_n$ is Cauchy in $X^r$, and $\mu_n \rightarrow  \mu$ in $X^r$ implies $\lim_{n } \nabla_{x,m} ^\ell \mu_n = \nabla_{x,m} ^\ell \mu$. Suppose that $\mu_n$ is a Cauchy sequence in $X^r$. Then we know that all $\bar{M}_{k} [\mu_n]$ is a Cauchy sequence in $L^2 (\mathbb{R}^2)$ and so converges to $\bar{M}_k (x) \in L^2 (\mathbb{R}^2)$. Furthermore, we see that
\begin{equation}
F[\mu_n] (r) = \sum_{p=0} ^\infty \frac{\norm{\bar{M}_p[\mu_n]}_{L^2} }{p!} r^p \rightarrow \sum_{p=0} ^\infty \frac{\norm{\bar{M}_p}_{L^2}}{p!} r^p
\end{equation}
because $G_n (z) = F[\mu_n] (z)$ is a sequence of holomorphic functions in closed $r$-ball which is Cauchy in sup norm:
\begin{equation}
\begin{gathered}
\left | G_n (z) - G_m (z) \right | \le \sum_{p=0} ^\infty \frac{\norm{\bar{M}_p [ \mu_n] - \bar{M}_p [\mu_m] }_{L^2} }{p!} z^p \le F[\mu_n - \mu_m] (z)
\end{gathered}
\end{equation}
so $G_n (z)$ converges to some holomorphic function $G(z)$ uniformly in closed $r$-ball. Then we consider the power series representation of $G(z)$ near 0: its coefficients can be represented by Cauchy integral formula and we see
\begin{equation}
\frac{G^{(m)} (0)}{m!} = \frac{1}{2\pi i } \int_{C(0, a)} \frac{G(z)}{z^{m+1}} dz = \lim_{n\rightarrow \infty} \frac{1}{2\pi i} \int_{C(0,a)} \frac{G_n (z) }{z^{m+1}} dz = \lim_{n \rightarrow \infty} \norm{\bar{M}_m [\mu_n ] }_{L^2}.
\end{equation}
Note that $ M_{a,b} [ |\mu_n | ] (x)   \le \bar{M}_{a+b} [\mu_n] (x) $ and so by dominated convergence we have that
\begin{equation}
\begin{gathered}
M_{a,b} [| \mu_n | ] \rightarrow M_{a,b} ^{tv} \,\,\mathrm{in}\,\, L^2 (\mathbb{R}^2), \\
M_{a,b} [ \mu_n ^+ ] \rightarrow M_{a,b} ^+ \,\,\mathrm{in}\,\, L^2 (\mathbb{R}^2), \\
M_{a,b} [ \mu_n ^- ] \rightarrow M_{a,b} ^- \,\,\mathrm{in}\,\, L^2 (\mathbb{R}^2), \\
M_{a,b} ^{tv} = M_{a,b} ^+ + M_{a,b} ^-, \,\, \left | M_{a,b} ^+ \right |, \left | M_{a,b} ^- \right |,  \left | M_{a,b} ^{tv} \right | \le \bar{M}_{a+b}.
\end{gathered}
\end{equation}
where $\mu_n ^+$ is the positive part (due to Jordan decomposition) of $\mu_n$ and $\mu_n ^-$ is the negative part. In particular, the sequences $\{M_{a,b} ^+ (x) \}_{a,b}$ and $\{M_{a,b} ^- (x) \}_{a,b}$ are positive semidefinite sequences for almost every $x$, because they are pointwise limit of positive semidefinite sequences. Furthermore, by Lemma \ref{pointwiseCar}, and Theorem \ref{measureexist} we see that for almost all $x$, there is a nonnegative measure $\mu^+ (x)$ and $\mu^- (x)$ and subsequences $\mu_{n_k} ^+, \mu_{n_k} ^-$ such that
\begin{equation}
\begin{gathered}
M_{a,b} ^+ (x) = \int_{\mathbb{R}^2} m_1 ^a m_2 ^b \mu^+ (x; dm), \, \, M_{a,b} ^- (x) = \int_{\mathbb{R}^2} m_1 ^a m_2 ^b \mu^- (x; dm),\\
\lim_{k\rightarrow \infty} M_{a,b} [\mu_{n_k} ^+] (x) = M_{a,b} [ \mu^+] (x), \,  \lim_{k\rightarrow \infty}M_{a,b} [\mu_{n_k} ^-] (x) = M_{a,b} [ \mu^-] (x) \,\, \mathrm{a.e.},\\
\bar{M}_{p} (x) = \int_{\mathbb{R}^2} |m|^p \left ( \mu^+ (x; dm) + \mu^- (x; dm) \right ) = \bar{M}_p [\mu^+] (x) + \bar{M}_p [\mu^-] (x).
\end{gathered}
\end{equation}
Furthermore, by putting $\mu (x; dm) = \mu^+ (x; dm) - \mu^- (x; dm)$ we see that
\begin{equation}
\begin{gathered}
F[\mu] (r) = \sum_{p=0} ^\infty  \frac{\norm{\bar{M}_p [\mu]}_{L^2} }{p!} r^p = \lim_{n\rightarrow \infty} F[\mu_n] (r) < \infty.
\end{gathered}
\end{equation}
To show that $\mu_n$ converges to $\mu$ in $X^r$, we evaluate the equivalent norm $F[\mu - \mu_n ] ^e (r)$: first we know that from Theorem \ref{weakconvergence} and its remark, we see that up to subsequence $|m|^{2p} \mu_{m_k} = |m|^{2p} (\mu _{m_k} ^+ - \mu_{m_k} ^- )$ converges weakly to $|m|^{2p} \mu$. Therefore, $|m|^{2p} (\mu_n - \mu )$ is a weak limit of $|m|^{2p} (\mu_n - \mu_{m_k})$ for some subsequence $\mu_{m_k}$. Therefore, by Lemma \ref{Fatou}, we have
\begin{equation}
\lim \inf_k \left ( |m|^{2p} | \mu_n - \mu_{m_k} | \right ) (\mathbb{R}^2) (x) = \lim \inf_k \bar{M}_{2p} [ \mu_n - \mu_{m_k } ] (x) \ge \bar{M}_{2p} [ \mu_n - \mu ] (x)
\end{equation}
for almost all $x$. Therefore by Fatou's lemma, we have
\begin{equation}
\begin{gathered}
F [\mu_n - \mu] ^e (r) = \sum_{p=0} ^\infty \frac{ \norm{ \bar{M}_{2p} [\mu_n - \mu] }_{L^2}}{(2p)!} r^p  \\
\le \lim \inf_k  \sum_{p=0} ^\infty \frac{ \norm{ \bar{M}_{2p} [\mu_n - \mu_{m_k} ] }_{L^2}}{(2p)!} r^p
\end{gathered}
\end{equation}
which converges to $0$ as $n \rightarrow \infty$. Therefore, $\mu_n \rightarrow \mu$ in $X^r$. 
\end{proof}
Also, we consider approximation to identity by Gaussian in the space $X^{r}$. Let $g_\delta$ be the Gaussian function
\begin{equation}
g_\delta (z) = \frac{1}{2\pi \delta^2 } \exp \left ( -\frac{|z|^2}{2 \delta^2} \right ) 
\end{equation}
with standard deviation $\delta$. We only have weak convergence, but this is enough for our purpose.
\begin{lemma}
Given $\mu_0 \in X^{r}$ with $\mu_0(x; dm)$ nonnegative measures for all $x$, for almost every $x$ $\mu_0 ^\delta (x) = g_{\delta} *_x  (g_{\delta} *_m  \mu_0 )$ converges to $\mu_0 (x) $ weakly. Furthermore,
\begin{equation}
\begin{gathered}
M_{a,b} [\mu_0 ^\delta ] \rightarrow M_{a,b} [\mu_0] \,\,\mathrm{in}\,\, W^{k,2} \,\, \mathrm{if} \,\, M_{a',b'} [\mu_0] \in W^{k,2}\, \mathrm{for} \, \mathrm{all} \, a' + b' \le a+b\\
\bar{M}_p [\mu_0 ^\delta] \rightarrow \bar{M}_p [\mu_0] \,\,\mathrm{in}\,\, L^2 (\, \mathrm{or} \, L^p, 1 \le p < \infty ), \\
\norm{\mu_0 ^\delta}_{X^r} \le C \norm{\mu_0}_{X^r}
\end{gathered} 
\end{equation} \label{approxid}
\end{lemma}
\begin{proof}
We begin with $g_\delta *_m \mu_0 $. We first show that $g_\delta *_m \mu_0 \in X^r$. We have the following basic but frequently used estimate for convolution of moments:
\begin{equation}
\begin{gathered}
M_{a,b} [g_\delta *_m \mu_0] \\
= \sum_{p=0} ^a \sum_{q=0 } ^b \left ( \begin{matrix} a \\ p \end{matrix} \right ) \left ( \begin{matrix} b \\ q\end{matrix} \right ) \int (m_1 - n_1) ^p (m_2 - n_2 )^q g_\delta (m-n) dm n_1 ^{a-p} n_2 ^{b-q} \mu_0 (dn) \\
= \sum_{p=0} ^a \sum_{q=0 } ^b \left ( \begin{matrix} a \\ p \end{matrix} \right ) \left ( \begin{matrix} b \\ q\end{matrix} \right ) M_{p,q} [g_\delta (m) ] M_{a-p, b-q} [\mu_0]
\end{gathered} \label{convolutionestimate1}
\end{equation}
and
\begin{equation}
\begin{gathered}
\bar{M}_k [ g_\delta *_m \mu_0 ] \le \sum_{p=0} ^k \left ( \begin{matrix} k \\ p \end{matrix} \right ) \int |m-n|^p |n|^{k-p} g_\delta (m-n) dm |\mu_0 | (dn) \\
\le \sum_{p=0} ^k \left ( \begin{matrix} k \\ p \end{matrix} \right ) \delta ^p 2^{\frac{p}{2}} \Gamma \left ( \frac{p+2}{2} \right ) \bar{M}_{k-p} [\mu_0].
\end{gathered} \label{convolutionestimate2}
\end{equation}
Therefore, we have
\begin{equation}
F[g_{\delta} *_m \mu_0 ] (r) \le C F[\mu_0] (r)
\end{equation}
where 
\begin{equation}
C = \sum_{p=0 }^\infty \frac{1}{p!}  \Gamma \left ( \frac{p+2}{2} \right ) \left ( \delta 2^{\frac{1}{2}} r \right )^p \le C e^{C (\delta r)^2 }.
\end{equation}
So we have
\begin{equation}
\norm{g_\delta *_m \mu_0 }_{X^r} \le C e^{C(\delta r)^2 } \norm{\mu_0}_{X^r}.
\end{equation}
Also (\ref{convolutionestimate1}) and (\ref{convolutionestimate2}) shows that $M_{a,b} [g_\delta *_m \mu_0] (x)$ and $\bar{M}_k [g_\delta *_m \mu_0 ] (x ) $ are dominated by a $L^2$ function, and $M_{a,b} [ g_\delta *_m \mu_0 ] (x)$ converges to $M_{a,b} [\mu_0]$ in $L^2$ and also almost everywhere, and $\bar{M}_p [g_\delta *_m \mu_0 ]$ converges to $\bar{M}_p [\mu_0]$ in $L^2$ (or other $L^p$, $p< \infty$) and almost everywhere, as $\delta \rightarrow 0$. Therefore, by Theorem \ref{weakconvergence} we note that for almost $x$ $g_\delta *_m \mu_0 (x)$ converges to $\mu_0 (x )$ weakly in a subsequence. Also by (\ref{convolutionestimate1}) if all $M_{a', b'} [\mu_0] \in W^{k,2}$ for $a' + b' \le a+b$ then $M_{a,b} [g_\delta *_m \mu_0 ] \rightarrow M_{a,b} [\mu_0]$ in $W^{k,2}$. Since $\mu_0 \ge 0$, we have $M_{a,b} [\mu_0 ^\delta] = g_{\delta} *_x M_{a,b} [g_{\delta} *_m \mu_0 ]$ and $\bar{M}_k [\mu_0 ^\delta] = g_\delta *_x \bar{M}_k [g_\delta *_m \mu_0 ]$. Since convolution with $g_\delta$ is an approximate identity, all the conclusions of the lemma holds.
\end{proof}
Also we can prove the following:
\begin{lemma}
Let $\mu \in X^r$ is given by a smooth density $\mu = \mu(x, m) dm$. If $\bar{M}_{p} [\nabla_x ^k \mu] \in L^2$ for some nonnegative integer $p$, then for all $a, b \ge 0$ with $a+b = p$ we have $\nabla_x ^k M_{a,b} [\mu ] = M_{a,b} [\nabla_x ^k \mu ] \in L^2 (\mathbb{R}^2)$ and 
\begin{equation}
\norm{\nabla_x ^k M_{a,b} [\mu]}_{L^2 } \le \norm{\bar{M}_{p} [\nabla_x ^k \mu ]}_{L^2 }.
\end{equation}
Especially, if $\mu \in X^{k,r}$ then $M_{a,b} [\mu] \in W^{k,2}$ for all $a, b \ge 0$. Also, if $\mu(t) \in C^1 ([0, T], X^r )$ is a continuously differentiable family, and $\mu(t) = \mu(x, m, t) dm$ is given by the smooth density functions, then $\partial_t M_{a,b} [\mu] (t) = M_{a,b} [\partial_t \mu ] \in L^2$. \label{momentSobolev}
\end{lemma}
\begin{proof}
We prove only the first assertion; the second assertion can be proven in the same way. First, note that $\left | M_{a,b} [\nabla_x ^k  \mu ] \right | \le \bar{M}_{a+b} [\nabla_x ^k \mu ]$, so $M_{a,b} [\nabla_x ^k \mu] \in L^2$ for $a+b = p$. Then we have
\begin{equation}
\begin{gathered}
\partial_{x_i} M_{a,b} [\mu] (x) - M_{a,b} [\partial_{x_i} \mu ] (x) = \lim_{h \rightarrow 0 } \int_{\mathbb{R}^2 _m} m_1 ^a m_2 ^b \int_0 ^1 \partial_{x_i} \mu (x+hse_i, m) ds dm - M_{a,b} [\partial_{x_i} \mu] (x)  \\
= \lim_{h \rightarrow 0} \int_0 ^1 \left (M_{a,b} [\partial_{x_i} \mu] (x+ hse_i) - M_{a,b} [\partial_{x_i} \mu ] (x) \right ) ds
\end{gathered}
\end{equation}
by Taylor expansion and Fubini's theorem. On the other hand, since translation in space is continuous in $L^2(\mathbb{R}^2, dx)$ we have
\begin{equation}
\lim_{h \rightarrow 0} \norm{\int_0 ^1 \left (M_{a,b} [\partial_{x_i} \mu] (x+ hse_i) - M_{a,b} [\partial_{x_i} \mu ] (x) \right ) ds}_{L^2 (\mathbb{R}^2, dx ) } = 0
\end{equation}
and by Fatou we are done.
\end{proof}
\begin{remark}
We conclude this section with the remark showing that the growth of moments condition is a mild constraint to polymer distributions. We consider the following probability distribution $$ f(m, x) = \exp \left ( - \left (\frac{|m|^2}{c(x)} \right )^q \right ),$$ where $c(x)>0 $ is a parameter representing the degree of stretch of polymer at position $x$. For example, when $M_{0,0} = 1$ and $q = 1$, this corresponds to the case $\sigma = 2 c(x) \mathbb{I}$. Suppose that $c \in W^{1,2} (\mathbb{R}^2 )$. We can show that for some $0 < r < C {\norm{\nabla_x c}_{L^2} ^{ - \frac{1}{2} } } $, $f \in X^r$. First, by a direct calculation we obtain $$ \bar{M}_{2r} [f] (x) = 2 \pi \Gamma \left ( \frac{r+1}{q} \right ) |c(x) |^{r+1}, $$ and by Gagliardo-Nirenberg inequality (\cite{MR2759829}) we have
\begin{equation}
\norm{c}_{L^{2(r+1)}} ^{r+1} \le (r+1) ! \norm{\nabla_x c}_{L^2} ^{r+1}.
\end{equation}
Therefore,
\begin{equation}
\frac{\norm{\bar{M}_{2r} [f] }_{L^2} }{(2r)!} z^{2r} \le 2 \pi \norm{\nabla_x c}_{L^2} \frac{\Gamma \left ( \frac{r+1}{q} \right ) (r+1) ! }{(2r)!} \left ( \norm{\nabla_x c}_{L^2} ^{\frac{1}{2} } z \right )^{2r},
\end{equation}
as desired.
\end{remark}
\begin{remark}
Another example is the following:
\begin{equation}
f(m,x) = c(x) \frac{1}{Z} e^{-|m|}
\end{equation}
where $\int_{\mathbb{R}^2 _m } \frac{1}{Z} e^{-|m|} dm = 1$ and $c(x) \ge 0 \in L^1 \cap L^2$. Then for each $k$
\begin{equation}
\bar{M}_k [f] (x) = \int_{\mathbb{R}^2 _m } c(x) |m|^k \frac{1}{Z} e^{-|m|} dm = c(x) \frac{2\pi}{Z} \int_0 ^\infty r^{k+1} e^{-r} dr = \frac{2 \pi (k+1)!}{Z} c(x)
\end{equation}
and therefore we have
\begin{equation}
\norm{f}_{X^r} = \sum_{k=0} ^\infty \frac{2 \pi (k+1) \norm{c}_{L^2} }{Z } r^k < \infty 
\end{equation}
for $0 < r < 1$.
\end{remark}

\subsection{Moment solution and its properties} \label{Momsolprop}
Here we define the notion of moment solution and investigate its properties.
\begin{definition}[Moment solution]
Let $\mu = \mu(x, t; dm) \in L_{loc} ^1 ( [0, T] \times \mathbb{R}^2, \mathcal{M} (\mathbb{R}^2) )$ and $u \in L^\infty (0, T; L^\infty)$ with $\nabla_x u \in L^2 (0, T; L^\infty )$ be a given divergence free field. We say $\mu$ is a moment solution of the Fokker-Planck equation with velocity field $u$ if the following holds:
\begin{enumerate}
\item $\mu$ is a solution to the Cauchy problem 
\begin{equation}
\partial_t \mu =   \epsilon \Delta_m \mu + \nu_2 \Delta_x \mu - \nabla_x \cdot \left ( u (t) \mu \right ) - \nabla_m \cdot \left ( \left ( \nabla_x u (t) m - \nabla_m U \right ) \mu \right ) \label{FPwithu}
\end{equation}
with $\mu|_{t=0} = \mu_0$, 
\item $\mu = \mu(x, t; dm) dx dt$ is nonnegative measures for almost all $x, t$, and for almost all $t\in (0, T)$ 
\begin{equation}
\int_{\mathbb{R}^2_x} \int_{\mathbb{R}^2_m} \mu(x, t;dm) dx  \le \int_{\mathbb{R}^2_x} \int_{\mathbb{R}^2_m} \mu_0 (x;dm) dx
\end{equation}
and
\begin{equation}
\int_{\mathbb{R}^2_m} \mathrm{Tr} \left ( m \otimes \nabla_m U \right ) \mu(x, t; dm) \in  L^\infty (0, T; L^1_x )
\end{equation}
holds;
\item There is a nonincreasing, positive function $r : [0, T] \rightarrow \mathbb{R}_+$, $r(0) = r<\infty $ such that
\begin{equation}
\norm{\mu (t)}_{X^{r(t)} } \le C(r, T, \norm{u}) \norm{\mu (0) }_{X^r}
\end{equation}
holds; and
\item For all $a, b \ge 0$ $M_{a,b} [\mu] (x, t) \in L^2 (0, T; W^{1,2} )$ and $\partial_t M_{a,b} [\mu] (x, t) \in L^2 (0, T; W^{-1, 2} )$.
\end{enumerate}
\end{definition}
We see that for all $a, b \ge 0$ we have in fact $M_{a,b} \in C([0, T] ; L^2 )$. Also we see that moments of moment solutions are weak solutions for formal evolution equation of moments:
\begin{lemma}
Let $\mu \in L_{loc} ^1 ([0, T] \times \mathbb{R}^2, \mathcal{M} (\mathbb{R}^2 ) )$ be a moment solution of the Fokker-Planck equation. Then for all $a, b \ge 0$, $M_{a,b} [\mu] = M_{a,b}$ are weak solutions of the evolution equation
\begin{equation}
\begin{gathered}
\partial_t M_{a,b} + u \cdot \nabla_x M_{a,b} - \nu_2 \Delta_x M_{a,b} = - 2 q \epsilon (a+b) M_{a,b} [|m|^{2(q-1)} \mu] \\
+ \epsilon \left ( a (a-1) M_{a-2,b} + b(b-1) M_{a,b-2} \right ) + a \partial_1 u_1 M_{a,b} + a \partial_1 u_2 M_{a-1, b+1} + b\partial_2 u_1 M_{a+1, b-1} + b \partial_2 u_2 M_{a,b},
\end{gathered} \label{momentevolution}
\end{equation}
that is, for any $\Phi \in L^2 (0, T; W^{1,2} )$ with $\Phi (T) = 0$ with $\partial_t \Phi \in L^2 (0, T; W^{-1,2} )$, we have
\begin{equation}
\begin{gathered}
\int_0 ^T \langle \partial_t M_{a,b}, \Phi \rangle_{W^{-1,2}, W^{1,2}} dt + \int_0 ^T \langle u \cdot \nabla_x M_{a,b} , \Phi \rangle_{L^2, L^2 } + \nu_2 \int_0 ^T \langle \nabla_x M_{a,b}, \nabla_x \Phi \rangle_{L^2, L^2 } = \int_0 ^T \langle R, \Phi \rangle_{L^2, L^2}
\end{gathered} \label{energymoments}
\end{equation} 
where $R$ is all the terms in the right side of (\ref{momentevolution}).
\label{momentofmomentsols}
\end{lemma}
\begin{proof}
In (\ref{weaksolution}), put our test functions in the form of
\begin{equation}
\phi = \phi_1 (x, t) m_1 ^a m_2 ^b \psi_\alpha (m)
\end{equation}
where $\psi$ is a smooth cutoff and $\psi_\alpha (m) = \psi (\frac{m}{\alpha})$. Then we apply dominated convergence, and then we apply integration by parts to $\partial_t \phi_1 M_{a,b}$ term and $\nu_2 \Delta_x \phi_1 M_{a,b}$ term. Then by density we are done. 
\end{proof}
Also moment solution is unique, given initial data.
\begin{lemma}
Suppose $\mu_1$ and $\mu_2$ are two moment solutions with same initial data. Then $\mu_1 = \mu_2$ in $L_{loc} ^1 ([0, T] \times \mathbb{R}^2 , \mathcal{M} (\mathbb{R}^2))$.
\label{uniqueness}
\end{lemma}
\begin{proof}
This is a immediate consequence of Theorem \ref{FPKeu}. By definition, $\mu$ is a solution to the Cauchy problem of (\ref{FPwithu}). Then we have
\begin{equation}
\begin{gathered}
u \in L^1 (0, T; \mu(x, t ; dm) dx dt ), \,\,  \nabla_m U \otimes m \in L^1 (0, T; \mu(x, t; dm) dx dt), \\
\left | \nabla_x u (t) m \right | \le |\nabla_x u (t) |^2 + 1 + C |m|^{2q} \in L^1 (0, T; \mu(x, t; dm) dxdt ).
\end{gathered}
\end{equation}
Condition (\ref{SP}) is obvious.
\end{proof}

\section{Solution scheme for Fokker-Planck equation} \label{Solsch}
The purpose of this section is to prove the following theorem.
\begin{theorem}
Given a fluid velocity field $u$ and initial data $ \mu_0 $ satisfying (\ref{velinitcond}), (\ref{initpositivity}), (\ref{initmoment}), (\ref{initstress}), and (\ref{initentropy}), there exists a unique moment solution to the Fokker-Planck equation (\ref{FPwithu}). Furthermore, it is given by nonnegative densities $\mu(x, t; dm) = f(x, t, m)$ and moments $M_{a,b} = M_{a,b} [\mu]$ satisfy bounds (\ref{XrtypeM}), (\ref{L2H1M}), (\ref{L2H1M2}),  (\ref{L2H2qM}), and (\ref{LinfL1M}). Furthermore, if the fluid velocity fields $u$ and $v$ satisfy (\ref{velinitcond}) and if we let $f$ and $g$ be solutions to the Fokker-Planck equation (\ref{FPwithu}) with velocity field $u$ and $v$, respectively, and if we let $\sigma_1$ and $\sigma_2$ be corresponding stress fields for $f$ and $g$ respectively, then they satisfy the estimate (\ref{differenceestimate}).
\label{momentsolutionexists}
\end{theorem}

\subsection{Approximate solutions} \label{Appsol}
Our goal is to find a moment solution for Fokker-Planck equation, given a fluid velocity field $u$. We establish such solution by setting up an approximation scheme. There are two main modifications in the sequence of approximate solutions: the first is to introduce smooth cutoff to the drift and potential, so that the coefficients remain finite. This modification enables us to employ integration by parts in $m$ variable rigorously, and we can investigate of the bounds on moments. The second is to mollify velocity field and initial data to guarantee higher regularities. Let $\Psi$ be a smooth, decreasing compactly supported function in the closed half-line $\{ r \ge 0 \}$, $0 \le \Psi \le 1$, with $\Psi \equiv 1$ for $r \le 1$ and $\Psi \equiv 0$ for $r \ge 2$. Then for $\alpha > 0$, we let $\psi_\alpha (m) = \Psi \left ( \frac{|m|}{\alpha} \right )$.
\begin{definition}
Suppose that 
\begin{equation}
\begin{gathered}
u \in L^\infty (0, T; \mathbb{P} W^{2,2} ) \cap L^2 (0, T; \mathbb{P}W ^{3,2} ), \\
\partial_t u \in L^\infty (0, T; \nabla_x L^1 + L^2 ) \cap L^2 (0, T; \mathbb{P} W^{1,2} ), 
\end{gathered} \label{velinitcond}
\end{equation}
\begin{equation}
\mu_0 \ge 0, \int \int \mu_0 (dm) dx = 1 \label{initpositivity}
\end{equation}
\begin{equation}
\mu_0 \in X^r,  \label{initmoment}
\end{equation}
\begin{equation}
M_{a,b} [\mu_0 ] \in W^{1,2} \,\, \mathrm{for} \,\, a+b = 2p \le 8q-2, \bar{M}_{4q} [\mu_0] \in L^1, \label{initstress}
\end{equation}
\begin{equation}
\begin{gathered}
\mu_0 = f_0 (x, m) dm dx, \\
\int_{\mathbb{R}^2_m \times \mathbb{R}^2_x } f_0 \log f_0 dmdx \in \mathbb{R}, \\
\int_{\mathbb{R}^2 _x } |\Lambda (x) |^2 M_{0,0} [f_0 ] (x) dx < \infty, \Lambda (x) = \log \left ( \max ( |x|, 1 ) \right ), \label{initentropy}
\end{gathered}
\end{equation}
be given. For $\alpha > 0$, a $\alpha$-truncated Fokker-Planck solution of the Cauchy problem of (\ref{FPwithu}) with $\mu|_{t=0} = \mu_0 $ is a function $f^{\alpha} \in C^1 ([0, T] ; W^{k,2}_x W^{k,2}_m \cap X^{k,r})$, $k = 20$, satisfying
\begin{equation}
\begin{gathered}
\partial_t f^{\alpha} + u^\alpha \cdot \nabla_x f^\alpha + \nabla_m \cdot \left ( (\nabla_x u^\alpha) m \psi_\alpha f^\alpha \right ) = \epsilon \Delta_m f^\alpha + \epsilon \nabla_m \cdot \left ( (\nabla_m U) \psi_\alpha f^\alpha \right ) + \nu_2 \Delta_x f^\alpha, \\
f^\alpha (x,m,0) = \mu_0 ^{\frac{1}{\alpha}} (x, m) =: f_0 ^\alpha
\end{gathered} \label{truncatedFP}
\end{equation}
pointwise where $u^\alpha = g_{\frac{1}{\alpha}} *_x u$.
\end{definition}
We first start with existence and uniqueness of such $\alpha$-truncated Fokker-Planck solution. First note that
\begin{equation}
f_0 ^\alpha \in W^{p,2}_x W^{p,2}_m \cap X^{p,r} \cap W_{x} ^{p,1} W_{m,1} ^{p,1}, \,\, M_{0,0} [f^\alpha_0], \bar{M}_{2q} [f_0 ^\alpha] , \bar{M}_{4q} [f_0 ^\alpha ] \in W_x ^{p,1}
\end{equation}
for any $p \ge 0$: this is because $\nabla_x ^a \nabla_m ^b f_0 ^\alpha = (\nabla_x ^a g_{\frac{1}{\alpha}}) *_x (\nabla_m ^b g_{\frac{1}{\alpha}} ) *_m \mu_0$ so we can apply the same argument in Lemma \ref{approxid} to conclude that $\nabla_x ^a \nabla_m ^b f_0 ^\alpha \in X^r$, and using Young's inequality for measure 
\begin{equation}
\norm{h *_m \mu_0 }_{L^2 (\mathbb{R}^2; dm)} \le \norm{h}_{L^2} \norm{\mu_0}
\end{equation}
we see that $\nabla_x ^a \nabla_m ^b f_0 ^\alpha \in L^2 _x L^2 _m$. Also note that for all $p\ge 0$
\begin{equation}
\begin{gathered}
u^\alpha \in L^\infty (0, T; \mathbb{P} W^{p, 2} ), \partial_t u^\alpha \in L^\infty (0, T; \mathbb{P} W^{p,2} ), \\
u^\alpha \rightarrow u \,\mathrm{in}\,\, L^\infty (0, T; \mathbb{P} W^{2,2} ) \cap L^2 (0, T; \mathbb{P} W^{3,2} ).
\end{gathered}
\end{equation}
The equation (\ref{truncatedFP}) has a solution map:
\begin{equation}
\begin{gathered}
f^{\alpha} (t) = e^{t ( \epsilon \Delta_m + \nu_2 \Delta_x )} f_0 ^\alpha - \int_0 ^t \nabla_x (e^{\tau (\epsilon \Delta_m + \nu_2 x ) } )\cdot  u^\alpha (t-\tau) f^{\alpha} (t-\tau) d\tau \\
- \int_0 ^t \nabla_m (e^{\tau (\epsilon \Delta_m + \nu_2 x ) } ) \cdot (\nabla_x u^\alpha (t-\tau) m \psi_\alpha f^\alpha (t-\tau) ) d\tau + \int_0 ^t \epsilon \nabla_m (e^{\tau (\epsilon \Delta_m + \nu_2 x ) } ) \cdot ( f^\alpha (t-\tau) \nabla_m U \psi_\alpha ) d\tau.
\end{gathered} \label{truncatedsolution}
\end{equation}
Then we have
\begin{equation}
\begin{gathered}
\nabla_x ^p \nabla_m ^q f^{\alpha} (t) = e^{t (\epsilon \Delta_m + \nu_2 x ) } \nabla_x ^p \nabla_m ^q f_0 ^\alpha \\
- \int_0 ^t \nabla_x (e^{\tau (\epsilon \Delta_m + \nu_2 x ) } ) \sum_{p'} \left ( \begin{matrix} p \\ p' \end{matrix} \right ) \nabla_x^{p'}  u^\alpha (t-\tau) \nabla_x ^{p-p'} \nabla_m ^q f^{\alpha} (t-\tau) d\tau \\
- \int_0 ^t \nabla_m (e^{\tau (\epsilon \Delta_m + \nu_2 x ) } ) \cdot \sum_{p', q'} \left ( \begin{matrix} p \\ p' \end{matrix} \right ) \left ( \begin{matrix} q \\ q' \end{matrix} \right )  \nabla_x \nabla_x ^{p'}  u^{\alpha} (t-\tau) \nabla_m ^{q'} (m \psi_\alpha ) \nabla_x ^{p-p'} \nabla_m ^{q - q'} f^{\alpha} (t-\tau) d\tau \\
+ \int_0 ^t \nabla_m (e^{\tau (\epsilon \Delta_m + \nu_2 x ) } ) \cdot \sum_{q'}  \left ( \begin{matrix} q \\ q' \end{matrix} \right ) \nabla_m ^{q'} (\nabla_m U \psi_\alpha ) \nabla_x ^p \nabla_m ^{q - q'} f^{\alpha} (t-\tau) d\tau
\end{gathered} \label{truncatedsolution2}
\end{equation}
and
\begin{equation}
\begin{gathered}
\partial_t \nabla_x ^p \nabla_m ^q f^{\alpha} (t) = e^{t (\epsilon \Delta_m + \nu_2 x ) }  ( \epsilon \Delta_m + \nu_2 \Delta_x ) \nabla_x ^p \nabla_m ^q f_0 ^{\alpha} - e^{t (\epsilon \Delta_m + \nu_2 x ) } \nabla_x \cdot \nabla_x ^p (u^{\alpha} (0) \nabla_m ^q f_0 ^\alpha ) \\
- e^{t (\epsilon \Delta_m + \nu_2 x ) }  \nabla_x ^p (\nabla_x u^\alpha (0) \nabla_m \nabla_m ^q (m \psi_\alpha f_0 ^\alpha ) ) + \epsilon e^{t (\epsilon \Delta_m + \nu_2 x ) } \cdot \nabla_m \nabla_m ^q (\nabla_x ^p f_0 ^\alpha \nabla_m U \psi_\alpha ) \\
-\int_0 ^t \nabla_x ( e^{\tau (\epsilon \Delta_m + \nu_2 x ) } ) \cdot \nabla_x ^p (\partial_t u^\alpha \nabla_m ^q f^{\alpha} + u^\alpha \partial_t \nabla_m ^q f^{\alpha} )(t-\tau) d\tau \\
- \int_0 ^t \nabla_m (e^{\tau (\epsilon \Delta_m + \nu_2 x ) } ) \cdot \nabla_x ^p ( \partial_t \nabla_x u^\alpha (t-\tau) \nabla_m ^q (m \psi_\alpha f^\alpha (t-\tau) )  + \nabla_x u^\alpha  (t-\tau) \nabla_m ^q (m \psi_\alpha \partial_t f^\alpha (t-\tau)  ) ) d\tau \\
 + \epsilon \int_0 ^t \nabla_m (e^{\tau (\epsilon \Delta_m + \nu_2 x ) } ) \cdot \nabla_m ^q (\partial_t \nabla_x ^p f^{\alpha} (t-\tau) \nabla_m U \psi_\alpha ) d\tau.
\end{gathered}\label{truncatedsolution3}
\end{equation}
From this we conclude that the solution map (\ref{truncatedsolution}) is a contraction mapping in the complete metric space 
\begin{equation}
\{ f \in C^1 ([0, T], W_x ^{k,2} W_m ^{k,2} \cap X^{k,r}) \, : \, f(0) = f_0 ^\alpha \}
\end{equation}
since all the terms in (\ref{truncatedsolution}), (\ref{truncatedsolution2}), (\ref{truncatedsolution3}) are either of the form
\begin{equation}
e^{t (\epsilon \Delta_m + \nu_2 x ) } A \nabla_x ^{p'} \nabla_m ^{q'} f_0 ^\alpha
\end{equation}
where $A$ is $1$ or $\nabla_x ^{r'} u(0)$ and $p'$, $q'$ are derivatives higher than at most 2 degrees to the left hand side term that it occurs, or
\begin{equation}
\int_0 ^t \nabla_{x,m} ( e^{\tau (\epsilon \Delta_m + \nu_2 x ) } ) A \partial_t ^{r'} \partial_x ^{p'} \partial_m ^{q'} f^\alpha (t-\tau) d\tau
\end{equation} 
where $A$ is of the form of some constant, $\nabla_m ^{l'} (\nabla_m U \psi_\alpha )$, or $\nabla_x^{k'} \partial_t ^{i'} u^\alpha (t-\tau)$, and $\partial_t ^{r'} \partial_x ^{p'} \partial_m ^{q'} f^\alpha$ are terms with derivatives lower than or equal to the left hand side term that it occurs. The terms we denoted by $A$ are innocent, because $\norm{A}_{L^\infty (0, T; L^\infty _x ) } \le C(\alpha) <\infty$. Therefore, the $W^{k,2}_{x,m} \cap X^{k,r} $ norm of first term can be bounded by $C \norm{f_0 ^\alpha}_{W^{k+2,2}_{x,m} \cap X^{k+2,r} }$, which is finite, and the $W^{k,2}_{x,m} \cap X^{k,r} $ of the second term can be bounded by
\begin{equation}
C \int_0 ^t \frac{1}{\tau^{\frac{1}{2}}} \norm{f^\alpha}_{C([0, T]; W^{k,2}_{x,m} \cap X^{k,r} ) } d \tau = C \tau^{\frac{1}{2}} \norm{f^\alpha}_{C([0, T]; W^{k,2}_{x,m} \cap X^{k,r} )}.
\end{equation}
Furthermore, the left hand side is continuous in time since each term is either heat semigroup of some function or time integral of $L^1(0, T; W^{k,2}_{x,m} \cap X^{k,r}) $ functions. Therefore, by contraction mapping principle, there is unique function $f^\alpha \in C^1 ([0, T] ; W^{k,2}_x W^{k,2}_m \cap X^{k,r} )$ satisfying (\ref{truncatedsolution}). One consequence is that $f^\alpha$ is a classical solution of (\ref{truncatedFP}). That is, by Sobolev embedding $f^\alpha \in C^1 ([0, T] ; C^2(x,m))$ and satisfies (\ref{truncatedFP}) pointwise. Therefore, in view of the maximum principle, we have $f ^\alpha \ge 0$ for all $(x, m, t)$. Then same argument as above and $f^\alpha \ge 0$ show that $M_{0,0} [f^\alpha], \bar{M}_{2q} [f^\alpha], \bar{M}_{4q} [f^\alpha] \in C^1 \left ( [0, T], W^{k,1}_x \right ) $. 

\subsection{Uniform bounds on moments} \label{unifBdmom}
In this section, we investigate bounds on moments for approximate solutions, which is uniform in $\alpha$. By Lemma \ref{momentSobolev}, we conclude that 
\begin{equation}
M_{a,b} ^\alpha = M_{a,b} [f^\alpha] \in Lip(0, T; W^{2,2}),
\end{equation}
and we saw $M_{0,0} [f^\alpha], \bar{M}_{2q} [f^\alpha], \bar{M}_{4q} [f^\alpha] \in C^1 \left ( [0, T], W^{k,1}_x \right ) $. Also, since $\nabla_m f^\alpha \in X^r$, by integration by parts we wee that
\begin{equation}
\int_{\mathbb{R}^2} m_1 ^a m_2 ^b \nabla_m (m \psi_\alpha f^\alpha ) dm (x, t) = - \int_{\mathbb{R}^2} \nabla_m (m_1 ^a m_2 ^b) m \psi_\alpha f^\alpha dm \in L^\infty (0, T; L^2)
\end{equation}
and similar identity holds for $\epsilon \nabla_m \cdot (f^\alpha \nabla_m U \psi_\alpha )$ term. Therefore, we see that the following equation holds for all $a, b\ge 0$ and almost every $(x, t)$: 
\begin{equation}
\begin{gathered}
\partial_t M_{a,b} ^\alpha + u^\alpha \cdot \nabla_x M_{a,b}^\alpha - \nu_2 \Delta_x M_{a,b} ^\alpha  + 2q \epsilon (a+b) \int_{\mathbb{R}^2}m_1 ^a m_2 ^b |m|^{2(q-1)} \psi_\alpha f^\alpha dm \\
= \epsilon \left ( a(a-1) M_{a-2,b} ^\alpha + b(b-1) M_{a,b-2} ^\alpha \right ) \\ + a \partial_1 u_1 ^\alpha \int_{\mathbb{R}^2} m_1 ^a m_2 ^b \psi_\alpha f^\alpha dm + a \partial_1 u_2 ^\alpha  \int_{\mathbb{R}^2} m_1 ^{a-1} m_2 ^{b+1}  \psi_\alpha f^\alpha dm \\
+ b \partial_2 u_1 ^\alpha \int_{\mathbb{R}^2} m_1 ^{a+1} m_2 ^{b-1} \psi_\alpha f^\alpha dm + b \partial_2 u_2 ^\alpha \int_{\mathbb{R}^2} m_1 ^a m_2 ^b \psi_\alpha f^\alpha dm
\end{gathered} \label{alphamoments}
\end{equation}
and all the terms are in $L^\infty (0, T; L_x ^2)$. Especially, for 
\begin{equation}
\bar{M}_{2k} ^\alpha = \bar{M}_{2k} [f^\alpha]
\end{equation}
we have the following:
\begin{equation}
\begin{gathered}
\partial_t \bar{M}_{2k} ^\alpha + u^\alpha \cdot \nabla_x \bar{M}_{2k} ^\alpha - \nu_2 \Delta_x \bar{M}_{2k} ^\alpha + (2q) \epsilon (2k) \int_{\mathbb{R}^2} |m|^{2(k+q-1)} \psi_\alpha f^\alpha dm \\
= \epsilon (2k)^2 \bar{M}_{2(k-1) } ^\alpha +  \mathrm{Tr} \left ( (\nabla_x  u ^\alpha ) (2k) \int_{\mathbb{R}^2} |m|^{2(k-1) } m \otimes m \psi_\alpha f^\alpha dm \right ).
\end{gathered} \label{alpharadmoments}
\end{equation}
From (\ref{alphamoments}) and (\ref{alpharadmoments}) we derive four estimates independent of $\alpha$: the first one is a set of $L^2$ estimates for all even moments, which gives us an $X^r$ estimate for the limiting object. The second one is a set of $L^\infty (0. T; L^2) \cap L^2 (0, T; W^{1,2} )$ bounds for all moments. The third one is a set of $L^\infty (0, T; W^{1,2}) \cap L^2 (0, T; W^{2,2} )$ estimates for even moments up to degree $2q$, which enables us to establish regularity for the stress field $\sigma$. Finally we obtain a $L^p$ estimate, $1 \le p \le 2$ for $\bar{M}_{4q}$ and $M_{0,0}$, which gives us a $L^1$ bound for $\sigma$. Then we use them to bound $\partial_t M_{a,b} ^\alpha$ uniformly in $\alpha$, in the space $L^2 (0, T; W^{-1, 2})$. \newline 
To obtain first three bounds, we need to deal with the terms coming from restoring force $\nabla_m \cdot (\nabla_m U \psi_\alpha f^\alpha)$ because it contains higher moments. However, they are harmless in $L^2$ norm due to the following simple observation:
\begin{lemma}
Let $\mu_1 (dm) , \mu_2 (dm)$ be nonnegative measures and $p$ be a nonnegative integer. Then
\begin{equation}
\sum_{a, b \ge 0, a+b = 2p} M_{a,b} [\mu_1] M_{a,b} [\mu_2] \ge 0.
\end{equation}
\label{restoringpositive}
\end{lemma}
\begin{proof}
This follows from Cauchy-Schwarz inequality: if $a, b$ are odd, then 
\begin{equation}
\left | M_{a,b} [\mu_1] \right | \le \sqrt{M_{a+1, b-1} [\mu_1] } \sqrt{M_{a-1, b+1} [\mu_1] }
\end{equation}
and same for $M_{a,b} [\mu_2]$. Then the left side of the claimed inequality is bounded below by sum of perfect squares
\begin{equation}
\sum_{a'= 0} ^ {p-1}  \left ( \sqrt{M_{2(a'+1), 2(p - a' - 1)} [\mu_1 ] M_{2(a'+1), 2(p - a' - 1)} [\mu_2 ]} - \sqrt{M_{2a', 2(p-a')} [\mu_1 ] M_{2a', 2(p-a')} [\mu_2 ] } \right )^2 \ge 0.
\end{equation}
\end{proof}
\paragraph{$L^2$ bounds.} By multiplying $\bar{M}_{2k} ^\alpha$ to (\ref{alpharadmoments}) and integrating, and applying integration by parts to spatial derivatives for $ \nu_2 \Delta_x \bar{M}_{2k} ^\alpha$ term (which is rigorous since $ \nabla_x ^p \bar{M}_{2k} \in L^2$ for $p \le 2$) and $\bar{M}_{2k} ^\alpha \partial_t \bar{M}_{2k}^\alpha = \frac{1}{2} \partial_t \left ( \bar{M}_{2k} ^{\alpha} \right ) ^2$ (which is also rigorous since $\left (\bar{M}_{2k} ^\alpha \right )^2 \in C^1 ([0, T] ; L^1)$ ), and applying Lemma \ref{restoringpositive} as $\mu_1 = |m|^{2(k+q-1)} \psi_\alpha f^\alpha$ and $\mu_2 = |m|^{2k}  f^\alpha $ with $p=0$, and bounding $\psi_\alpha f^\alpha$ by $f^\alpha$ and $m \otimes m \psi_\alpha f^\alpha$ by $|m|^2 f^\alpha$ we have
\begin{equation}
\begin{gathered}
\frac{1}{2} \frac{d}{dt} \norm{\bar{M}_{2k} ^\alpha }_{L^2} ^2 + \nu_2 \norm{\nabla_X \bar{M}_{2k} ^\alpha}_{L^2} ^2 \le \epsilon (2k)^2 \norm{\bar{M}_{2(k-1)}^\alpha}_{L^2} \norm{\bar{M}_{2k} ^\alpha}_{L^2} + 2k  \norm{\nabla_x u (t) }_{L^\infty } \norm{\bar{M}_{2k} ^\alpha}_{L^2} ^2
\end{gathered} 
\end{equation}
where Young's inequality $\norm{\nabla_x u^\alpha (t) }_{L^\infty} \le \norm{g_\alpha}_{L^1} \norm{\nabla_x u(t) }_{L^\infty}$ is used. Dividing this by $(2k)! \norm{\bar{M}_{2k} ^\alpha}_{L^2}$, multiplying $z^{2k}$ and summing those up for all $k\ge0$ we get
\begin{equation}
\begin{gathered}
\frac{d}{dt} \sum_{k=0} ^\infty \frac{\norm{\bar{M}_{2k}  ^\alpha (t) }_{L^2} }{(2k)!} z^{2k} \le 2 \epsilon \sum_{k=1} ^\infty \frac{\norm{\bar{M}_{2(k-1)}  ^\alpha (t) }_{L^2} }{(2(k-1))!} z^{2k} + 2 k \norm{\nabla_x u (t) }_{L^\infty } \sum_{k=0} ^\infty \frac{\norm{\bar{M}_{2k}  ^\alpha (t) }_{L^2} }{(2k)!} z^{2k}.
\end{gathered}
\end{equation}
Introducing
\begin{equation}
F_e ^\alpha (t; z) = \sum_{k=0} ^\infty \frac{\norm{\bar{M}_{2k}  ^\alpha (t) }_{L^2} }{(2k)!} z^{2k}
\end{equation}
we get
\begin{equation}
\frac{d}{dt} F_e ^\alpha (t; z) \le 2 \epsilon z^2 F_e ^\alpha (t; z) + \norm{\nabla_x u (t) }_{L^\infty} z \frac{d}{dz} F_e ^\alpha (t; z).
\end{equation}
Therefore, we have
\begin{equation}
F_e ^\alpha (t; z) \le F_e ^\alpha (0; z e^{\int_0 ^t \norm{\nabla_x u (\tau) }_{L^\infty} d\tau } ) \exp \left ( 2\epsilon \int_0 ^t z^2 e^{2 \int_s ^t \norm{\nabla_x u(\tau)}_{L^\infty} d\tau } ds \right ),
\end{equation}
in other words,
\begin{equation}
\norm{f^\alpha (t) }_{X^{\frac{r}{\int_0 ^t \norm{\nabla_x u}_{L^\infty} d\tau }}} \le e^{2\epsilon T r^2 }  \norm{f_0 ^\alpha}_{X^r} \le C(r, T) \norm{\mu_0}_{X^r} \label{Xrtype}
\end{equation}
where the last inequality comes from Lemma \ref{approxid}. We also establish $L^\infty (0, T; L^2) \cap L^2 (0, T; W^{1,2} )$ estimates for all moments. For $a, b\ge 0$ with $a+b =2k \le 2p$, we multiply $M_{a,b} ^\alpha$ to each of (\ref{alphamoments}), sum over all such $a, b$, and integrate in $x$. Again we bound truncated terms $\psi_\alpha$ by $1$ and if $m_1 ^a m_2 ^b$ by $|m|^{a+b}$. Then we get
\begin{equation}
\begin{gathered}
\frac{1}{2} \frac{d}{dt} \norm{\dot{\vec{M}}_{2k} ^\alpha }_{L^2} ^2 + \nu_2 \norm { \nabla_x \dot{\vec{M}}_{2k} ^\alpha }_{L^2} ^2 \le C \epsilon(2k)^2 \norm{\dot{\vec{M}}_{2(k-1)} ^\alpha}_{L^2} \norm{\dot{\vec{M}}_{2k} ^\alpha }_{L^2} + C k \norm{\nabla_x u(t) }_{L^2} \norm{ \dot{\vec{M}}_{2k} ^\alpha }_{L^4} ^2
\end{gathered}
\end{equation}
where
\begin{equation}
\dot{\vec{M}}_{2k} ^\alpha = \left ( M_{2k, 0} ^\alpha, M_{2k-1, 1 } ^\alpha, \cdots, M_{0, 2k} ^\alpha \right ).
\end{equation}
Again we used Lemma \ref{restoringpositive} with $\mu_1 = |m|^{2(q-1)} \psi_\alpha f^\alpha$ and $\mu_2 = f^\alpha$. Then by Ladyzhenskaya's inequality,
\begin{equation}
\norm{\dot{\vec{M}}_{2k} ^\alpha }_{L^4} ^2 \le C \norm{\dot{\vec{M}}_{2k} ^\alpha }_{L^2} \norm{ \nabla_x \dot{\vec{M}}_{2k} ^\alpha }_{L^2},
\end{equation}
and we have the following by summing over all $k \le p$:
\begin{equation}
\begin{gathered}
\frac{d}{dt} \sum_{k=0} ^p \norm{\dot{\vec{M}}_{2k} ^\alpha }_{L^2} ^2 + \nu_2 \sum_{k=0} ^p \norm { \nabla_x \dot{\vec{M}}_{2k} ^\alpha }_{L^2} ^2  \le C(\epsilon, \nu_2) p^2 (\norm{\nabla_x u (t) }_{L^2} ^2 +1 )  \sum_{k=0} ^p \norm{\dot{\vec{M}}_{2k} ^\alpha }_{L^2} ^2
\end{gathered}
\end{equation}
or by introducing
\begin{equation}
\vec{M}_{2p} ^{e,\alpha} = \left ( \dot{\vec{M}}_0 ^\alpha, \dot{\vec{M}}_2 ^\alpha, \cdots, \dot{\vec{M}}_{2p} ^\alpha \right )
\end{equation}
we have
\begin{equation}
\frac{d}{dt} \norm{\vec{M}_{2p} ^{e,\alpha}}_{L^2} ^2 + \nu_2  \norm{ \nabla_x \vec{M}_{2p} ^{e,\alpha}}_{L^2} ^2  \le C(\epsilon, \nu_2) p^2 (\norm{\nabla_x u (t) }_{L^2} ^2 +1 ) \norm{\vec{M}_{2p} ^{e,\alpha}}_{L^2} ^2
\end{equation}
and by Gr\"{o}nwall we have
\begin{equation}
\begin{gathered}
\norm{\vec{M}_{2p} ^{e,\alpha} (t) }_{L^2} ^2 + \nu_2 \int_0 ^t \norm{\nabla_x \vec{M}_{2p} ^{e,\alpha} (s) }_{L^2} ^2 ds  \le \exp  \left (Cp^2 \left ( \norm{\nabla_x u  }_{L^\infty(0, T; L^2)} ^2 T + T  \right ) \right ) C(p) \norm{\vec{M}_{2p} ^{e} [\mu_0] }_{L^2} ^2.
\end{gathered} \label{L2H1even}
\end{equation}
Then using this we can find a $L^\infty (0, T; L^2) \cap L^2 (0, T; W^{1,2})$ bound for $M_{a,b}$ where $a+b = 2p+1$; from (\ref{alphamoments}) we bound all terms of the form $ \int_{\mathbb{R}^2} m_1 ^{a'} m_2 ^{b'}  \psi_\alpha f^\alpha dm$ by $C \int_{\mathbb{R}^2} (|m|^{a'+b'-1}+  |m|^{a'+b'+1 }  ) f^\alpha dm$, that is, we bound truncation $\psi_\alpha$ by $1$, and moments with odd degree $m_1 ^{a'} m_2 ^{b'}$ by arithmetic mean of neighboring radial moments $|m|^{a'+b'-1} + |m|^{a'+b'+1}$. Then using all the same techniques, we obtain
\begin{equation}
\begin{gathered}
\norm{\vec{M}_{2p+1} ^\alpha (t) }_{L^2} ^2 + \nu_2 \int_0 ^t \norm{\nabla_x \vec{M}_{2p+1} ^\alpha (s) }_{L^2} ^2 ds \le C(p, \epsilon)^{\norm{\nabla_x u}_{L^\infty (0, T; L^2 ) }  ^2 T+ T}  \norm{\vec{M}_{2(p+1)} [\mu_0]}_{L^2} ^2 \label{L2H1odd}
\end{gathered}
\end{equation}
where
\begin{equation}
\vec{M}_{2p+1} ^\alpha  = \left ( \dot{\vec{M}}_{0} ^\alpha, \dot{\vec{M}}_1 ^\alpha, \cdots, \dot{\vec{M}}_{2p+1} ^\alpha \right ).
\end{equation}
Note that instead of bounding $\int_0 ^T \norm{\nabla_x u (t) }_{L^2} ^2 dt$ by $\norm{\nabla_x u}_{L^\infty (0, T; L^2) } ^2 T$ we can bound it by $\norm{\nabla_x u}_{L^2 (0, T; L^2)} ^2$ to obtain a similar estimate
\begin{equation}
\begin{gathered}
\norm{\vec{M}_{2p} ^{e,\alpha} (t) }_{L^2} ^2 + \nu_2 \int_0 ^t \norm{\nabla_x \vec{M}_{2p} ^{e,\alpha} (s) }_{L^2} ^2 ds  \le \exp  \left (Cp^2 \left ( \norm{\nabla_x u  }_{L^2(0, T; L^2)} ^2 + T  \right ) \right ) C(p) \norm{\vec{M}_{2p} ^{e} [\mu_0] }_{L^2} ^2,
\end{gathered} \label{L2H1even2}
\end{equation}
which is crucial in global well-posedness, and
\begin{equation}
\begin{gathered}
\norm{\vec{M}_{2p+1} ^\alpha (t) }_{L^2} ^2 + \nu_2 \int_0 ^t \norm{\nabla_x \vec{M}_{2p+1} ^\alpha (s) }_{L^2} ^2 ds \le C(p,\epsilon)^{\norm{\nabla_x u}_{L^2 (0, T; L^2 ) }  ^2 + T}  \norm{\vec{M}_{2(p+1)} [\mu_0]}_{L^2} ^2. \label{L2H1odd2}
\end{gathered}
\end{equation}
\paragraph{$W^{1,2}$ bounds.} Then, we consider the third estimate, $L^\infty (0, T; W^{1,2} ) \cap L^2 (0, T; W^{2,2} )$ bounds for even moments of degree up to $2k$, where $k=4q-1$. We can apply same technique for odd moments too, but we only need even moments for the proof of our result. We multiply $- \Delta_x M_{a,b} ^\alpha$ to the equation (\ref{alphamoments}) and integrate: again integration by parts are rigorous. We use previous pointwise bound for truncated moments, and we get
\begin{equation}
\begin{gathered}
\frac{d}{dt} \norm{\nabla_x \vec{M}_{2k} ^{e,\alpha}}_{L^2} ^2 + \nu_2  \norm{ \Delta_x \vec{M}_{2k} ^{e,\alpha}}_{L^2} ^2 \le C(\epsilon) k^2  \norm{\nabla_x \vec{M}_{2k} ^{e,\alpha}}_{L^2} ^2 \\
+ C \norm{\nabla_x u(t) }_{L^2} \norm{\nabla_x \vec{M}_{2k} ^{e,\alpha}}_{L^4} ^2 + C(\nu_2)  k^2 \norm{\nabla_x u (t)}_{L^4} ^2 \norm{ \vec{M}_{2k} ^{e,\alpha}}_{L^4} ^2 + C(\epsilon, \nu_2) (kq)^2  \norm{ \vec{M}_{2(k+q-1)} ^{e,\alpha}}_{L^2} ^2 
\end{gathered}
\end{equation}
and again by Ladyzhenskaya's inequality, we have
\begin{equation}
\begin{gathered}
\frac{d}{dt} \norm{\nabla_x \vec{M}_{2k} ^{e,\alpha}}_{L^2} ^2 + \nu_2  \norm{ \Delta_x \vec{M}_{2k} ^{e,\alpha}}_{L^2} ^2  \\
 \le C(\epsilon, \nu_2 ) k^2 \left ( 1 + \norm{\nabla_x u (t) }_{L^2 } ^2 + \norm{\nabla_x u (t) }_{L^2 }  \norm{\Delta_x  u (t) }_{L^2 } \right )\norm{\nabla_x \vec{M}_{2k} ^{e,\alpha}}_{L^2} ^2  + C(\epsilon, \nu_2) (kq)^2  \norm{ \vec{M}_{2(k+q-1)} ^{e,\alpha}}_{L^2} ^2, 
\end{gathered}
\end{equation}
and again by Gr\"{o}nwall we have
\begin{equation}
\begin{gathered}
\norm{\nabla_x \vec{M}_{2k} ^{e,\alpha} (t)}_{L^2} ^2 + \nu_2  \int_0 ^t \norm{ \Delta_x \vec{M}_{2k} ^{e,\alpha} (s) }_{L^2} ^2  ds 
\\
\le C(\epsilon, \nu_2, k, q )^{ T+\norm{\nabla_x u}_{L^\infty (0, T; L^2 ) } ^2 T + \norm{\nabla_x u}_{L^\infty (0, T; L^2 ) } \norm{ \Delta_x  u}_{L^2 (0, T; L^2) } T^{\frac{1}{2}} } \\ 
\left ( \norm{\nabla_x \vec{M}_{2k} ^{e} [\mu_0] }_{L^2} ^2 + C(q, \epsilon )^{\norm{\nabla_x u}_{L^\infty (0, T; L^2 ) } ^2 T + T }  \norm{ \vec{M}_{2(k+q-1)} ^{e} [\mu_0] }_{L^2} ^2  \right ).
\end{gathered} \label{L2H2q}
\end{equation}
\paragraph{$L^1$ bounds.} In addition, we have $L^1$ bound for $\bar{M}_{4q}^\alpha $ and $M_{0,0} ^\alpha$: first we have
\begin{equation}
\partial_t M_{0,0} ^\alpha + u^\alpha \cdot \nabla M_{0,0} ^\alpha - \nu_2 \Delta_x M_{0,0} ^\alpha = 0
\end{equation}
and we can integrate them rigorously to conclude
\begin{equation}
\norm{M_{0,0} ^\alpha (t) }_{L^1} = \norm{M_{0,0} [\mu_0] }_{L^1}.
\end{equation}
Also, we have, by pointwise estimate
\begin{equation}
\begin{gathered}
\int_{\mathbb{R}^2} |m|^{2(3q-1)} \psi_\alpha f^\alpha dm \ge 0,  \\
\left | \int_{\mathbb{R}^2} |m| ^{2(2q-1) } m \otimes m \psi_\alpha f^\alpha dm \right | \le \bar{M}_{4q} ^\alpha, \\
\bar{M}_{2(2q-1)}^\alpha \le C(q) (M_{0,0}^\alpha + \bar{M}_{4q} ^\alpha )
\end{gathered}
\end{equation}
and integrating we get
\begin{equation}
\frac{d}{dt} \norm{\bar{M}_{4q} ^\alpha }_{L^1} \le C(q, \epsilon ) (\norm{\nabla_x u (t) }_{L^\infty } + 1) \norm{\bar{M}_{4q} ^\alpha }_{L^1} + C(q, \epsilon) \norm{M_{0,0} ^\alpha}_{L^1},
\end{equation}
and here by Agmon's inequality
\begin{equation}
\norm{\nabla_x u(t)}_{L^\infty} \le \norm{\nabla_x u(t) }_{L^2} ^\frac{1}{2} \norm{\Delta_x \nabla_x u(t) }_{L^2} ^\frac{1}{2}
\end{equation}
and by Gr\"{o}nwall we have
\begin{equation}
\norm{\bar{M}_{4q} ^\alpha (t) }_{L^1} \le C(q, \epsilon  ) ^{\norm{\nabla_x u}_{L^2 (0, T; W^{2,2} )} T^{\frac{1}{2}} + T } (\norm{\bar{M}_{4q} [\mu_0] }_{L^1}  + C(q, \epsilon ) \norm{M_{0,0} [\mu_0 ] }_{L^1} T ), \label{LinfL1}
\end{equation}
and from this we can say that $\bar{M}_{4q} ^\alpha$ (and also $\bar{M}_{2q}$ by the above pointwise estimate) is bounded in $L^\infty (0, T; L^p)$ where $1 \le p \le 2$ uniformly in $\alpha$ due to interpolation, bounds depend only on initial data. 
\paragraph{$W^{-1,2}$ bounds for $\partial_t M_{a,b}$s.} Finally, due to (\ref{alphamoments}), we notice that $\partial_t M_{a,b} ^\alpha$ is uniformly bounded in $L^2 (0, T; W^{-1, 2} )$; since $u^\alpha \in L^\infty (0, T; L^\infty )$ and $\nabla_x u^\alpha \in L^2 (0, T; L^\infty ) $ are uniformly bounded and all $M_{a,b} ^\alpha \in L^\infty (0, T; L^2 ) \cap L^2 (0, T; W^{1,2} )$ are uniformly bounded, terms involving $u^\alpha$ are uniformly bounded in $L^2 (0, T; L^2)$. Other terms except for $\Delta_x M_{a,b}^\alpha$ are uniformly bounded in $L^\infty (0 ,T; L^2)$, and $\Delta_x M_{a,b} ^\alpha$ is uniformly bounded in $L^2 (0, T; W^{-1, 2} )$. 
\paragraph{Weak limit of moments. } Since
\begin{equation}
\begin{gathered}
L^\infty (0, T; L^2) = \left ( L^1 (0, T; L^2) \right )^*, \,\, L^2 (0, T; L^2) = \left (L^2 (0, T; L^2) \right )^*, \\
L^\infty (0, T; L^q) = \left ( L^1 (0, T; L^{q'}) \right )^*, 1<q<2, \frac{1}{q} + \frac{1}{q'} = 1, \\
 L^2 (0, T; W^{-1,2}) =\left ( L^2 (0, T; W^{1,2} ) \right )^* ,
\end{gathered}
\end{equation}
by Theorem \ref{dualofBanach}, and since we have bounds (\ref{Xrtype}), (\ref{L2H1even}), (\ref{L2H1odd}), (\ref{L2H1even2}),  (\ref{L2H1odd2}), (\ref{L2H2q}), (\ref{LinfL1}) (and $L^\infty (0, T; L^p), 1 < p < 2$ bounds due to interpolation), by Banach-Alaoglu there is a weak* limit $M_{a,b}$,
\begin{equation}
M_{a,b} ^\alpha \rightarrow M_{a,b}
\end{equation}
in the weak-* topology of $L^\infty (0, T; L^2) \cap L^2 (0, T; W^{1,2} )$ with the bounds
\begin{equation}
\begin{gathered}
\sum_{p=0} ^\infty \frac{\norm{\bar{M}_{2p} (t) }_{L^2} }{(2p)! } \left ( \frac{r}{\exp \left ( \int_0 ^t \norm{\nabla_x u(s) }_{L^\infty} ds \right ) } \right )^{2p} \le C(r, T) \norm{\mu_0}_{X^r}, 
\end{gathered}\label{XrtypeM}
\end{equation}
\begin{equation}
\begin{gathered}
\norm{\vec{M}_{k}}_{L^\infty (0, T; L^2 )}^2 + \nu_2 \norm{\vec{M}_k}_{L^2 (0, T; W^{1,2} ) } ^2  \le C(k) ^{\norm{\nabla_x u}_{L^\infty (0, T; L^2 ) } ^2 T + T } \norm{\vec{M}_{k+(k \mathrm{mod} 2 )} [\mu_0] }_{L^2} ^2, 
\end{gathered}\label{L2H1M}
\end{equation}
\begin{equation}
\begin{gathered}
\norm{\vec{M}_{k}}_{L^\infty (0, T; L^2 )}^2 + \nu_2 \norm{\vec{M}_k}_{L^2 (0, T; W^{1,2} ) } ^2  \le C(k) ^{ T+ \norm{\nabla_x u}_{L^2 (0, T; L^2 ) } ^2 } \norm{\vec{M}_{k+(k \mathrm{mod} 2 )} [\mu_0] }_{L^2} ^2, 
\end{gathered}\label{L2H1M2}
\end{equation}
\begin{equation}
\begin{gathered}
\norm{\vec{M}_{8q-2} ^e}_{L^\infty (0, T; W^{1,2} ) } ^2 + \nu_2 \norm{\vec{M}_{8q-2} ^e }_{L^2 (0, T; W^{2,2} ) } ^2 \\  \le C(\epsilon, \nu_2, q) ^{T + \norm{u}_{L^\infty (0, T; W^{1,2})} ^2 T + \norm{u}_{L^\infty (0, T; W^{1,2})} \norm{u}_{L^2 (0, T; W^{2,2})} T^{\frac{1}{2}} }  \left ( \norm{\vec{M}_{8q-2} ^e [\mu_0] }_{W^{1,2} } ^2 + \norm{\vec{M}_{16q-6} [\mu_0 ] }_{L^2} ^2 \right ), 
 \end{gathered} \label{L2H2qM}
\end{equation}
\begin{equation}
\begin{gathered}
\norm{M_{0,0} (t)}_{L^1} = \norm{M_{0,0} [\mu_0] }_{L^1}, \\
\norm{\bar{M}_{4q}  }_{L^\infty (0, T; L^1)} \le C(q, \epsilon  ) ^{\norm{\nabla_x u}_{L^2 (0, T; W^{2,2} )} T^{\frac{1}{2}} + T } (\norm{\bar{M}_{4q} [\mu_0] }_{L^1}  + C(q, \epsilon )T \norm{M_{0,0} [\mu_0]}_{L^1}  )
\label{LinfL1M}
\end{gathered}
\end{equation}
where the last bound in (\ref{LinfL1M}) is due to bounds on $L^\infty (0, T; L^p )$, $1 < p \le 2$, and the fact that $p \rightarrow \norm{f}_{L^p}$ is continuous. Furthermore, $\partial_t M_{a,b} \in L^2 (0, T; W^{-1, 2} )$ with bounds depending only on the initial data, due to weak* convergence. Also we have $\norm{M_{0,0} (t)}_{L^1} = \norm{M_{0,0} [\mu_0] }_{L^1}$ instead of $\le$ sign by the last assertion of Theorem \ref{FPKeu}: take $V = |m|^2 + \log \max (|x|, 1) $, where $\log \max (|x|, 1)$ should be understood, by a slight abuse of notation, a smooth, bounded function equals it for $|x| > 2$. Then $K(t) = C + \norm{u(t)}_{L^\infty}$, $H(t) = C \norm{\nabla_x u (t) }_{L^\infty } $ works. We remark that (\ref{L2H1M}) and (\ref{L2H1M2}) look similar, but in the estimate (\ref{L2H1M2}) requires only a bound on $\norm{\nabla_x u}_{L^2 (0, T; L^2)}$, and this fact will be used in proving global well-posedness of the coupled system.

\subsection{Existence of moment solution} \label{momsolexistence}
In this subsection, we prove the existence of moment solution using the limits $\{ M_{a,b} \}_{a,b}$. There are two points to remark: first, since the convergence of $M_{a,b}^\alpha$ to $M_{a,b}$ is weak and not pointwise a priori, so we need Aubin-Lions compactness lemma to make the convergence locally pointwise. Second, since the Fokker-Planck equation we consider is fully parabolic, in fact we can rely on parabolic theory to find limit density function. First we establish positive semidefiniteness for $\{M_{a,b} \}_{a,b}$. For all $\alpha > 0$, the sequence $\{M_{a,b} ^\alpha \}_{(a,b)}$ are positive semidefinite, since they are moments of nonnegative measures. Therefore,
\begin{equation}
\int_0 ^T \int_{\mathbb{R}^2}\sum_{i,j} c_i c_j M_{a_i + a_j, b_i + b_j } ^\alpha (x, t) \phi(x, t) dxdt \ge 0
\end{equation}
for all nonnegative test functions $\phi \in L^1 (0, T; L^2)$: then by the weak* limit 
\begin{equation}
\int_0 ^T \int_{\mathbb{R}^2}\sum_{i,j} c_i c_j M_{a_i + a_j, b_i + b_j }  (x, t) \phi(x, t) dxdt \ge 0
\end{equation}
and that means, $\{M_{a,b} \}_{(a,b)}$ is also positive semidefinite. Similarly, 
\begin{equation}
\int_0 ^T \int_{\mathbb{R}^2} \phi (x, t) \left (\bar{M}_{a+b} ^\alpha \pm M_{a,b} ^\alpha \right ) dxdt \ge 0
\end{equation}
so for almost all $(x, t)$ $|M_{a,b} (x, t) | \le \bar{M}_{a+b} (x, t) $. Then, from (\ref{XrtypeM}) and Lemma \ref{pointwiseCar} we see that for almost all $(x, t)$ there is a nonnegative measure $\mu = \mu(x, t; dm)$ such that $M_{a,b} (x, t) = M_{a,b} [\mu] (x, t)$ for all $a, b \ge 0$. It remains to show that actually $\mu$ is a weak solution to the Fokker-Planck equation: first we show that for $\phi \in C_0 ^\infty ([0, T] \times \mathbb{R}^2_x \times \mathbb{R}^2_m )$ with $\phi(T,x,m) = 0$
we have
\begin{equation}
\begin{gathered}
\int_0 ^T \int_{\mathbb{R}^2_x} \int_{\mathbb{R}^2 _m} \left ( \partial_t \phi + u(t) \cdot \nabla_x \phi + \left ( (\nabla_x u (t) ) - \nabla_m U \right ) m \cdot \nabla_m \phi + \epsilon \Delta_m \phi + \nu_2 \Delta_x \phi \right ) \\
\mu(x, t; dm) dx dt = - \int_{\mathbb{R}^2 _x } \int_{\mathbb{R}^2 _m} \phi(0, x, m) \mu_0 (x; dm) dx.
\end{gathered}
\end{equation}
Suppose that $\mathrm{supp} \,\, \phi \subseteq [R_1, R_2] \times B(0, R)_x \times B(0, R)_m$, which is a compact rectangle. Let $\eta$ be a $C_0 ^\infty ([0, T] \times \mathbb{R}^2_x )$ function, $0 \le \eta \le 1$, $\eta = 1 $ in $[R_1, R_2] \times B(0, R)_x$ and $\eta =0$ outside $[R_1 - 1, R_2 +1] \times B(0, 2R)_x$. Then for any $a, b \ge 0$, we have
\begin{equation}
\begin{gathered}
\eta M_{a,b} ^\alpha \in L^2 (0, T; W_0 ^{1,2} ( \Omega)) \rightarrow \eta M_{a,b} \,\, \mathrm{weak *} \,\,\mathrm{in} \,\, L^2 (0, T; W_0 ^{1,2} (\Omega) ) \\
\partial_t (\eta M_{a,b} ^\alpha )  \in L^2 (0, T; W^{-1, 2} (\Omega) )  \rightarrow \partial_t (\eta M_{a,b} ) \,\, \mathrm{weak *} \,\,\mathrm{in} \,\, L^2 (0, T; W ^{-1,2} (\Omega) )
\end{gathered}
\end{equation}
where $\Omega = [R_1 - 1, R_2 +1] \times B(0, 2R)_x$. By Rellich-Kondrachov theorem $W_0 ^{1,2} (\Omega) \subseteq L^2 (\Omega)$ is compact and $L^2 (\Omega) \subseteq W^{-1, 2} (\Omega)$ is continuous. Therefore, by Aubin-Lions lemma we see that there is a subsequence $\eta M_{a,b} ^{\beta}$ which converges to $\eta M_{a,b}$ in $L^2 (0, T; L^2(\Omega))$. By a standard diagonalization method, there is a subsequence $\eta M_{a,b} ^\gamma$ such that all moments $\eta M_{a,b} ^\gamma$ converges to $\eta M_{a,b}$ in the topology of $L^2 (0, T; L^2 (\Omega))$. Therefore, there is a subsequence, again denoted by $\eta M_{a,b} ^\alpha$, converges to $\eta M_{a,b}$ almost everywhere, for all moments $a, b \ge 0$. Especially, $M_{a,b} ^\alpha (x, t) \rightarrow M_{a,b} (x, t)$ for almost all $(x, t) \in [R_1, R_2] \times B(0, R)_x $. Therefore, by Theorem \ref{weakconvergence} we see that $\mu^{\alpha} (x, t; dm)$ converges weakly to $\mu (x, t; dm)$. Note that $f^\alpha$ satisfies
\begin{equation}
\begin{gathered}
\int_0 ^T \int_{\mathbb{R}^2_x} \int_{\mathbb{R}^2 _m}  ( \partial_t \phi + u^\alpha (t) \cdot \nabla_x \phi + \left ( (\nabla_x u^\alpha (t) )m - \nabla_m U \right )  \psi_\alpha \cdot \nabla_m \phi \\ + \epsilon \Delta_m \phi + \nu_2 \Delta_x \phi  ) 
f^\alpha (x, t; dm) dx dt = - \int_{\mathbb{R}^2 _x } \int_{\mathbb{R}^2 _m} \phi(0, x, m) f_0 ^\alpha (x; dm) dx.
\end{gathered}
\end{equation}
If $\alpha > R$ then $m\psi_\alpha \cdot \nabla_m \phi = m \nabla_m \phi$. Also for almost every $x, t$ 
\begin{equation}
\begin{gathered}
\int_{\mathbb{R}^2_m } \left (\partial_t \phi - \nabla_m U  \cdot \nabla_m \phi + \epsilon \Delta_m \phi + \nu_2 \Delta_x \phi \right ) f^{\alpha} (x, t; dm) \\
 \rightarrow \int_{\mathbb{R}^2_m } \left (\partial_t \phi - \nabla_m U  \cdot \nabla_m \phi + \epsilon \Delta_m \phi + \nu_2 \Delta_x \phi \right ) \mu (x, t; dm)
\end{gathered}
\end{equation}
by weak convergence. Furthermore, the left term is bounded by $C_\phi (x, t) \eta M_{0,0}[f^\alpha]$, where $\eta M_{0,0} [f^\alpha] \rightarrow \eta M_{0,0} [\mu] \in L^2 (0, T; L^2 (\Omega) )$ and $C_\phi (x, t) \in L^2 (0, T; L^2 (\Omega))$ so we can apply generalized dominated convergence theorem to conclude that
\begin{equation}
\begin{gathered}
\int_0 ^T \int_{\mathbb{R}^2_x} \int_{\mathbb{R}^2_m } \left (\partial_t \phi - \nabla_m U  \cdot \nabla_m \phi + \epsilon \Delta_m \phi + \nu_2 \Delta_x \phi \right ) f^{\alpha} (x, t; dm) \\
 \rightarrow \int_0 ^T \int_{\mathbb{R}^2_x} \int_{\mathbb{R}^2_m } \left (\partial_t \phi - \nabla_m U  \cdot \nabla_m \phi + \epsilon \Delta_m \phi + \nu_2 \Delta_x \phi \right ) \mu (x, t; dm).
\end{gathered}
\end{equation}
Finally, for the term $u^{\alpha} (t) \cdot \nabla_x\phi + \nabla_x u^\alpha (f) m \cdot \nabla_m \phi$, we note that since $C_0 ^\infty ([0, T] \times \mathbb{R}^2_x \times \mathbb{R}^2_m ) = C_0 ^\infty ([0, T] \times \mathbb{R}^2_x ) \otimes C_0 ^\infty (\mathbb{R}^2_m )$ we only need to consider functions of the form $\phi(x, m, t) = \phi_1 (x, t) \phi_2 (m)$. Then the integral involving $u^{\alpha} (t) \cdot \nabla_x\phi $ becomes
\begin{equation}
\int_0 ^T \int_{\mathbb{R}^2 _x} u^{\alpha} (t) \cdot \nabla_x \phi_1 \int_{\mathbb{R}^2 _m} \phi_2 (m) f^{\alpha} (x, t; dm) dxdt
\end{equation}
and we note that $u^{\alpha} (t) \cdot \nabla_x \phi_1 \rightarrow u(t) \cdot \nabla_x \phi_1$ in $L^2 (0, T; L^2 (\Omega))$ and 
\begin{equation}
\int_{\mathbb{R}^2 _m} \phi_2 (m) f^{\alpha} (x, t; dm) \rightarrow  \int_{\mathbb{R}^2 _m} \phi_2 (m) \mu (x, t; dm)
\end{equation}
in $L^2 (0, T; L^2(\Omega))$ as before. We can deal with the term $\nabla_x u^\alpha (t) m \cdot \nabla_m \phi$ in the same way. Finally, by Lemma \ref{approxid} we see that $\int_{\mathbb{R}^2_m} \phi(0, x, m) f_0 ^\alpha (x; dm)$ converges to $\int_{\mathbb{R}^2_m}\phi(0, x, m) \mu (x; dm)$ almost every $x$, and they are bounded by $C_\phi M_{0,0} [f_0 ^\alpha]$ which converges to $C_\phi M_{0,0} [\mu_0]$ in $L^2$ (but since $\phi (0) $ is compactly supported it converges in $L^1$ too) so by generalized dominated convergence
\begin{equation}
- \int_{\mathbb{R}^2 _x } \int_{\mathbb{R}^2 _m} \phi(0, x, m) f_0 ^\alpha (x; dm) dx \rightarrow - \int_{\mathbb{R}^2 _x } \int_{\mathbb{R}^2 _m} \phi(0, x, m) \mu_0 (x; dm) dx.
\end{equation}
Similarly, for $\phi \in C_0 ^\infty (\mathbb{R}^2 _x \times \mathbb{R}^2 _m )$ we see that
\begin{equation}
\begin{gathered}
\int \phi f^\alpha (t) dmdx - \int \phi f^\alpha _0 dmdx = A_\alpha (t)
\end{gathered}
\end{equation}
and
\begin{equation}
\begin{gathered}
A_{\alpha} (t) \rightarrow \int_0 ^t \int_{m,x}  \left (u \cdot \nabla_x \phi +  ( \nabla_x u m - \nabla_m U )  \cdot \nabla_m \phi + \epsilon \Delta_m \phi + \nu_2 \Delta_x \phi \right ) \mu(x,\tau;dm)dx d\tau,
\end{gathered}
\end{equation}
where
\begin{equation}
A_\alpha (t) = \int_0 ^t \int_{m,x}  \left (u^\alpha \cdot \nabla_x \phi +  ( \nabla_x u^\alpha m - \nabla_m U )  \psi_\alpha \cdot \nabla_m \phi + \epsilon \Delta_m \phi + \nu_2 \Delta_x \phi \right ) f^\alpha dmdx d\tau.
\end{equation}
Note that by
\begin{equation}
\begin{gathered}
\norm{u^\alpha}_{L^1 (\mathbb{R}^2_m \times \mathbb{R}^2_x \times [0, T] ; f^\alpha dm dx dt) } \le \norm{u}_{L^\infty (0 , T; L^\infty )} \norm{M_{0,0} [\mu_0]}_{L^1}, \\
\norm{\nabla_x u^\alpha m \psi_\alpha}_{L^1 (\mathbb{R}^2_m \times \mathbb{R}^2_x \times [0, T] ; f^\alpha dm dx dt) }  \le \norm{\nabla_x u}_{L^\infty (0, T; L^2 ) } C(T, u) \norm{\vec{M}_2 [\mu_0] }_{L^2}, \\
\norm{\nabla_m U m \psi_\alpha}_{L^1 (\mathbb{R}^2_m \times \mathbb{R}^2_x \times [0, T] ; f^\alpha dm dx dt) } \le C(T, u) (\norm{\bar{M}_{2q} [\mu_0]}_{L^1} + C \norm{M_{0,0} [\mu_0] }_{L^1}),
\end{gathered} \label{densitybounds}
\end{equation}
we see that $|A_{\alpha} (t) | \le C(\phi) |t|$, where $C(\phi)$ depends only on $\phi$ and independent of $\alpha$. Furthermore, again $\int \phi f^\alpha (x, t, m) dm$ is pointwise bounded by $C_\phi M_{0,0} [f^\alpha] $, and note that in a ball $V \in \mathbb{R}^2 _x$ containing the support of $\int \phi f^\alpha (x, t, m) dm$ and a smooth cutoff $\eta$ which is $1$ in $\bar{V}$, with support contained in another ball $W$, $M_{0,0} [\mu^\alpha] \eta \in L^\infty (0, T; W^{1,2} (W) )$ with $\partial_t M_{0,0} \eta \in L^2 (0, T; W^{-1, 2} (W) )$: and $W^{1,2} (W) \subset L^2 (W)$ is again compact. Therefore again by Aubin-Lions, we see that for a subsequence $M_{0,0} [f^\alpha] \eta \rightarrow M_{0,0}[\mu] \eta$ strongly in $L^\infty (0, T; L^2 )$. Therefore, we conclude, by generalized dominated convergence, for almost every $t \in [0, T]$ $\int \phi f^\alpha (t) dmdx \rightarrow \int \phi \mu(t; dm) dx$, and we know that $\int \phi f_0 ^\alpha dmdx \rightarrow \int \phi \mu_0(dm) dx$. Therefore, we proved (\ref{solCauchy}).
Then we prove that in fact $\mu$ can be represented as a density function $f(x, t, m)$. Here we use the same argument to prove Theorem \ref{FPKeu}, used in \cite{MR3443169}. Let $U_k = B(0, k)_x \times B(0, k)_m \subset \mathbb{R}^2 _x \times \mathbb{R}^2 _m $, $J_k =[\frac{T}{k}, T \left ( 1 - \frac{1}{k} \right ) ]$, and $W_k$ be a neighborhood of $\bar{U}_k \times J_k$ with compact closure in $\mathbb{R}^2 _x \times \mathbb{R}^2 _m \times (0, T)$, for each $k > 2$. We then consider the subsequence of $f^{\alpha}$ that converging to $\mu(x, t; dm)$, what we used before.  Since we have
by Theorem \ref{FPKeu} 
\begin{equation}
\norm{f^\alpha}_{L^{\frac{7}{6}} (U_k \times J_k )} \le C(W_k , T, u, \mu_0)
\end{equation}
for each $k>2$. Then by Banach-Alaoglu and standard diagonalization technique, we can find a subsequence of $f^\alpha$ which converges weakly to a function $f(x, m, t)$ in $L^{\frac{7}{6} } (U_k \times J_k )$ for all $k>2$. Furthermore, $f(x, m, t) dm = \mu(x, t; dm)$ for almost every $(x, t)$. 

\subsection{Dependence on fluid velocity fields} \label{Flfielddep}
In this subsection, we prove the last assertion of Theorem \ref{momentsolutionexists}. Suppose that $u, v$ satisfies (\ref{velinitcond}) and $f, g$ be solutions of two microscopic equations with velocity field $u$ and $v$ respectively and same initial data $\mu_0$ satisfying conditions (\ref{initpositivity}), (\ref{initmoment}), (\ref{initstress}), (\ref{initentropy}). Also $f^\alpha$ and $g^\alpha$ is defined same as before. Then we have
\begin{equation}
\begin{gathered}
\partial_t (f^\alpha - g^\alpha) + u^\alpha \cdot \nabla_x (f^\alpha - g^\alpha ) + (\nabla_x u^\alpha ) m \psi_\alpha \cdot \nabla_m (f^\alpha - g^\alpha ) \\
- \nabla_m \cdot ( \nabla_m U \psi_\alpha (f^\alpha - g^\alpha ) ) - \epsilon \Delta_m (f^\alpha - g^\alpha ) - \nu_2 \Delta_x (f^\alpha - g^\alpha ) \\
= - (u^\alpha - v^\alpha ) \cdot \nabla_x g^\alpha - \nabla_x (u^\alpha - v^\alpha )m \psi_\alpha \cdot \nabla_m g^\alpha
\end{gathered} \label{difference}
\end{equation}
in the classical sense. Let $sgn_{\beta}$ be a smooth, increasing regularization of sign function where $sgn_{\beta} (s) = sign(s)$ for $|s| \ge \beta$, and $|s|_\beta = \int_0 ^s sgn_\beta (r) dr$. By multiplying $|m|^{2k} sgn_\beta ( f^\alpha - g^\alpha )$, where $k \le 2q - 1$, to (\ref{difference}) we have
\begin{equation}
\begin{gathered}
\partial_t \left ( |m|^{2k} |f^\alpha - g^\alpha |_\beta \right ) + u^\alpha \cdot \nabla_x \left ( |m|^{2k} |f^\alpha - g^\alpha |_\beta \right ) + \nabla_x u^\alpha m \psi_\alpha |m|^{2k} \cdot \nabla_m |f^\alpha - g^\alpha |_{\beta} \\
- \nabla_m \cdot ( \nabla_m U \psi_\alpha ) |m|^{2k} (f^\alpha - g^\alpha ) sgn_\beta (f^\alpha - g^\alpha ) -\nu_2 |m|^{2k} sgn_\beta (f^\alpha - g^\alpha ) \Delta_x (f^\alpha - g^\alpha ) \\
- \nabla_m U \psi_\alpha \cdot \nabla_m |f^\alpha -g^\alpha |_\beta |m|^{2k} - \epsilon |m|^{2k} sgn_{\beta} (f^\alpha - g^\alpha ) \Delta_m (f^\alpha - g^\alpha ) \\
= - (u^\alpha - v^\alpha ) \cdot \nabla_x g^\alpha |m|^{2k} sgn_\beta (f^\alpha - g^\alpha)  - \nabla_x (u^\alpha - v^\alpha ) m \psi_\alpha \cdot \nabla_m g^\alpha |m|^{2k} sgn_\beta (f^\alpha - g^\alpha).
\end{gathered}
\end{equation}
Integrating in $m$ variable, we have
\begin{equation}
\begin{gathered}
(\partial_t + u^\alpha \cdot \nabla_x ) \int |m|^{2k} |f^\alpha - g^\alpha|_\beta dm - ( \nabla_x u^\alpha ) : \int \nabla_m (m \psi_\alpha |m|^{2k} ) |f^\alpha -  g^\alpha |_\beta dm \\
= I_1 + I_2 + I_3 + I_4
- \int (u^\alpha - v^\alpha ) \cdot \nabla_x g^\alpha |m|^{2k} sgn_\beta (f^\alpha- g^\alpha ) dm \\
 - \int \nabla_x (u^\alpha - v^\alpha ) m \psi_\alpha \cdot \nabla_m g^\alpha |m|^{2k} sgn_\beta (f^\alpha -g^\alpha ) dm
\end{gathered}
\end{equation}
where 
\begin{equation}
\begin{gathered}
I_1 = \int \nabla_m \cdot (\nabla_m U \psi_\alpha ) |m|^{2k} (f^\alpha - g^\alpha) sgn_\beta (f^\alpha - g^\alpha) dm, \\
I_2 = - \int \nabla_m \cdot (\nabla_m U \psi_\alpha |m|^{2k} ) |f^\alpha - g^\alpha |_\beta dm, \\
I_3 = \epsilon \int |m|^{2k} sgn_\beta (f^\alpha -g^\alpha ) \Delta_m (f^\alpha - g^\alpha ) dm, \\
I_4 = \nu_2 \int |m|^{2k} sgn_\beta (f^\alpha - g^\alpha ) \Delta_x (f^\alpha - g^\alpha ) dm.
\end{gathered}
\end{equation}
Note that
\begin{equation}
\begin{gathered}
\left | \int \nabla_m (m \psi_\alpha |m|^{2k} ) |f^\alpha - g^\alpha|_\beta dm \right | \le C\frac{1}{\alpha} \int |m|^{2k+1} |f^\alpha - g^\alpha | dm + C k  \int |m|^{2k} |f^\alpha - g^\alpha |_\beta dm,
\end{gathered}
\end{equation}
and
\begin{equation}
\begin{gathered}
I_1 + I_2 = \int \nabla_m \cdot (\nabla_m U \psi_\alpha) |m|^{2k} \left ( (f^\alpha - g^\alpha )  sgn_\beta (f^\alpha - g^\alpha ) - |f^\alpha - g^\alpha |_\beta \right ) dm \\
- 2k \int |m|^{2(k-1) } m \cdot \nabla_m U \psi_\alpha |f ^\alpha - g^\alpha |_\beta dm
\end{gathered}
\end{equation}
and the first term, denoted by $J_{\alpha, \beta}$, is bounded pointwise by $$C \left ( \frac{1}{\alpha} \bar{M}_{2(k+q) -1} [f^\alpha + g^\alpha] + \bar{M}_{2(k+q-1)} [f^\alpha + g^\alpha ] \right )$$ and pointwisely converges to 0 as $\beta \rightarrow 0$. On the other hand, the second term is nonpositive. Thus $I_1 + I_2 \le J_{\alpha, \beta}$. On the other hand,
\begin{equation}
\begin{gathered}
I_3 = -\epsilon \int \nabla_m \left  (|m|^{2k} sgn_\beta (f^\alpha - g^\alpha ) \right ) \cdot \nabla_m (f^\alpha - g^\alpha ) dm \\
= - \epsilon \int 2k |m|^{2(k-1)} m \cdot \nabla_m |f^\alpha - g^\alpha |_\beta dm - \epsilon \int |m|^{2k} sgn ' _\beta  (f^\alpha - g^\alpha ) | \nabla_m (f^\alpha - g^\alpha ) |^2 dm \\
\le 2k \epsilon \int \nabla_m \cdot ( |m|^{2(k-1)} m ) |f^\alpha - g^\alpha |_\beta dm,
\end{gathered}
\end{equation}
and finally
\begin{equation}
\begin{gathered}
I_4 = \nu_2 \int \nabla_x \cdot \left ( |m|^{2k} sgn_\beta (f^\alpha -g^\alpha ) \nabla_x (f^\alpha - g^\alpha ) \right ) dm  - \nu_2 \int |m|^{2k} sgn ' _\beta (f^\alpha - g^\alpha ) |\nabla_x (f^\alpha - g^\alpha ) |^2 dm \\
\le \nu_2  \nabla_x \cdot \left ( \int  \left ( \nabla_x ( |m|^{2k} |f^\alpha - g^\alpha |_\beta ) \right ) dm \right ).
\end{gathered}
\end{equation}
Therefore, we have
\begin{equation}
\begin{gathered}
(\partial_t + u^\alpha \cdot \nabla_x ) \left ( \int |m|^{2k} |f^\alpha - g^\alpha|_\beta dm \right ) \\
\le C \norm{\nabla_x u (t)}_{L^\infty _x}  \left ( k   \int |m|^{2k} |f^\alpha - g^\alpha|_\beta dm  + \frac{1}{\alpha}   \int |m|^{2k+1} |f^\alpha - g^\alpha|_\beta dm  \right ) \\
+J_{\alpha, \beta} + C k^2 \epsilon \int |m|^{2(k-1)} |f^\alpha - g^\alpha | _\beta dm + \nu_2 \nabla_x \cdot \int \nabla_x (|m|^{2k} |f^\alpha -g^\alpha |_\beta ) dm \\
+ \norm{(u -v) (t)}_{L^\infty _x} \int |m|^{2k} |\nabla_x g^\alpha | dm + \norm{ \nabla_x (u-v) (t) }_{L^\infty _x } \int |m|^{2k+1} |\nabla_m g^\alpha| dm.
\end{gathered}
\end{equation}
Then we multiply $\int |m|^{2k} |f^\alpha - g^\alpha |_\beta dm$ and integrate in $x$, and divide by $\norm{\int |m|^{2k} |f^\alpha - g^\alpha |_\beta (t) dm }_{L^2 _x } $ we have
\begin{equation}
\begin{gathered}
\frac{d}{dt} \norm{\int |m|^{2k} |f^\alpha - g^\alpha |_\beta (t) dm }_{L^2 _x } \\
\le C ( \norm{\nabla_xu (t) }_{L^\infty  _ x } + 1) \norm{\int |m|^{2k} |f^\alpha - g^\alpha |_\beta (t) dm }_{L^2 _x } + \norm{J_{\alpha, \beta } }_{L^2 _x}  \\
\norm{\int |f^\alpha - g^\alpha |_\beta dm}_{L_x ^2}+ \frac{C\norm{\nabla_xu (t) }_{L^\infty  _ x }}{\alpha} \left ( \norm{\bar{M}_{2k+1} [f^\alpha]}_{L^2 _x} + \norm{\bar{M}_{2k+1} [g^\alpha]}_{L^2 _x} \right )\\
 + \norm { (u-v) (t) }_{L^\infty _x} \left (\int \bar{M}_{4k} [g^\alpha] \left ( \int \frac{|\nabla_x g^\alpha |^2}{g^\alpha} dm\right ) dx \right )^{\frac{1}{2}} \\
+ \norm{ \nabla_x  (u-v) (t) }_{L^\infty _x}  \left (\int \bar{M}_{4k+2} [g^\alpha] \left ( \int \frac{|\nabla_m g^\alpha |^2}{g^\alpha} dm\right ) dx \right )^{\frac{1}{2}}
\end{gathered}
\end{equation}
Since $f^\alpha (0) = g^\alpha (0)$, by Gr\"{o}nwall we have
\begin{equation}
\begin{gathered}
\norm{\int |m|^{2k} |f^\alpha - g^\alpha |_\beta (t) dm }_{L^2 _x } \\ 
\le \exp (C ( \norm{\nabla_x u}_{L^1 (0, T; L^\infty _x )} + 1 ) ) ( \int _0 ^T \norm{J_{\alpha, \beta} } _{L^2} dx    \\
+ \frac{C}{\alpha} \norm{\nabla_x u } _{L^2 (0, T; L^\infty _x ) } \left ( \norm{\bar{M}_{2k+1} [f^\alpha]}_{L^2 (0, T; L^2 _x)} + \norm{\bar{M}_{2k+1} [g^\alpha]}_{L^2 (0, T; L^2 _x)} \right ) + \norm{\int |f^\alpha - g^\alpha |_\beta dm }_{L^1 (0, T; L^2 _x ) } \\
+ \int_0 ^T \norm {u-v (t) }_{L^\infty _x } \norm{ \bar{M}_{4k} [g^\alpha] (t) }_{L^\infty _x } ^{\frac{1}{2}} \left ( \int \int \frac{|\nabla_x g^\alpha |^2}{g^\alpha} dm dx \right ) ^{\frac{1}{2}} dt \\
 \int_0 ^T \norm {\nabla_x (u-v) (t) }_{L^\infty _x } \norm{ \bar{M}_{4k+2} [g^\alpha] (t) }_{L^\infty _x } ^{\frac{1}{2}} \left ( \int \int \frac{|\nabla_m g^\alpha |^2}{g^\alpha} dm dx \right ) ^{\frac{1}{2}} dt )
\end{gathered} \label{difference2}
\end{equation}
and by (\ref{L2H1M}) we have that $\norm{\bar{M}_{2k+1} [f^\alpha] }_{L^\infty _t L^2 _x }  + \norm{\bar{M}_{2k+1} [g^\alpha]}_{L^\infty _t L^2 _x } \le C$ where $C$ depends only on initial data $\mu_0$ and $\nabla_x u, \nabla_x v$, independent of $\alpha$. Also, by (\ref{L2H2qM}), and by Agmon's inequality, we have $\bar{M}_{4k} [g^\alpha], \bar{M}_{4k+2} [g^\alpha] \in L^2 (0, T; L^\infty )$ with bounds depending only on initial data $\mu_0$ and velocity field $v$, again independent of $\alpha$. Also, $g^\alpha$ satisfies the conditions of Theorem \ref{entropycon} : 
\begin{equation}
\begin{gathered}
\int \int |v(x, t) |^2 g^\alpha (x, m, t) dm dx \le \norm{v}_{L^\infty (0, T; L^\infty_x ) } ^2 \norm{M_{0,0}[g^\alpha] }_{L^\infty (0, T; L^1 ) }, \\
\int _0 ^T \int \int |\nabla_x v (x, t) m |^2 g^\alpha (x, m, t) dm dx \le \norm {\nabla_x v }_{L^2 (0, T; L^\infty _x ) } ^2 \norm{\bar{M}_2 [g^\alpha] }_{L^\infty (0, T; L^1 _x ) }, \\
\int \int |\nabla_m U|^2 g^\alpha (x, m, t) dm dx \le C(q) \norm{\bar{M}_{4q-2} [g^\alpha]}_{L^\infty (0, T; L^1 ) }
\end{gathered}
\end{equation}
and since $$\log \left ( \max ( \sqrt{|x|^2 + |m|^2 }, 1 ) \right ) \le \frac{\log 2}{2} + \log \left ( \max ({|x|},1 ) \right ) + \log \left ( \max ({|m|},1 ) \right ), $$
$\log \left ( \max ({|m|},1 ) \right ) ^2 \le C (1+ |m|^2)$ and so $$\int \int \log \left ( \max ({|m|},1 ) \right ) ^2 g^\alpha (t)  dmdx \le C \left ( \norm{M_{0,0}[g^\alpha] }_{L^\infty (0, T; L^1 _x ) } + \norm{\bar{M}_{2}[g^\alpha ] }_{L^\infty (0, T; L^1 _x ) }  \right ),$$ which is bounded by a constant depending only on $u$ and $\mu_0$, and not in $\alpha$, it suffices to bound $$\int \int \log \left ( \max ({|x|},1 ) \right ) ^2  g^\alpha dm dx = \int M_{0,0} [g^\alpha ] \log \left ( \max ({|x|},1 ) \right ) ^2 dx. $$ Let $\Psi(x)$ be a smooth, nonnegative function in $x$ such that $\Psi \ge \log \left ( \max ({|x|},1) \right ) ^2$, $\Psi = \log \left ( \max ({|x|},1 ) \right ) ^2$ for $|x| \ge 2$. Since $M_{0,0} [g^\alpha]$ satisfies 
\begin{equation}
\partial_t M_{0,0} [g^\alpha] + v^\alpha \cdot \nabla_x M_{0,0} [g^\alpha] = \nu_2 \Delta_x M_{0,0} [g^\alpha]
\end{equation}
it can be easily seen that 
\begin{equation}
\begin{gathered}
\int \int \log \left ( \max ({|x|},1 ) \right ) ^2 M_{0,0} [g^\alpha] (t) dx  \\ 
\le C \left ( 1+ \norm{v}_{L^\infty (0, T; L^\infty _x ) } T + \int \log \left ( \max ({|x|},1 ) \right ) ^2 M_{0,0} [\mu_0 ^\alpha ] dx  \right )
\end{gathered}
\end{equation}
but note that $M_{0,0} [\mu_0 ^\alpha] (x) = (g_{\alpha^{-1}} *_x M_{0,0} [\mu_0] ) (x)$, and we have a following simple inequality
\begin{equation}
\log \left ( \max (|x+y|, 1 ) \right ) ^2 \le 4 + 2 \log \left ( \max (|x|, 1 ) \right ) ^2 + 2 \log \left ( \max (|y|, 1 ) \right ) ^2
\end{equation}
so we have
\begin{equation}
\begin{gathered}
\int \Lambda (x) ^2 g_{\alpha^{-1} } *_x M_{0,0} [\mu_0 ] (x) dx \le 4 \norm{M_{0,0} [\mu_0]}_{L^1} + 2 \int \Lambda ^2 M_{0,0} [\mu_0] dx \\
+ 2 \norm{M_{0,0} [\mu_0] }_{L^1} \int g_{\alpha ^{-1} } (x) \Lambda (x) ^2 dx.
\end{gathered}
\end{equation}
However, note that 
\begin{equation}
 \int g_{\alpha ^{-1} } (x) \Lambda (x) ^2 dx = \int_{|x| \ge 1} g_{\alpha ^{-1} } (x) \left ( \log |x| \right )^2 dx
\end{equation}
and if $|x| \ge 1$ and $\alpha \ge 4$, $g_{\alpha^{-1} } (x) \le g_{4} (x)$ so again we can find a bound for $\int \int \log \left ( \max ({|x|},1 ) \right ) ^2 M_{0,0} [g^\alpha] (t) dx$ which depends only in uniform data and $v$, is independent of $\alpha$ (for large enough $\alpha$), and is uniform in $[0, T]$. Also note that our initial condition implies that $\int \int \mu_0 ^\alpha \left | \log \mu_0 ^\alpha \right | dm dx < \infty$. Then by the bound obtained in the proof of Theorem \ref{entropycon}, we conclude that
\begin{equation}
\begin{gathered}
\int_0 ^T \int \int \frac{|\nabla_x g^\alpha|^2 + |\nabla_m g^\alpha| ^2 }{g^\alpha} dm dx dt \\
 \le C ( T \norm{v}_{L^\infty (0, T; L^\infty )} ^2 \norm {M_{0,0}[ \mu_0]}_{L^1} + (\norm{\nabla_x v}_{L^2 (0, T; L^\infty )} + T) \norm{\bar{M}_2 [g^\alpha ]}_{L^\infty (0, T; L^1) } \\
  + \norm{\bar{M}_{4q-2} [g^\alpha] }_{L^\infty (0 ,T; L^1 ) } T  + \norm{v}_{L^\infty (0, T; L^\infty ) } T + 1 + \int \Lambda ^2 M_{0,0} [\mu_0] dx ) \\
+ \int \int \mu_0 ^\alpha \log \mu_0 ^\alpha dm dx.
\end{gathered}
\end{equation}
But note that by Jensen's inequality applied to $\Phi(s) = s \log s $, we have, for each $(x, m)$,
\begin{equation}
\begin{gathered}
\mu_0 ^\alpha \log \mu_0 ^\alpha (x, m) \\
= \int \int \mu_0 (x-y, m-n) g_{\alpha^{-1} } (y, n) \log  \left ( \int \mu_0 (x-y', m-n' ) g_{\alpha^{-1} } (y', n' ) \right ) dn dy \\
= \Phi \left ( \mathbb{E}_{g_{\alpha^{-1} }} [ \mu_0 ( x - \cdot , m-\cdot  )]  \right ) \le \mathbb{E}_{g_{\alpha^{-1} }}  \left [ \Phi (\mu_0 (x - \cdot, m - \cdot )) \right ] \\
= \int \int g_{\alpha^{-1} } (y, n) \Phi ( \mu_0 ( x - y, m - n ) ) dn dy
\end{gathered}
\end{equation}
and therefore
\begin{equation}
\int \int \mu_0 ^\alpha \log \mu_0 ^\alpha dx dm \le \int \int \mu_0 \log \mu_0 dx dm.
\end{equation}
Therefore, by H\"{o}lder's inequality we can bound the last two terms in the (\ref{difference2}) by
\begin{equation}
\begin{gathered}
 \norm{\bar{M}_{4k} [ g^\alpha] + \bar{M}_{4k+2} [g^\alpha]}_{L^2 (0, T; L^\infty _x ) }^{\frac{1}{2} }  \left ( \int_0 ^T \int \int \frac{|\nabla_x g^\alpha|^2 + |\nabla_m g^\alpha| ^2 }{g^\alpha} dm dx dt \right )^{\frac{1}{2} } \norm{u-v}_{L^4 (0, T; W^{1, \infty} ) } \\
\le \norm{u-v}_{L^4 (0, T; W^{1, \infty} ) } C(\norm{v}, \norm{\mu_0}, T)
\end{gathered}
\end{equation}
where $C(\norm{v}, \norm{\mu_0}, T)$ depends only on those three (except for coefficients like $\nu_2, \epsilon$), is increasing in each of the variables, and does not blow up for finite $\norm{v}, \norm{\mu_0}$, or $T$. The term
\begin{equation}
\norm{\int |f^\alpha - g^\alpha|_\beta dm }_{L^1 (0, T; L^2 ) }
\end{equation}
can be bounded in the same way, just plugging in $k=0$ to (\ref{difference2}) and removing the term $\norm{\int |f^\alpha - g^\alpha|_\beta dm }_{L^1 (0, T; L^2 ) }$ in the right side, and since we have $|m|^{2k} |f^\alpha- g^\alpha |, \,\,|m|^{2k} |f^\alpha - g^\alpha |_\beta \le |m|^{2k} (f^\alpha + g^\alpha )$ by taking $\beta \rightarrow 0$ to apply dominated convergence and taking $\alpha \rightarrow \infty$ we have
\begin{equation}
\norm{\int |m|^{2k} |f (t) - g(t) | dm }_{L^2 _x } \le C( \norm{u}, \norm{v}, \norm{\mu_0}, T ) \norm{u - v}_{L^4 (0, T; W^{1, \infty } ) }
\end{equation}
where again $C(\norm{u}, \norm{v}, \norm{\mu_0}, T)$ depends only on those four (except for coefficients like $\nu_2, \epsilon$), is increasing in each of the variables, and does not blow up for finite $\norm{u}, \norm{v}, \norm{\mu_0}$, or $T$. Here $\norm{u} = \norm{u}_{L^\infty (0, T; W^{2,2} ) \cap L^2 (0, T; W^{3,2} ) }$ and similar for $\norm{v}$, and $\norm{\mu_0}$ is a bound for (\ref{initstress}) and (\ref{initentropy}). Let
\begin{equation}
\sigma_1 = \int \nabla_m U \otimes m f dm, \,\, \sigma_2 = \int \nabla_m U \otimes m g dm.
\end{equation} 
Then in the weak sense as in Lemma \ref{momentofmomentsols}, we have
\begin{equation}
\partial_t (\sigma_1 - \sigma_2 ) + u \cdot \nabla_x (\sigma_1 - \sigma_2 ) - \nu_2 \Delta_x (\sigma_1 - \sigma_2 ) = I_1 + I_2
\end{equation}
where
\begin{equation}
\begin{gathered}
I_1 = - (u-v) \cdot \nabla_x \sigma_2 \\
+ 4q (q-1) \int |m|^{2(q-2) } \left ( (\nabla_x u -\nabla_x v ) : m \otimes m \right ) m \otimes m g dm \\
+ 2q \int |m|^{2(q-1) } ( (\nabla_x u - \nabla_x v ) m \otimes m + m \otimes (\nabla_x u - \nabla_x v ) m ) g dm
\end{gathered}
\end{equation}
and
\begin{equation}
\begin{gathered}
I_2 = - (2q)^3 \epsilon \int |m|^{4(q-1)} m \otimes m (f-g) dm \\
+ 4q(q-1) \int |m|^{2(q-2)} ( (\nabla_x u) : m \otimes m ) m \otimes m (f-g) dm \\
+ 2q \int |m|^{2(q-1) } ( (\nabla_x u) m \otimes m + m \otimes (\nabla_x u )m ) (f-g) dm \\
+ 2q\epsilon \left ( 4q (q-1) \int |m|^{2(q-2) } m \otimes m (f-g) dm + 4 \int |m|^{2(q-1) } \mathbb{I} (f-g) dm \right )  .
\end{gathered}
\end{equation}
Then we see that 
\begin{equation}
\norm{I_1 (t) }_{L^2} + \norm{I_2 (t) }_{L^2} \le C(\norm{\mu_0 } , \norm{u}, \norm{v}, T ) \norm{u-v}_{L^4 (0, T; W^{1, \infty } ) }.
\end{equation}
Therefore, by multiplying $\sigma_1 - \sigma_2$ and integrating in $x$ variables, and using $\sigma_1 (0) = \sigma_2 (0)$ we have
\begin{equation}
\sup_{0 \le t \le T} \norm {\sigma_1 (t) - \sigma_2 (t) }_{L^2} ^2 + \nu_2 \int _0 ^T \norm {\nabla_x (\sigma_1 - \sigma_2 ) } _{L^2 } ^2 dt \le C T \norm{u-v}_{L^4 (0, T; W^{1, \infty } ) } ^2 .
\end{equation}
Also, multiplying $- \Delta_x (\sigma_1 - \sigma_2)$ and integrating in $x$ variable we get
\begin{equation}
\sup_{0 \le t \le T} \norm { \nabla_x (\sigma_1 (t) - \sigma_2 (t) ) }_{L^2} ^2 + \nu_2 \int _0 ^T \norm {\Delta_x (\sigma_1 - \sigma_2 ) } _{L^2 } ^2 dt \le C T \norm{u-v}_{L^4 (0, T; W^{1, \infty } ) } ^2 .
\end{equation}
In conclusion, we have
\begin{equation}
\begin{gathered}
\norm{\sigma_1 - \sigma_2 } _{L^\infty (0, T; W^{1,2}) \cap L^2 (0, T; W^{2,2} ) } \\
 \le C(\nu_2, \norm{u}, \norm{v}, \norm{\mu_0 }, T )\sqrt{T} \norm{u -v }_{L^\infty (0, T; \mathbb{P}W^{2,2} ) \cap L^2 (0, T; \mathbb{P}W^{3,2} ) }
\end{gathered} \label{differenceestimate}
\end{equation}
again $C(\nu_2, \norm{u}, \norm{v}, \norm{\mu_0 }, T ) $ has the same property as before, and $C \rightarrow \infty$ as $\nu_2 \rightarrow 0$.
\begin{remark}
If we assume the initial data $\mu_0$ for $f$, and the initial data $\nu_0$ for $g$ do not coincide, then previous arguments give the following modification of (\ref{differenceestimate}):
\begin{equation}
\begin{gathered}
\norm{\sigma_1 - \sigma_2 } _{L^\infty (0, T; W^{1,2}) \cap L^2 (0, T; W^{2,2} ) } \\
 \le C \norm{\sigma_1 (0) -\sigma_2 (0)}_{W^{1,2} } + C \sqrt{T} \left ( \norm{u -v }_{L^\infty (0, T; \mathbb{P}W^{2,2} ) \cap L^2 (0, T; \mathbb{P}W^{3,2} ) } +  \norm{\bar{M}_0 [\mu_0 - \nu_0 ] + \bar{M}_{4q-2} [\mu_0 - \nu_0 ] }_{L^2} \right ).
\end{gathered} \label{diffdiffinitialdata}
\end{equation}
For any $k \ge 0$, the term $\bar{M}_{2k} [\mu_0 - \nu_0 ] $ cannot be controlled by $\bar{M}_{2k} [\mu_0] - \bar{M}_{2k} [\nu_0]$. However, this term is unavoidable; it is possible that $\bar{M}_{2k} [\mu_0] = \bar{M}_{2k} [\nu_0]$ while $\mu_0 \ne \nu_0$.
\end{remark}
Therefore, we have proved Theorem \ref{momentsolutionexists}.
\begin{remark}
As mentioned before, the condition (\ref{initmoment}) can be dropped in proving local and global well-posedness of the coupled system: we can only assume (\ref{velinitcond}), (\ref{initpositivity}), (\ref{initstress}), and that $\norm{\vec{M}_{16q} [\mu_0 ] }_{L^2 _x } < \infty$ to show that there exists a unique weak solution to the Fokker-Planck equation (\ref{FPwithu}), satisfying all the conditions fo the definition for the moment solution except for third one, and satisfying bounds (\ref{L2H1M}), (\ref{L2H1M2}),  (\ref{L2H2qM}), and (\ref{LinfL1M}). Also, note that (\ref{initentropy}) is used only for the estimate (\ref{differenceestimate}), which is used in proving local existence of the coupled system. 
\end{remark}
\begin{remark}
In the condition (\ref{initentropy}), the condition $\int_{\mathbb{R}^2} |\Lambda(x) |^2 M_{0,0} [f_0] (x) dx <\infty$, which controls the growth of $f_0$ at infinity, is introduced in many kinetic models, for example, Boltzmann equation (\cite{MR1014927}). Although the physical interpretation of the above condition is not evident, that condition guarantees us that the entropy $\int f \log f dx$ remains greater than $-\infty$. Here is an example showing that if we do not have such restriction,  our solution starts with finite entropy but fall into $- \infty$ entropy after some time. Suppose that we are solving 1-dimensional heat equation $\partial_t f = \partial_x ^2 f$ in the whole line, and let the initial data be 
\begin{equation}
f_0 (x) = \sum_{n=1} ^\infty \mathrm{1}_{(10n - a_n , 10n + a_n ) } (x)
\end{equation}
where
\begin{equation}
a_n = \frac{c}{(n+1) \left ( \log (n+1) \right )^2 }
\end{equation}
where $c$ is chosen that $\sum_{n=1} ^\infty a_n \le 1$ and $a_n < \frac{1}{2} $ for all $n$. Let 
\begin{equation}
\Phi (s) = s \log s, \,\, g_r (x) = \frac{1}{\sqrt{4 \pi r } } e^{-\frac{x^2}{4r} }.
\end{equation}
Then $\int_{\mathbb{R}} \Phi (f_0) dx = 0$, since $\Phi (f_0) (x) = 0$ for all $x$. Then 
\begin{equation}
f(x, t) = \sum_{n=1} ^\infty g_t * \mathrm{1}_{(10n - a_n, 10n + a_n ) } (x)
\end{equation}
and we see that $\norm{f(t) }_{L^\infty} \le \norm{f_0}_{L^1} \norm{g_t}_{L^\infty} < \frac{1}{4 \sqrt{\pi t}} \le \frac{1}{2}$ for all $t>1 $ and $f(x, t) \ge 0$ for all $(x, t)$. For $t = 1$, if $|x - 10n | < t$, we see that 
\begin{equation}
\frac{1}{2} \ge f(x, t) \ge g_t * \mathrm{1}_{(10n - a_n , 10n + a_n ) } (x) \ge \frac{a_n} { e \sqrt{\pi}} = \frac{a_n}{\sqrt{\pi t } } e^{-t}.
\end{equation}
Then since $\Phi (s) $ is decreasing for $0 \le s \le \frac{1}{2}$, $\Phi(f(x, t) ) \le \Phi (\frac{a_n}{e \sqrt{\pi}} ) = \frac{a_n}{e \sqrt{\pi}} \log a_n - a_n \log (e \sqrt{\pi}).$ Then
\begin{equation}
\int_{\mathbb{R}} \Phi (f(x, t) ) dx \le \sum_{n=1} ^\infty \int_{(10n - t, 10n + t ) }\Phi (f(x, t) ) dx \le 2 \sum_{n=1} ^\infty a_n \log a_n - 2 C = - \infty.
\end{equation}
Therefore, although $f_0$ started with zero entropy, $f(t)$ has $-\infty$ entropy at $t = 1$. Same argument shows that $f(t)$ has $-\infty$ entropy for all $t > 1$.
\end{remark}

\section{Local and global well-posedness of the coupled system} \label{wellposedness}

\subsection{Local well-posedness} \label{lwp}
Using the results in section \ref{Solsch}, we can prove the local existence of the system. We define the function space $\mathcal{X}$ as 
\begin{equation}
\mathcal{X} = L^\infty (0, T; \mathbb{P} W^{2,2} ) \cap L^2 (0, T; \mathbb{P} W^{3,2} ).
\end{equation}
For the subspace of $\mathcal{X}$ defined by
\begin{equation}
\tilde{\mathcal{X}} = \{ u \in \mathcal{X} \, : \, \partial_t u \in L^\infty (0, T; \nabla_x L^1 + L^2) \cap L^2 (0, T; \mathbb{P} W^{1,2} ) \}
\end{equation}
by Theorem \ref{momentsolutionexists} we know that there exists a unique moment solution to the Fokker-Planck equation (\ref{FPwithu}), denoted by $\mu$. Then we define
\begin{equation}
\sigma [u] = \int_{\mathbb{R}^2} m \otimes \nabla_m U \mu(dm).
\end{equation}
We set up a fixed point equation $u = F(u)$ in $\tilde{\mathcal{X}}$. We establish a contraction mapping in $\mathcal{X}$ and observe that if $u \in \tilde{\mathcal{X}}$ then $F(u) \in \tilde{\mathcal{X}}$ too. Following \cite{MR2989441}, our $F$ is defined as
\begin{equation}
F(u) = e^{\nu_1 t \Delta_x } u_0 + Q_1 (u,u) + L_1 (\sigma)
\end{equation}
where 
\begin{equation}
Q_1 (u, v) = -\int_0 ^t e^{\nu_1 (t-s) \Delta_x } \mathbb{P} (u(s) \cdot \nabla_x v(s) ) ds
\end{equation}
and
\begin{equation}
L_1 (\sigma) = K \int_0 ^t e^{\nu_! (t-s) \Delta_x } \mathbb{P} \left ( \nabla_x \cdot \sigma (s) \right ) ds. 
\end{equation}
We check that
\begin{equation}
\begin{gathered}
\norm{Q_1 (u, v) }_{\mathcal{X}} \le \delta \norm{u}_{\mathcal{X}} \norm{v}_{\mathcal{X}}, \\
\norm{L_1 (\sigma)}_{\mathcal{X}} \le C_1 \norm{\sigma}_{L^\infty (0, T; W^{1,2}) \cap L^2 (0, T; W^{2,2} ) }, \\
\norm{\sigma}_{L^\infty (0, T; W^{1,2}) \cap L^2 (0, T; W^{2,2} )}   \le C_2 C_3 ^{ \delta \norm{u}_{\mathcal{X}} ^2 }, \\
\end{gathered} \label{operatorbounds}
\end{equation}
where $\delta$ can be made as small as we want by making $T$ small.The first and second one can be found in \cite{MR2989441}, and the third one is a direct consequence of (\ref{L2H1M}). Using (\ref{operatorbounds}) we can find $A$ and $\delta$ (so we adjust $T$ too) such that if $\norm{u}_{\mathcal{X}} \le A$, then $\norm{F(u)}_{\mathcal{X}} \le A$. If $\norm{u}_{\mathcal{X}} \le A$, then we have
\begin{equation}
\norm{F(u)}_{\mathcal{X} } \le A_0 + \delta A^2 + C_1 C_2 C_3 ^{\delta A^2 },
\end{equation}
where $A_0$ depends only on initial data and $C_1, C_2, C_3$ are independent of $A$. For example, we can put $A = A_0 + 1 + C_1 C_2 C_3$ and choose $\delta$ small enough so that $\delta A^2 < 1$.  Also, by (\ref{differenceestimate}) we have
\begin{equation}
\norm{\sigma[u] - \sigma[v] }_{L^\infty (0, T; W^{1,2}) \cap L^2 (0, T; W^{2,2} )} \le C_4 \delta \norm{u-v}_{\mathcal{X} }
\end{equation} 
where $C_4 = C_4( A, A_0)$. Then 
\begin{equation}
\begin{gathered}
\norm{F(u) - F(v) }_{\mathcal{X}} \\
\le \norm{Q_1 (u, u-v ) }_{\mathcal{X}} + \norm{Q_1 (u-v, v)}_{\mathcal{X}} + \norm{L_1 (\sigma[u] - \sigma[v] ) }_{\mathcal{X}} \\
\le \delta (2A + C_1 C_4 ) \norm{u - v}_{\mathcal{X}}.
\end{gathered}
\end{equation}
Therefore, by choosing $\delta$ small enough again, we see that the sequence $u^{n+1} = F(u^n)$, $u^1$ be the solution of Navier-Stokes equation with initial data $u_0$ converges exponentially to the unique fixed point. Therefore, we have proved the following.
\begin{theorem}
Given $u_0 \in \mathbb{P}W^{2,2} $, $\mu_0$ satisfying (\ref{initpositivity}), (\ref{initmoment}), (\ref{initstress}), and (\ref{initentropy}), there is a $T_0 > 0$ such that there is a unique solution $(u, f)$ to (\ref{system}) for $t \in (0, T_0 )$ satisfying (\ref{velinitcond}) and $f$ is the unique moment solution of the Fokker-Planck equation with velocity field $u$. \label{localexistence}
\end{theorem}

\subsection{Global well-posedness} \label{gwp}
From this point, we investigate the global existence: we need to establish the bound
\begin{equation}
\begin{gathered}
\frac{1}{2} \norm{u}_{L^\infty (0, T; L^2 ) } ^2 + \sup_{0 \le t \le T} \frac{K}{2q(2q-1) } \norm{\sigma(t)}_{L^1} + \nu_1 \norm{\nabla_x u}_{L^2 (0, T; L^2 )} ^2 \\
\le A(\epsilon, q) \norm{M_{0,0} }[\mu_0 ] T + \frac{1}{2} \norm{u_0}_{L^2} + \frac{K}{2q(2q-1)} \norm{\sigma_0}_{L^1} =B_1 (T).
\end{gathered} \label{uL2sigmaL1}
\end{equation}
Here $B_1 (T)$ depends only on initial data and $T$. For this we come back to our approximating sequence $f^\alpha$: by multiplying $u$ to the first equation of (\ref{system}) and adding $C = \frac{K}{2q(2q-1)}$ times of (\ref{alpharadmoments}), and using the pointwise estimate $|m|^{2(q-1)} \le A + |m|^{4q-2}$ then integrating we obtain
\begin{equation}
\begin{gathered}
\frac{d}{dt} \left ( \frac{1}{2}  \norm{u(t) }_{L^2} ^2 + C \int \bar{M}_{2q} ^\alpha (t) dx \right ) + \nu_1 \norm{\nabla_x u (t) }_{L^2} ^2 \le CA \norm{M_{0,0}[\mu_0]}_{L^1} \\
+ \int \int |m|^{4q-2} (1 - \psi_\alpha ) f^\alpha dm dx  \\
+ \int \mathrm{Tr} \left ( \nabla_x u^\alpha \int |m|^{2(q-1)} m \otimes m \psi_\alpha f^\alpha dm - \nabla_x u \int |m|^{2(q-1)} m\otimes m f dm \right  ) dx.
\end{gathered}
\end{equation}
Then we have
\begin{equation}
\begin{gathered}
\norm{u}_{L^\infty (0, T; L^2)}^2 + C \norm{\bar{M}_{2q} ^\alpha}_{L^\infty (0, T; L^1) } + 2 \nu_1 \norm{\nabla_x u} ^2 _{L^2(0, T; L^2)} \le \norm{u_0}_{L^2} ^2 + C \norm{\bar{M}_{2q} ^\alpha (0)}_{L^1} + A T \norm{M_{0,0} [\mu_0]}_{L^1} \\
+ I_1 + I_2 + I_3 + I_4,
\end{gathered}
\end{equation}
where
\begin{equation}
\begin{gathered}
I_1 = \int _0 ^T \int \int |m|^{4q - 2} (1 - \psi_\alpha ) f^\alpha dm dx dt, \\
I_2 = \int_0 ^T \int \mathrm{Tr} \left ( (\nabla_x u^\alpha - \nabla_x u ) \int |m|^{2(q-1)} m \otimes m \psi_\alpha f^\alpha dm \right ) dx dt, \\
I_3 = \int_0 ^T \int \mathrm{Tr} \left ( \nabla_x u \int |m|^{2(q-1)} m \otimes m (\psi_\alpha - 1) f^\alpha dm \right ) dx dt, \\
I_4 = \int_0 ^T \int \mathrm{Tr} \left ( \nabla_x u \left ( \int |m|^{2(q-1)} m \otimes m f^\alpha dm - \int |m|^{2(q-1)} m \otimes m f dm \right ) \right ) dx dt.
\end{gathered}
\end{equation}
First we note that $\lim_{\alpha \rightarrow \infty} \norm{\bar{M}_{2q} ^\alpha}_{L^\infty (0, T; L^1)} \ge \norm{\bar{M}_{2q} [f] }_{L^\infty (0, T; L^1)}$. Then we note that for $k < 2q$
\begin{equation}
\int |m|^{2k} (1 - \psi_\alpha) f^\alpha dm \le \int_{|m| \ge \alpha} |m|^{2k} \left ( \frac{|m|}{\alpha} \right )^{4q - 2k} f^\alpha dm \le \frac{1}{\alpha} \int |m|^{4q} f^\alpha dm.
\end{equation}
Then we also note that $\int |m|^{4q} f^\alpha dm$ is uniformly bounded, say by $C$, in $L^\infty (0, T; L^1_ x ) $ by (\ref{LinfL1}). Therefore, we have $\lim_{\alpha} I_1 = \lim_{\alpha} I_3 = 0$. Then we note that $M_{a,b} ^\alpha  $ converges to $M_{a,b} [f]$ in weak* topology of $L^2 (0, T; L^2)$. Since $\nabla_x u \in L^2 (0, T; L^2)$, we see that $I_4 \rightarrow 0$ as $\alpha \rightarrow \infty$. Finally, we note that $\int |m|^{2q} \psi_\alpha f^\alpha dm$ is uniformly bounded in $L^\infty (0, T; L^2)$. Also, for each $t$, $\norm{\nabla_x u^\alpha (t) - \nabla_x u (t)}_{L^2} \rightarrow 0$ as $\alpha \rightarrow \infty$, so by dominated convergence in $t$ variable, we conclude that $\norm{\nabla_x u^\alpha - \nabla_x u}_{L^1 (0, T; L^2 ) } \rightarrow 0$. Therefore, $\lim_\alpha I_2 = 0$. In conclusion, we have
\begin{equation}
\begin{gathered}
\norm{u}_{L^\infty (0, T; L^2)}^2 + C \norm{\mathrm{Tr} \sigma}_{L^\infty (0, T; L^1) } + 2 \nu_1 \norm{\nabla_x u}_{L^2(0, T; L^2)} \le A T \norm{M_{0,0} [\mu_0]}_{L^1} ,
\end{gathered}
\end{equation}
and since $|\sigma_{12}| \le \frac{1}{2} \mathrm{Tr} (\sigma)$ we obtain (\ref{uL2sigmaL1}). From (\ref{L2H1M2}) we see that
\begin{equation}
\norm{\sigma}_{L^\infty (0, T; L^2 )} ^2 + \nu_2 \norm{\nabla_x \sigma }_{L^2 (0, T; L^2 ) } ^2 \le B_2 (T) \label{B2_1}
\end{equation}
where again $B_2 (T) = C(q)^{T + B_1(T)} \norm{\bar{M}[\mu_0]_{2q}}_{L^2} $ depends only on initial data and $T$. Then we take curl to the first equation of the (\ref{system}) to get vorticity equation: for $\omega = \nabla_x ^\perp \cdot u$
\begin{equation}
\partial_t \omega + u \cdot \nabla_x \omega = \nu_1 \Delta_x \omega +  K \nabla_x ^\perp \cdot \nabla_x \cdot \sigma. \label{vorticity}
\end{equation}
Multiplying $\omega$ to (\ref{vorticity}) and integrating, we obtain
\begin{equation}
\norm{\omega}_{L^\infty(0, T; L^2) } ^2 + \nu_1 \norm{\nabla_x \omega}_{L^2 (0, T; L^2) } \le C(\nu_1) \norm{\nabla_x \sigma}_{L^2(0, T; L^2) } ^2 = C B_2 (T). \label{B2_2}
\end{equation}
Then by (\ref{L2H2qM}) we have 
\begin{equation}
\norm{\sigma}_{L^\infty (0, T; W^{1,2} ) } ^2 + \nu_2 \norm{\sigma}_{L^2  (0, T; W^{2,2} ) } ^2 \le B_3 (T) \label{B3}
\end{equation}
where $B_3 (T) =  C(\epsilon, \nu_2, q, K) ^{T + CB_2 (T) T + B_2 (T)  \sqrt{T} }$ again depends only on initial data and $T$. Finally, by multiplying $- \Delta_x \omega$ to (\ref{vorticity}) and integrating, we have
\begin{equation}
\begin{gathered}
\norm{\nabla_x \omega}_{L^\infty (0, T; L^2) } ^2 + \norm{\Delta_x \omega}_{L^2 (0, T ; L^2 ) } ^2 \le \exp \left (C \int_0 ^T \norm{u(t)}_{L^2} ^2 \norm{\omega (t) }_{L^2} ^2 dt \right ) \\
\left ( \norm{\nabla_x \omega(0)}_{L^2} ^2 + C(K, \nu_1) \norm{\Delta_x \sigma}_{L^2 (0, T; L^2)} ^2 \right ) \\
\le \exp \left (C B_1 (T) B_2 (T) \right ) \left ( \norm{\nabla_x \omega(0)}_{L^2} ^2 + C(K, \nu_1) B_3 (T) \right ) =B_4 (T).
\end{gathered} \label{B4}
\end{equation}
Therefore, we see that
\begin{equation}
\norm{u}_{\mathcal{X} } \le B_1 + CB_2 + B_4 = B_5, \label{B5}
\end{equation}
which only depends on initial data and $T$. Thus, we have the global existence, following the proof of \cite{MR2989441}. Theorem \ref{localexistence} guarantees that there is $T_0 > 0$ such that the solution exists for $[0, T_0]$. We consider the maximal interval of existence: $T_1 = \sup T_0 \le T$ such that the solution exists for $[0, T_0]$. Then it must be that $T_1 = T$, because otherwise we could extend the solution beyond $T_1$.
\begin{theorem}
Given $u_0 \in \mathbb{P}W^{2,2} $, $\mu_0$ satisfying (\ref{initpositivity}), (\ref{initmoment}), (\ref{initstress}), and (\ref{initentropy}), and arbitrary $T > 0$, there is a unique solution $(u, f)$ to (\ref{system}) for $t \in (0, T )$ satisfying (\ref{velinitcond}) and $f$ is the unique moment solution of the Fokker-Planck equation with velocity field $u$. In addition, the bounds (\ref{uL2sigmaL1}), (\ref{B2_1}), (\ref{B2_2}), (\ref{B3}), (\ref{B4}), and (\ref{B5}) are satisfied.
\label{globalexistence}
\end{theorem}
\begin{remark}
In fact, local Lipschitz dependence of solution on the initial data can be proved with similar standard energy estimates in this subsection, together with (\ref{diffdiffinitialdata}). That is, if $u_0, v_0 \in \mathbb{P}W^{2,2}$ and $\mu_0, \nu_0$ satisfy (\ref{initpositivity}), (\ref{initmoment}), (\ref{initstress}), and (\ref{initentropy}), then 
\begin{equation}
\norm{u-v}_{\mathcal{X}} \le C (u_0, v_0, \mu_0, \nu_0 , T ) \left ( \norm{u_0 - v_0 }_{\mathbb{P}W^{2,2} } + \norm{\sigma[\mu_0] - \sigma[\nu_0 ] }_{W^{1,2} } + \sum_{k=0} ^{2q-1} \norm{\bar{M}_{2k} [\mu_0 - \nu_0 ] }_{L^2} \right ).
\end{equation}
\end{remark}
\begin{corollary}
Suppose that $q=1$ in the system (\ref{system}), in other words, $U(m) = |m|^2$. Suppose that the initial data $u_0, \mu_0$ satisfies conditions $u_0 \in \mathbb{P}W^{2,2}$, (\ref{initpositivity}), (\ref{initmoment}), (\ref{initstress}), and (\ref{initentropy}), and (\ref{initdenentropy}). Then $$(u, \sigma, \rho) = (u, \int m \otimes \nabla_m U f dm, M_{0,0} [f])$$ is the unique strong solution for the diffusive Oldroyd-B equation
\begin{equation}
\begin{gathered}
\partial_t u + u \cdot \nabla_x u = - \nabla_x p +  \nu_1 \Delta_x u + K \nabla_x \cdot \sigma, \\
\nabla_x \cdot u =0, \\
\partial_t \sigma + u \cdot \nabla_x \sigma = (\nabla_x u ) \sigma + \sigma (\nabla_x u )^T - 2 \epsilon \sigma + 2 \epsilon \rho \mathbb{I} + \nu_2 \Delta_x \sigma, \\
\partial_t \rho + u \cdot \nabla_x \rho = \nu_2 \Delta_x \rho, \\
u(0) = u_0, \sigma(0) = \int m \otimes \nabla_m U \mu_0 dm, \rho(0) = M_{0,0} [\mu_0 ].
\end{gathered}
\end{equation}
\end{corollary}
\begin{proof}
It is a consequence of Lemma \ref{momentofmomentsols} and Theorem \ref{globalexistence}. Although $\sigma$ is a weak solution of the corresponding equation of (\ref{momentevolution}), it has enough regularity to perform integration by parts, so in fact it is a strong solution. By the uniqueness of diffusive Oldroyd-B system (\cite{MR2989441}), it is the unique solution.
\end{proof}

\subsection{Free energy bound} \label{freeEbound}
In this section, we prove the estimates (\ref{entropyestimw}) and (\ref{entropyestms}). For this purpose, we briefly review the proof of Theorem \ref{entropycon}. We follow the proof in \cite{MR3443169}.
\begin{proof}[proof of Theorem \ref{entropycon} ]
For simplicity, we assume that $a^{ij} (x, t) = a^{ij}$ for some constant, positive definite matrix $(a^{ij})_{ij}$. We use the following simple observation: given two nonnegative functions $f_1, f_2 \in L^1 (\mathbb{R}^d)$, for every measurable function $\psi$ with the property that $|\psi|^2 f_1 \in L^1 (\mathbb{R}^d)$ we have
\begin{equation}
\int_{\mathbb{R}^d} \frac{| (\psi f_1 ) * f_2 |^2 }{f_1 * f_2 } dx \le \int_{\mathbb{R}^d} |\psi|^2 f_1 dx \int_{\mathbb{R}^d} f_2 dx, \label{convolid}
\end{equation}
where $\frac{| (\psi f_1 ) * f_2 (x) |^2 }{f_1 * f_2 (x) } := 0$ if $f_1 * f_2 (x) = 0$. Also we set $$\rho * \omega_\epsilon (x, t) : = \int_{\mathbb{R}^d} \omega_\epsilon (x-y) \rho (y, t) dy,$$ where $\omega_\epsilon (x) = \epsilon ^{-d} g\left (\frac{x}{\epsilon } \right )$ where $g$ is the standard Gaussian and $\epsilon \in (0, T)$. Then $\mu = \rho dxdt$ and in the Sobolev sense
\begin{equation}
\partial_t (\rho * \omega_\epsilon ) = (a^{ij} \rho ) * (\partial_{x_i} \partial_{x_j} \omega_\epsilon ) - (b^i \rho ) * \partial_{x_i} \omega_\epsilon. \label{mollified}
\end{equation}
We have the following version of $\rho * \omega_\epsilon$ defined by the formula
\begin{equation}
\rho * \omega_\epsilon (x, t) := \rho * \omega_\epsilon (x, 0) + \int_0 ^t v(x,s) ds \label{version}
\end{equation}
where $v$ is the right side of (\ref{mollified}). One can readily check that this version is absolutely continuous in $t$ on $[0, T]$ and belongs to the class $C_b ^\infty (\mathbb{R}^d )$ in $x$, and for almost every $t$, including $t=0$, this version coincides for all $x$ with the original version defined by convolution. This version is bounded pointwise by $\epsilon^{-d}$, for all $(x,t) \in \mathbb{R}^d \times [0, T]$. We set
\begin{equation}
\rho_\epsilon := \rho * \omega_\epsilon, \,\, f_\epsilon (x, t) := \rho_\epsilon (x, t) + \epsilon \max (1, |x| )^{-(d+1) },
\end{equation}
where $\rho * \omega_\epsilon$ should be understood as the version (\ref{version}) and by $\max (1, |x| )^{-(d+1) }$ we mean, again by a slight abuse of notation, a smooth, bounded function equals it for $|x| > 2$. Since the function $\rho \Lambda$ is integrable, there is $\tau$ as close to $T$ as we wish such that
\begin{equation}
\int_{\mathbb{R}^d} \rho (x, \tau) \Lambda(x) dx < \infty,
\end{equation}
and for every $\epsilon = \frac{1}{n}$ our version of $\rho_\epsilon (x, \tau)$ coincides with the function $\rho (\cdot, \tau) * \omega_\epsilon (x)$ for all $x$. Then by inequality $$ \log \max (|x+y| , 1) \le \log \max (|x|, 1) + |y|$$ gives the estimate
\begin{equation}
\int_{\mathbb{R}^d} f_\epsilon (x, \tau) \Lambda(x) dx \le \int_{\mathbb{R}^d} \rho (x, \tau) \Lambda (x) dx + \epsilon M_1, \label{Lambdaestim}
\end{equation}
where $M_1$ is a constant independent of $\epsilon$. Then by (\ref{mollified}), we have
\begin{equation}
\begin{gathered}
\int_0 ^\tau \int_{\mathbb{R}^d} \partial_t (\rho * \omega_\epsilon ) \log f_\epsilon dx dt = \int_0 ^\tau \int_{\mathbb{R}^d} \left ( a^{ij} \left ( \rho * \partial_{x_i} \partial_{x_j} \omega_\epsilon \right ) - \left ( b^i \rho \right ) * \partial_{x_i} \omega_\epsilon \right ) \log f_\epsilon dxdt,
\end{gathered} \label{evolconvol1}
\end{equation}
and by $|\log f_\epsilon | \le C \left ( \log \frac{1}{\epsilon} + 1  + \Lambda \right )$, (\ref{convolid}), $|b| \in L^2(\mu)$, and the estimate $$|\log \max (|x+y|, 1) |^2 \le 4 + 2 | \log \max (|x|, 1) |^2 + 2 | \log \max (|y|, 1 ) | ^2 $$ the integrand of the right side of (\ref{evolconvol1}) is integrable in $\mathbb{R}^d \times (0, T)$. Furthermore, one can observe that one can integrate by parts of the right side of (\ref{evolconvol1}) using the similar argument: therefore we get
\begin{equation}
\begin{gathered}
\int_0 ^\tau \int_{\mathbb{R}^d} \partial_t \rho_\epsilon \log f_\epsilon dxdt \\
= - \int_0 ^ \tau \int_{\mathbb{R}^d} \frac{ \partial_{x_i} f_\epsilon }{f_\epsilon} \left ( a^{ij}  \partial_{x_j} \left ( \rho * \omega_\epsilon \right ) -(b^i \rho) * \omega_\epsilon \right ) dx dt .
\end{gathered} \label{evolconvol2}
\end{equation}
The integrand of left side of (\ref{evolconvol2}) can be written as $\partial_t (f_\epsilon \log f_\epsilon ) - \partial_t \rho_\epsilon$, and since $\rho_t$ are probability measures, the left side of (\ref{evolconvol2}) equals
\begin{equation}
L_\epsilon := \int_{\mathbb{R}^d} \left ( f_\epsilon (x, \tau) \log f_\epsilon (x, \tau) - f_\epsilon (x, 0) \log f_\epsilon (x, 0) \right ) dx.
\end{equation}
By (\ref{Lambdaestim}) and $|\log f_\epsilon | \le C\left (\log \left (\frac{1}{\epsilon } \right ) + 1 + \Lambda \right  )$ we have $f_\epsilon (\cdot ,\tau) \log f_\epsilon (\cdot ,\tau ) \in L^1 (\mathbb{R}^d )$ and similarly $f_\epsilon (\cdot ,0) \log f_\epsilon (\cdot, 0) \in L^1 (\mathbb{R}^d )$. By Jensen's inequality applied to $\Phi(s) = s \log s$, we have
\begin{equation}
\begin{gathered}
\int_{\mathbb{R}^d} f_\epsilon (x,0) \log f_\epsilon (x,0) dx \\
 \le \lambda \int \Phi \left ( \frac{\rho_\epsilon}{\lambda}  \right ) dx  + (1 - \lambda) \int \Phi \left ( \frac{\epsilon}{1-\lambda} \max(|x|,1) ^{-(d+1) } \right ) dx \\
\le \int \rho_0 \log \rho_0 dx + \log \frac{1}{\lambda} + \epsilon \int \frac{1}{1 - \lambda} \max(|x|, 1) ^{-(d+1)} dx
\end{gathered}
\end{equation}
for any $\lambda \in (0, 1)$. On the other hand, by Csisz\'{a}r-Kullback-Pinsker inequality (\cite{MR3497125})
\begin{equation}
\begin{gathered}
\int f \log f - f \log g - f + g dx \ge \frac{1}{2} \norm{f - g}_{L^1} ^2, \mathrm{where} \,\, f, g \in L^1, f \ge 0, g > 0, \int f = \int g = 1
\end{gathered}
\end{equation}
with $f = \frac{1}{\norm{f_\epsilon}_{L^1}  } f_\epsilon = \frac{1}{1+\epsilon C} f_\epsilon $ and $g = \frac{1}{\norm{\max(|x|, 1)^{-(d+1)} }_{L^1 } } \max(|x|, 1)^{-(d+1)} = C \max(|x|, 1)^{-(d+1)}$, we have
\begin{equation}
\begin{gathered}
\int f_\epsilon (x, \tau) \log f_\epsilon (x, \tau) dx \ge (1+ C\epsilon) \log C(1 + C\epsilon ) - (d+1) \int f_\epsilon (x, \tau) \Lambda (x) dx \\
\ge -(d+1) \int \rho (x, \tau) \Lambda(x) dx + o(\epsilon).
\end{gathered}
\end{equation}
From (\ref{evolconvol2}) we obtain
\begin{equation}
\begin{gathered}
\int_0 ^\tau \int_{\mathbb{R}^d} a^{ij} \frac{\partial_{x_i} f_\epsilon}{f_\epsilon} \partial_{x_j} f_\epsilon dxdt = \int_0 ^\tau \int_{\mathbb{R}^d} \frac{\partial_{x_i} f_\epsilon}{f_\epsilon} \left ( (b^i \rho ) * \omega_\epsilon + \epsilon a^{ij} \partial_{x_j} \max(|x|, 1) ^{-(d+1) } \right ) dxdt - L_\epsilon
\end{gathered}
\end{equation}
and the right side in this inequality is bounded by 
\begin{equation}
\begin{gathered}
\left ( \int_0 ^\tau \int_{\mathbb{R}^d} \frac{|\nabla f_\epsilon |^2}{f_\epsilon} dxdt  \right ) ^{\frac{1}{2}} \left ( \norm{b}_{L^2 (\mu ) } + o(\epsilon ) \right ) \\
+ (d+1) \int \rho(x, \tau) \Lambda(x) dx + o(\epsilon) + \int \rho_0 \log \rho_0 dx - \log \lambda + \frac{o(\epsilon)}{1-\lambda}.  
\end{gathered}
\end{equation}
Using $A \ge m\mathbb{I}$, taking $\epsilon \rightarrow 0$, using Fatou's lemma, and putting $\lambda \rightarrow 1$ we get
\begin{equation}
m^2 \int_0 ^\tau \int_{\mathbb{R}^d} \frac{|\nabla \rho|^2}{\rho} dxdt \le \left (\norm{b}_{L^2 (\mu )} ^2 + (d+1) \int \rho(x, \tau) \Lambda(x) dx + \int \rho_0 \log \rho_0 dx \right )
\end{equation}
as desired.
\end{proof}
To prove entropy estimate, we start from (\ref{evolconvol2}) and $L_\epsilon$: first we prove that as $\epsilon \rightarrow 0$, $$ \int_0 ^\tau \int_{\mathbb{R}^d} \frac{\partial_{x_i} f_\epsilon}{f_\epsilon} \partial_{x_j}  (\rho * \omega_\epsilon ) dxdt \rightarrow \int_0 ^\tau \int \frac{ \partial_{x_i} \rho \partial_{x_j} \rho }{\rho} dxdt. $$ We begin with observing that $\frac{\partial_{x_i} f_\epsilon}{f_\epsilon} \partial_{x_j}  (\rho * \omega_\epsilon )$ is bounded by 
\begin{equation}
q_\epsilon =  \frac{|\nabla \rho * \omega_\epsilon | ^2 }{\rho_\epsilon} + C \epsilon \frac{1}{f_\epsilon ^{\frac{1}{2} } }  \max (|x|, 1) ^{-(d+2)} \frac { |\nabla \rho * \omega_\epsilon | }{\left (\rho * \omega_\epsilon \right ) ^{\frac{1}{2} } }.
\end{equation}
For almost all $t \in (0, \tau)$, $ \int_{\mathbb{R}^d} \frac{ |\nabla \rho (x, t) |^2}{\rho (x, t) } dx < \infty$. Therefore, for such $t$, by (\ref{convolid}) we see that $$ \int \frac{|\nabla \rho * \omega_\epsilon (t)|^2 } { \rho * \omega_\epsilon (t) } dx \le \int_{\mathbb{R}^d} \frac{ |\nabla \rho (x, t) |^2}{\rho (x, t) } dx, $$ and so for integral over $t$ too. Then the integral over $\mathbb{R}^d \times [0, \tau)$ of the second term is also bounded by $\sqrt{\epsilon} C$, where $C$ is independent of $\epsilon$. Therefore, we have $$ \lim \sup _{\epsilon \rightarrow 0} \int_0 ^\tau \int q_\epsilon (x, t) dx dt  \le \int_0 ^\tau \int_{\mathbb{R}^d} \frac{ |\nabla \rho (x, t) |^2}{\rho (x, t) } dx dt. $$  On the other hand, note that $q_\epsilon (x, t) \rightarrow \frac{|\nabla \rho (x, t) |^2 } {\rho (x, t) }$ for almost all $(x, t) \in \mathbb{R}^d \times [0, \tau) $, at least for a subsequence of $\epsilon = \frac{1}{n}$ because we have $L^1 (x, t)$ convergence of approximate identity in $x$ variable. Therefore, by Fatou's lemma we have $$ \int_0 ^\tau \int_{\mathbb{R}^d} \frac{ |\nabla \rho (x, t) |^2}{\rho (x, t) } dx \le \lim \inf \int_0 ^\tau \int q_\epsilon (x, t) dx dt. $$ Therefore, we see that $\frac{\partial_{x_i} f_\epsilon}{f_\epsilon} \partial_{x_j}  (\rho * \omega_\epsilon )$ is bounded by $q_\epsilon (x, t)$ pointwise, which is integrable and converges to $ \frac{ |\nabla \rho (x, t) |^2}{\rho (x, t) }$ pointwise, and its integral also converges to the integral of the limit. Therefore, by generalized dominated convergence, we prove the claim. In a similar manner, we see that $$ \int_0 ^\tau \int_{\mathbb{R}^d} \frac { | (b^i \rho) * \omega_\epsilon |^2 } {\rho * \omega_\epsilon } dx dt \rightarrow \int_0 ^\tau \int_{\mathbb{R}^d} \frac { | (b^i \rho) |^2} {\rho} dxdt. $$ Then again $\frac{\partial_{x_i} f_\epsilon } {f_\epsilon} (b^i \rho) *\omega_\epsilon$ is bounded by 
\begin{equation}
q_\epsilon ' = \left ( \frac{|\nabla \rho * \omega_\epsilon| }{\sqrt{\rho_\epsilon } } + C \epsilon \frac{1}{f_\epsilon ^{\frac{1}{2} } } \max (|x|, 1) ^{-(d+2)} \right ) \frac{| (b^i \rho) * \omega_\epsilon | }{\sqrt{\rho_\epsilon } },
\end{equation}
and we can again use generalized dominated convergence to conclude that
\begin{equation}
\int_0 ^\tau \int_{\mathbb{R}^d} \frac{\partial_{x_i} f_\epsilon } {f_\epsilon} (b^i \rho) *\omega_\epsilon dxdt \rightarrow \int_0 ^\tau \int_{\mathbb{R}^d} b^i \partial_{x_i} \rho dx dt.
\end{equation}
On the other hand, to control $L_\epsilon$ term we observe that $\Psi (x) = x \log x - x + 1 \ge 0$ for all $x \ge 0$: then for $g = \max (|x|, 1) ^{-(d+1) }$ by Fatou we have
\begin{equation}
\begin{gathered}
\int \rho \log \rho (\tau) dx + (d+1) \int \rho (\tau) \Lambda dx -1 + \int g dx = \int \Psi \left  ( \frac{\rho (\tau) }{g} \right ) g dx \\
 \le \lim \inf_{\epsilon \rightarrow 0} \int \Psi \left ( \frac{f_\epsilon}{g} \right ) g dx = \lim \inf \int f_\epsilon \log f_\epsilon (\tau) dx + (d+1) \int \rho (\tau) \Lambda dx -1 + \int g dx.
\end{gathered}
\end{equation}
Here we used that $\int f_\epsilon (\tau) \Lambda dx \rightarrow \int \rho (\tau) \Lambda dx$, which comes from (\ref{Lambdaestim}) and Fatou. Therefore, by taking $\epsilon \rightarrow 0$ to (\ref{evolconvol2}) we get
\begin{equation}
\int \rho \log \rho (\tau) dx + \int_0 ^\tau \int_{\mathbb{R}^d} \frac{a^{ij} \partial_{x_i} \rho \partial_{x_j} \rho}{\rho} dxdt - \int_0 ^\tau \int_{\mathbb{R}^d} b^i \partial_{x_i} \rho dxdt \le \int \rho_0 \log \rho_0 dx. \label{entropyineq}
\end{equation}
Applying (\ref{entropyineq}) to our equation, and applying integration by parts to $b^i \partial_{x_i} \rho$, which is possible since $b^^i \rho, (\partial_{x_i} b^i ) \rho \in L^1$, we get
\begin{equation}
\begin{gathered}
\int f(\tau) \log f(\tau) dmdx + \int_0 ^\tau \int \nu_2 \frac{|\nabla_x f|^2}{f} + \epsilon \frac{|\nabla_m f|^2}{f} dmdx dt \\
- \epsilon \int_0 ^\tau \int \Delta_m U f dmdxdt \le \int f_0 \log f_0 dmdx.
\end{gathered} \label{Entropypart}
\end{equation}
On the other hand, applying similar argument as (\ref{uL2sigmaL1}), we have
\begin{equation}
\begin{gathered}
\int \bar{M}_{2q} ^\alpha (\tau) dx + \epsilon (2q) ^2 \int_0 ^\tau \int \bar{M}_{4q-2} ^\alpha dx dt \\
 = \int \bar{M}_{2q} ^\alpha (0) dx +  \int _0 ^\tau \int \mathrm{Tr} ( (\nabla_x u ) \sigma ) dx dt + \epsilon (2q) ^2 \int_0 ^\tau \int \bar{M}_{2(q-1) } ^\alpha dx dt + I_\alpha
 \end{gathered} \label{mollifiedTracestress}
\end{equation}
where $I_\alpha \rightarrow 0$ as $\alpha \rightarrow \infty$. Note that we know, by weak convergence, $$ \int \bar{M}_{2q} [f] (\tau) dx + \epsilon (2q)^2 \int_0 ^\tau \int \bar{M}_{4q - 2} [f] dxdt $$ does not exceed the limit inferior of the left side of (\ref{mollifiedTracestress}). On the other hand, we need $$ \int_0 ^\tau \int \bar{M}_{2(q-1)} ^\alpha dxdt \rightarrow  \int_0 ^\tau \int \bar{M}_{2(q-1)} [f] dxdt,  $$ which can be obtained by the following: since $$ \int \Lambda(x) |m|^{2(q-1)} f^\alpha dmdx \le \int \left ( |\Lambda| ^2 + |m|^{4(q-1) } \right ) f^\alpha dmdx $$ we see, from bounds in section \ref{unifBdmom} and section \ref{Flfielddep} we note that $\int_0 ^\tau \int \Lambda(x) \bar{M}_{2(q-1)} ^\alpha dx dt$ is bounded by some constant $C$ depending on initial data $\mu_0$ and $u$, uniform in $\alpha$. Therefore, for any $R > 1$, we have $$ \int_0 ^\tau \int_{|x| > R} \bar{M}_{2(q-1) }^\alpha dxdt \le \frac{C}{\log R}.$$ On the other hand, we note that $$ |\nabla_x \bar{M}_{2(q-1) } ^\alpha | \le \bar{M}_{4(q-1) } ^\alpha + \int \frac{|\nabla_x f^\alpha | ^2 } {f^\alpha} dm, $$ and bounds in section \ref{unifBdmom} and section \ref{Flfielddep} gives that $\norm{\nabla_x \bar{M}_{2(q-1) } ^\alpha }_{L^1 (0, T; L^1) } $ is uniformly bounded in $\alpha$. Also, by (\ref{alphamoments}) we can see that $\partial_t \bar{M}_{2(q-1) } ^\alpha \in L^1 (0, T; W^{-1, 1} )$  is uniformly bounded in $\alpha$: for terms involving velocity fields, one can use $L^2$ bounds on moments, and for plain moment terms one note that the highest moment in that equation has degree $4(q-1)$, and it has bound in $L^\infty (0, T; L^1)$, which is uniform in $\alpha$. Then for any $B(0, R)$, $W^{1,1} (B(0, R) ) \subset L^1 (B (0, R) )$ compactly by Rellich-Kondrachov, and $L^1 (B(0, R) ) \subset W^{-1, 1} (B(0, R) )$ by Morrey-Sobolev embedding $W^{1, q'} \subset L^\infty $ for $q < \frac{d}{d-1} $. Therefore, by Aubin-Lions, by applying some cutoff function if necessary, we have $$ \lim_{\alpha \rightarrow \infty} \int_0 ^\tau \int_{B(0, R) } \bar{M}_{2(q-1) } ^\alpha dx dt = \int_0 ^\tau \int_{B(0, R) } \bar{M}_{2(q-1) } [f] dx dt .$$ To summarize, we have
\begin{equation}
\begin{gathered}
\int \bar{M}_{2q} [f] (\tau) dx + \epsilon (2q) ^2 \int_0 ^\tau \int \bar{M}_{4q-2} [f] dxdt  \\
\le \int \bar{M}_{2q} [f_0] dx +  \int_0 ^\tau \int \mathrm{Tr} ( (\nabla_x u ) \sigma ) dx dt + \epsilon (2q) ^2 \int_0 ^\tau \int \bar{M}_{2(q-1) } [f]  dx dt,
\end{gathered} 
\end{equation}
or, in other words,
\begin{equation}
\begin{gathered}
- \int \log \left ( e^{-U(m) } \right ) f (\tau) dmdx + \epsilon \int_0 ^\tau \int |\nabla_m U |^2 f dmdxdt \\
- \epsilon \int_0 ^\tau \int \Delta_m U f dmdx dt \le - \int \log \left ( e^{-U(m) } \right ) f_0 dmdx +   \int_0 ^\tau \int \mathrm{Tr} ( (\nabla_x u ) \sigma ) dx dt.
\end{gathered}\label{Tracepart}
\end{equation}
Note that we can apply integration by parts to the term $\int_0 ^\tau \int \Delta_m f dmdxdt$: since $| \nabla_m U \nabla_m f | \le \frac{|\nabla_m f|^2}{f} + |\nabla_m U|^2 f$ so it is integrable in $L^1 ([0, T] \times \mathbb{R}^{2+2} )$, and $\nabla_m U f$ is also integrable. Therefore, by adding (\ref{Entropypart}) and (\ref{Tracepart}), and adding the velocity part we get
\begin{equation}
\begin{gathered}
\int f(\tau) \log \frac{f(\tau) }{e^{-U}/Z } dmdx + \epsilon \int_0 ^\tau \int f \left | \nabla_m \log \left ( \frac{f} {e^{-U}/Z  } \right ) \right | ^2 dm dx dt \\
+ \nu_2 \int_0 ^\tau \int f \left | \nabla_x \log \left ( \frac{f} {e^{-U}/Z  } \right ) \right | ^2 dm dx dt + \frac{1}{K} \norm{u(\tau) }_{L^2} ^2 + \frac{\nu_1} {K} \int_0 ^\tau \norm{\nabla_x u }_{L^2} ^2 dxdt \\
\le \norm{u_0}_{L^2} ^2 + \int f_0 \log \frac{f_0 }{e^{-U}/Z } dmdx
\end{gathered} \label{entropyestimw}
\end{equation}
where $Z = \int e^{-U} dm$.
On the other hand, suppose that
\begin{equation}
\int M_{0,0} [f_0] \log \left ( M_{0,0} [f_0] \right ) dx < \infty. \label{initdenentropy}
\end{equation}
Using the same technique as before, we can show that
\begin{equation}
\begin{gathered}
\lim_{\epsilon \rightarrow 0}  \int_0 ^\tau \int \nu_2 \frac{ | \nabla_x M_{0,0} ^\epsilon |^2 } {M_{0,0} ^\epsilon } dx dt = \int_0 ^\tau  \int \nu_2 \frac{ | \nabla_x M_{0,0} |^2 } {M_{0,0} } dx dt,\\
\lim_{\epsilon \rightarrow 0} \int M_{0,0} ^\epsilon (x, 0) \log  M_{0,0} ^\epsilon (x, 0) dx = \int M_{0,0}(x, 0) \log M_{0,0} (x, 0) dx.
\end{gathered}
\end{equation} 
where $M_{0,0} ^\epsilon = M_{0,0} * \omega_\epsilon + \epsilon \max (|x|, 1) ^{-3}$ and the remaining task is to show $$\lim\inf_{\epsilon \rightarrow 0} \int M_{0,0} ^\epsilon (x, \tau) \log  M_{0,0} ^\epsilon (x, \tau) dx = \int M_{0,0}(x, \tau) \log M_{0,0} (x, \tau) dx.$$ For this we recall the following fact about Fatou from \cite{MR1817225}, which comes from Br\'{e}zis-Lieb inequality (\cite{MR699419}) : if $\{h_n \}$ is a sequence of nonnegative functions, converging almost everywhere to $h$, and $\int h_n$ is uniformly bounded, then
\begin{equation}
\lim \inf_n \int |h_n - h| + \int h = \lim \inf _n \int h_n.
\end{equation} 
We apply this to $\Psi \left (\frac{M_{0,0} ^\epsilon}{g} \right ) g \ge 0$, where as before $\Psi(s) = s \log s - s + 1 $ and $g (x) = \max (|x|, 1) ^{-3} $. We know that for $f \ge 0, f \in L^1 \cap L^2, \int f \Lambda < \infty$, we have a pointwise estimate
\begin{equation}
f | \log f | \le C f \Lambda + C g + |f|^2, \label{pointwiseflogf}
\end{equation}
where the first term corresponds to the case $g(x) ^2 \le f(x) \le 1$, the second term corresponds to the case $0 \le f \le g(x) ^2$, and the last term corresponds to the case $f(x) > 1$. Therefore, 
\begin{equation}
\int_{\mathbb{R}^2} \Psi \left (\frac{M_{0,0} ^\epsilon}{g} \right ) g dx = \int M_{0,0} ^\epsilon \log M_{0,0} ^\epsilon + 3 M_{0,0}^\epsilon \Lambda - M_{0,0}^\epsilon + g dx
\end{equation}
so by (\ref{pointwiseflogf}) and (\ref{Lambdaestim}) they are uniformly bounded in $\epsilon$. Thus it suffices to show $$ \int \left | \Psi \left (\frac{M_{0,0} ^\epsilon}{g} \right ) - \Psi \left (\frac{M_{0,0}}{g} \right ) \right | g dx \rightarrow 0. $$ But this term is bounded by
\begin{equation}
\int | M_{0,0} ^\epsilon \log M_{0,0} ^\epsilon - M_{0,0} \log M_{0,0} | + |M_{0,0} ^\epsilon - M_{0,0} | ( \Lambda + 1) dx, 
\end{equation}
which converges to 0 by the pointwise estimate (\ref{pointwiseflogf}) and generalized dominated convergence theorem. Therefore, we have
\begin{equation}
\int M_{0,0} (\tau) \log M_{0,0} (\tau) dx + \nu_2 \int_0 ^\tau \int \frac{|\nabla_x M_{0,0} |^2} {M_{0,0}} dxdt = \int M_{0,0} [f_0] \log M_{0,0} [f_0] dx. \label{relentropypart}
\end{equation}
Noting that $$\frac{|\nabla_x M_{0,0} |^2} {M_{0,0}} = M_{0,0} \left | \nabla_x \left ( \log M_{0,0} \right ) \right |^2 = \int f \left | \nabla_x \left ( \log M_{0,0} \right ) \right |^2 dm$$ and by subtracting (\ref{relentropypart}) to (\ref{entropyestimw}) we get
\begin{equation}
\begin{gathered}
\int f(\tau) \log \frac{f(\tau) }{ M_{0,0} [f(\tau) ] e^{-U}/Z } dmdx + \epsilon \int_0 ^\tau \int f \left | \nabla_m \log \left ( \frac{f} { M_{0,0} [f] e^{-U}/Z  } \right ) \right | ^2 dm dx dt \\
+ \nu_2 \int_0 ^\tau \int f \left | \nabla_x \log \left ( \frac{f} {M_{0,0} [f] e^{-U}/Z  } \right ) \right | ^2 dm dx dt + \frac{1}{K} \norm{u(\tau) }_{L^2} ^2 + \frac{\nu_1} {K} \int_0 ^\tau \norm{\nabla_x u }_{L^2} ^2 dxdt \\
\le \norm{u_0}_{L^2} ^2 + \int f_0 \log \frac{f_0 }{M_{0,0} [f_0] e^{-U}/Z } dmdx.
\end{gathered} \label{entropyestms}
\end{equation}
Therefore, we have proved the following.
\begin{theorem}\label{Freeenergybound}
If the system (\ref{system}) has initial data satisfying $u_0 \in \mathbb{P}W^{2,2}$, (\ref{initpositivity}), (\ref{initmoment}), (\ref{initstress}), and (\ref{initentropy}), then for almost all $\tau \in (0, +\infty)$ (\ref{entropyestimw}) holds. If in addition (\ref{initdenentropy}) holds, then (\ref{entropyestms}) also holds for almost all $\tau \in (0, +\infty)$.
\end{theorem}

\section{Conclusion}
We proved global regularity of the 2D incompressible Navier-Stokes equation coupled with diffusive Fokker-Planck equation, for a large class of data. We defined the class of data by the size of macroscopic variables, and this newly proposed class has some advantages over previously used ones. In addition, we proved that the free energy of the system does not increase over time. To prove the result, we defined the moment solution, which is a weak solution with controllable moments. 
\paragraph{Acknowledgements.} The author is supported by Samsung scholarship. The author wants to express his deep gratitude to Prof. Peter Constantin for his kind support, encouragement and insightful discussions with him. The author also appreciates his valuable suggestions, which critically helped the author to apply the notion of moment solutions to models other than Hookean models. The author is also very grateful to Prof. Charles Fefferman for his kind support and encouragement. The author thanks Prof. Vlad Vicol for his support and a lot of helpful advice. The author is grateful to Prof. Elliot Lieb for helpful discussions. Discussions with Huy Nguyen, Federico Pasqualotto, and Theo Drivas were really helpful. In-Jee Jeong and Sung-Jin Oh also gave a lot of useful advice.

\bibliographystyle{abbrv}
\bibliography{jl1}
\end{document}